\documentclass{article}
\usepackage[margin=1in]{geometry} 
\usepackage{amsmath,amsthm,amssymb,enumitem, bbm, amsthm}
\usepackage[style=alphabetic]{biblatex}
\usepackage{hyperref}
\hypersetup{
	colorlinks=true,
	linkcolor=blue,
	citecolor=blue,
	filecolor=magenta,      
	urlcolor=cyan,
	linktoc=page
}
\usepackage{graphics}
\usepackage{tikz}
\usepackage[dvipsnames]{xcolor}
\usepackage{ stmaryrd }
\usepackage{xcolor}
\usepackage{placeins}
\usepackage{graphicx}
\usepackage{epstopdf}
\usepackage{float}
\usepackage{caption}
\usepackage{subcaption}
\usepackage{geometry}
\usepackage{parskip}
\renewcommand{\P}{\mathbb{P}}

\newcommand{\hmA}{\hat{\mathcal A}}
  
\newcommand{\R}{\mathbb{R}}  
\newcommand{\Z}{\mathbb{Z}}
\newcommand{\N}{\mathbb{N}}
\newcommand{\Q}{\mathbb{Q}}

\newcommand{\prob}{\mathbb{P}}
\newcommand{\Lag}{\mathcal{L}}
\newcommand{\A}{\mathcal{A}}
\newcommand{\cZ}{\mathcal{Z}}
\newcommand{\xto}{\xrightarrow}

\newcommand{\E}{\mathbb{E}}

\newcommand{\Sheet}{\mathcal{S}}

\newcommand{\x}{\mathbf{x}}
\newcommand{\y}{\mathbf{y}}

\newcommand{\ext}{\textrm{ext}}

\newcommand{\cA}{\mathcal{A}}

\newcommand{\cS}{\mathcal{S}}
\newcommand{\cL}{\mathcal{L}}

\newcommand{\fl}[1]{{\left\lfloor #1 \right\rfloor}}

\newcommand{\sset}{\subset}
\newcommand{\lf}{\left}
\newcommand{\rg}{\right}

\usepackage{MnSymbol}

\DeclareMathOperator*{\sh}{sh}

\newcommand{\bx}{\mathbf{x}}
\newcommand{\by}{\mathbf{y}}

\newcommand{\bp}{\mathbf{p}}
\newcommand{\bq}{\mathbf{q}}

\newcommand{\bu}{\mathbf{u}}

\newcommand{\II}[1]{\llbracket #1 \rrbracket}

\usepackage{amsthm}

\newtheoremstyle{noindentplain}
{3pt}        
{3pt}        
{\itshape}   
{0pt}        
{\bfseries}  
{.}          
{ }          
{}           

\newtheoremstyle{noindentdefinition}
{3pt}        
{3pt}        
{\normalfont}
{0pt}        
{\bfseries}  
{.}          
{ }          
{}           

\theoremstyle{noindentplain}
\newtheorem{thm}{Theorem}[section]
\newtheorem{cor}[thm]{Corollary}
 
\newtheorem{prop}[thm]{Proposition}
\newtheorem{lemma}[thm]{Lemma}

\theoremstyle{noindentdefinition}
\newtheorem{defn}[thm]{Definition}
\newtheorem{remark}[thm]{Remark}

\numberwithin{equation}{section}

\newcommand*{\defeq}{\mathrel{\vcenter{\baselineskip0.5ex \lineskiplimit0pt
			\hbox{\scriptsize.}\hbox{\scriptsize.}}}%
	=}

\newcounter{CommentCounter}

\newif\ifShowComments
\ShowCommentstrue

\title{Down-the-middle isometries for the Airy and KPZ sheets}
\author{Duncan Dauvergne \footnote{University of Toronto, Department of Mathematics, Toronto, Canada. e-mail: duncan.dauvergne@utoronto.ca} \and Fardin Syed \footnote{University of Toronto, Department of Mathematics, Toronto, Canada. e-mail: fardin.syed@mail.utoronto.ca}}

\begin{document}

	\maketitle

	\begin{abstract}
		We find a new coupling between the Airy sheet $\mathcal{S}$ and the Airy line ensemble $\{ \mathcal{A}_n\}_{n \in \N}$, which we call the down-the-middle isometry. In this coupling, the Airy sheet is encoded by last passage values that start at $\mathcal{A}_k(0)$ for some $k \in \N$ and end on the top line $\mathcal{A}_1$. Using this coupling, we give a new construction of the (extended) Airy sheet as a scaling limit of Brownian last passage percolation and study coalescence of geodesics in the directed landscape. This coupling extends to discrete polymers and last passage percolation models, as well as to the KPZ line ensemble and KPZ sheet.
	\end{abstract}

	\begin{figure}[htbp]
		\centering
		\includegraphics[width=0.4\textwidth]{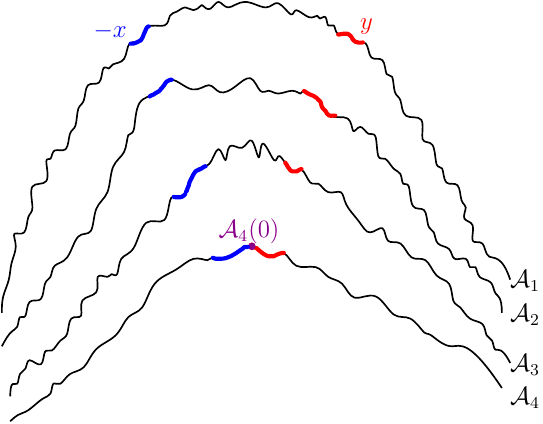} 
		\caption{A sketch of the down-the-middle isometry for the Airy sheet value $\mathcal{S}(x, y)$. The representation includes two last passage values in the Airy line ensemble, one ending at $\mathcal{A}_1(-x)$ and the other ending at $\mathcal{A}_1(y)$, starting at a common point $\mathcal{A}_k(0)$. The maximum value attained over all such $k$ exactly corresponds to $\mathcal{S}(x, y)$ in the coupling described in Theorem \ref{T:shadow-sheet-simple}.} 
		\label{fig:DtM_AirySketch}
	\end{figure} 
	
	\tableofcontents
	
	\section{Introduction}

	The KPZ (Kardar-Parisi-Zhang) universality class is a family of one-dimensional random growth models and two-dimensional random metric and polymer models which exhibit the same universal behavior under rescaling. Our understanding of the KPZ universality class has improved dramatically in the last thirty years, propelled by the discovery of a handful of exactly solvable models. Baik, Deift, and Johansson \cite{baik1999distribution} were the first group to harness this exact solvability, confirming the characteristic KPZ scaling exponents and proving a one-point limit theorem for the longest increasing subsequence in a uniform permutation. Since then, similar limit theorems have been shown for a variety of solvable models, including certain exclusion processes, directed polymers, last passage percolation, and the KPZ equation itself, see \cite{johansson2000shape, PSale, tracy2009asymptotics, amir2011probability} for a few highlights. The full scaling limit for random growth models in the KPZ class, the KPZ fixed point, was constructed by \cite{matetski2016kpz}, and the full scaling limit for random metrics and coupled random growth, the directed landscape, was constructed by \cite{dauvergne2022directed}. 
	
	One of the highlights of the present work is a new construction of the directed landscape. Our proof bypasses much of the technical difficulty of \cite{dauvergne2022directed} and the companion paper \cite{dauvergne2021bulk}, and gives access to new geometric information. The construction is based on a down-the-middle (DtM) isometry for the Robinson-Schensted-Knuth correspondence, which also holds for discrete models of last passage percolation and directed polymer models. In the setting of the KPZ equation, a DtM isometry also allows us to construct the KPZ sheet from the KPZ line ensemble.
	
	\subsection{The DtM representation for the directed landscape}
	
	To make things more concrete, consider a sequence of continuous functions $f = (f_i:\R \to \R, i \in \Z)$. For $x \le y$ and $m \ge n$, a path $\pi$ from $(x, m)$ to $(y, n)$ is a nonincreasing cadlag function from $[x, y]$ to $\II{n, m}$. Letting $t_i = \inf\{t \in [x, y] : \pi(t) \le i\}$, we define the length of $\pi$ with respect to the environment $f$ by 
	$$
	\|\pi\|_f := \sum_{i = n}^m f_i(t_{i-1}) - f_i(t_i),
	$$
	and define the last passage value in $f$ from $(x, m)$ to $(y, n)$ by 
	$$
	f[(x, m) \to (y, n)] = \sup_\pi \|\pi\|_f,
	$$
	where the supremum is over all paths from $(x, m)$ to $(y, n)$.
	Here and throughout we will picture $f$ with matrix coordinates, so that the index $m$ increases as we move down the page, and our maximum is over up-right paths, as in Figure \ref{fig:LP-basic}. Last passage percolation is best thought of as a potentially real-valued metric, where we can only assign distances to points in a certain order. The triangle inequality is reversed since we maximize, rather than minimize, path length.
	
	When the functions $f_i = B_i$ are independent Brownian motions, this model is known as Brownian last passage percolation (LPP). Brownian LPP is an exactly solvable, random directed metric in the KPZ universality class. Its exact solvability comes from a version of the celebrated Robinson-Schensted-Knuth (RSK) correspondence. More precisely, if we consider only the functions $B = (B_1, \dots, B_n)$ on the interval $[0, \infty)$, then the semi-discrete RSK correspondence is a bijection $B \mapsto W$ which maps $B$ to a sequence of non-intersecting Brownian motions $W_i:[0, \infty) \to \R, W_1 \ge W_2 \ge \cdots \ge W_n$. The top line satisfies 
	$$
	W_1(x) = B[(0,n) \to (x, 1)],
	$$
	and the remaining lines have \textit{multi-path} last passage interpretations, see Section \ref{SS:semi-discrete-RSK}. There are tractable exact formulas for non-intersecting Brownian motions, which allow us to take asymptotics for $W$.
	
	\begin{figure}
		\centering
		\includegraphics[width=3in]{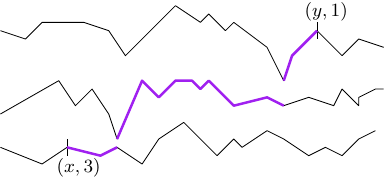}
		\caption{An example of last passage across three functions. The purple path is the last passage path from $(x, 3)$ to $(y, 1)$.}
		\label{fig:LP-basic}
	\end{figure}
	
	While not all exact solvability for KPZ models comes from the RSK correspondence, the first limit theorems all relied on RSK, see e.g. \cite{baik1999distribution, johansson2000shape}, and it remains a powerful tool. A high watermark for this approach was the construction of the Airy line ensemble \cite{PSale, CH} as the full scaling limit of the $W$ side of the RSK map. The parabolic Airy line ensemble is a random sequence of continuous functions $\cA = (\mathcal A_i:\R \to \R)$ ordered so that $\cA_1 > \cA_2 > \cdots$. The process $x \mapsto \cA(x) + x^2$ is stationary.
	
	The directed landscape $\mathcal L:\R^4_\uparrow \to \R$, $\R^4_\uparrow:= \{(x, s; y, t)  \in \R^4 : s < t\}$ was constructed in \cite{dauvergne2022directed} as the full scaling limit of Brownian LPP. The original construction goes through RSK, by means of the RSK isometry \cite{noumi2002, biane2005littelmann, dauvergne2022directed, DNV2}: for all $x \le y \in [0, \infty)$,
	\begin{equation}
		\label{E:BW-iso}
		B[(x,n) \to (y, 1)] = W[(x, n) \to (y, 1)].
	\end{equation}
	Much of the work in \cite{dauvergne2022directed} and the companion paper \cite{dauvergne2021bulk} is devoted to taking a scaling limit of this isometry. The final result of this work is a probabilistic description of the Airy sheet $\cS(\cdot, \cdot) := \cL(\cdot, 0; \cdot, 1)$ in terms of Busemann functions across the Airy line ensemble. Given a uniqueness theorem for the Airy sheet, the full construction of the directed landscape follows from soft and straightforward geometric arguments.
	
	The goal of the present paper is to explore a different \textit{down-the-middle (DtM)} isometry for RSK. In the discrete setting, this isometry has a slightly more complicated presentation than \eqref{E:BW-iso} above. However, it is much more suitable for taking limits. In particular, we obtain a new description of the Airy sheet, together with a new proof of convergence.
	\begin{thm}\label{T:shadow-sheet-simple}
		Let $B = (B_1, \dots, B_n)$ be a family of independent two-sided Brownian motions, and define
		\begin{align*}
			S_n(x, y ) &= n^{-1/3} \left( B[(2 x n^{2/3} , n) \to (n + 2 y n^{2/3} , 1) ] - 2n -  2(y - x) n^{2/3} \right).
		\end{align*}
		Then $S_n$ converges in law, in the uniform-on-compact topology, to a limit $\mathcal{S}$, called the \textbf{Airy sheet}. The law of $\mathcal{S}$ is uniquely characterized by the following two properties:
		\begin{itemize}[nosep]
			\item (Stationarity) For all $c \in \R$, $\mathcal{S}(\cdot + c , \cdot + c) \stackrel{d}{=}  \mathcal{S} (\cdot, \cdot)$.
			\item (Down-the-middle isometry) $\cS$ can be coupled to a parabolic Airy line ensemble $\cA$ so that for all $x, y \ge 0$:
			\begin{align} \label{E:DtM_AirySheet_simple}
				\mathcal{S}(x, y) 
				&=\max_{k \in \N}  \tilde{\mathcal{A}}[(0,k) \to (x, 1)] + \mathcal{A}_k(0) + \mathcal{A}[(0, k) \to (y, 1)] .
			\end{align}
			Here $\tilde \cA$ is a simple reflection of $\cA$, $\tilde{\cA}_i(x) := \cA_i(-x)$, and the maximum above is almost surely attained for all $x, y \ge 0$.
		\end{itemize}
	\end{thm}
	The convergence of $S_n$ was originally proven in \cite[Theorem 1.3]{dauvergne2022directed}, with the aforementioned `Busemann function' description for the limit. A stronger version of this theorem is stated as Theorem \ref{T:shadow-sheet}; the coupling there gives a bijection between the Airy line ensemble and the extended Airy sheet.
	
	The identity \eqref{E:DtM_AirySheet_simple} can be viewed as an LPP value between the points $(-x, 1)$ and $(y, 1)$ in the Airy line ensemble $\mathcal A$. Here paths first go down from $(-x, 1)$ until they hit the vertical axis, and then travel back up to $(y, 1)$. Path length is obtained by adding up all Airy increments along the path together with the value of the turnaround point at $0$, with negative signs included for the increments obtained while the path travels downwards. See Figure \ref{fig:DtM_AirySketch} for an example.

	The astute reader may notice that the coupling in \eqref{E:DtM_AirySheet} is not a standard one, e.g. we do not have $\cS(0, x) = \cA_1(x)$ for all $x \in \R$. Rather, this only holds for $x \ge 0$, and for $x \le 0$ we have $\cS(|x|, 0) = \cA_1(x)$. Higher-index Airy lines have a similar relationship to the extended Airy sheet, see \eqref{E:AiryLE_retrieval}. The existence of these kinds of couplings is a manifestation of shift-invariance for exactly solvable KPZ models, e.g.\ see \cite{borodin2022shift, dauvergne2022hidden, galashin2021symmetries, zhang2023shift, he2026shift}. The specific shift invariance we use here follows from results of \cite{dauvergne2022hidden}. We give a quick self-contained proof in Lemma \ref{L:An-limit}.   
	
	\begin{remark}[Proof ingredients]
		\label{R:ingredients}
		Just as in \cite{dauvergne2022directed}, one of the main ingredients in our proof is uniform-on-compact convergence of the edge of non-intersecting Brownian motions to the Airy line ensemble, shown in \cite{CH}. From this input, tightness of $S_n$ follows by short soft arguments, see \cite[Section 4]{dauvergne2021scaling}. Rather than rehash those arguments we take tightness as a black box. Of course, the main novel piece of the proof is showing a prelimiting version of \eqref{E:DtM_AirySheet_simple} between Brownian LPP and non-intersecting Brownian motions. Given this input, to check that this identity passes to the limit we just need to show that the maximizing index is tight. Here we give a two-page proof by applying recent strong results of \cite{wu2024applications} together with straightforward estimates on GUE eigenvalues. However, even without the input from \cite{wu2024applications}, the tightness problem is quite manageable, and we expect that many methods could resolve it. For example, in Appendix \ref{sec:shadow_sheet_geomproof}, we present a different geometric approach which shows that this problem is equivalent to proving a lower bound on the growth rate of a geodesic $k$-melon. The latter problem was understood in \cite{basu_interlacing}. 
	\end{remark}

		\subsection{The DtM representation for other models}
		
		As with the usual RSK isometry \cite{noumi2002, dauvergne2022hidden, corwin2020invariance}, the down-the-middle isometry is a property of the (lattice) geometric RSK correspondence. It passes to all degenerations, including the standard RSK correspondence, and the KPZ equation/sheet, where it simplifies. We state a special case of the down-the-middle isometry here, and refer the reader to Theorem \ref{thm:RSK_Representation_discrete} for the statement in full generality.  
		
		Let $M: \II{1, n}  \times  \II{1, m}  \to (0, \infty), u \mapsto M_u$ be an $m \times n$ matrix.  For $p = (x, s), q = (y, t) \in \Z^2$ with $x \le y$ and $s \ge t$, define the \emph{polymer partition function}
		\begin{align*}
			Z_M(p; q) = \sum_{\pi : p \to q} \prod_{u \in \pi} M_u,
		\end{align*}
		where the sum is over all up-right lattice paths $\pi$ starting at $p$ and ending at $q$. Note that as in the semi-discrete setting, we picture paths in matrix coordinates.
		
		For vectors $\bp, \bq \in (\Z^2)^k$ with the same ordering constraints on each pair $(p_i, q_i)$, we similarly define the $k$-path partition function
		$$
		Z_M(\bp; \bq) = \sum_{\pi : \bp \to \bq} \prod_{u \in \pi_1 \cup \dots \cup \pi_k} M_u,
		$$ 
		where now the sum is over all $k$-tuples of disjoint up-right lattice paths $\pi = (\pi_1, \dots, \pi_k)$, where $\pi_i$ starts at $p_i$ and ends at $q_i$.
		
		Next, define $\Gamma_{n, m} = \{ (x, y) \in  \II{1, n} \times \II{1, m} :x \ge y\}$. Given the matrix $M$, there is a unique function $\Phi(M):\Gamma_{n, m} \to (0, \infty)$ such that for all $j \in \II{1, n}$ and $k \in \II{1, j \wedge m}$ we have the equality
		\begin{align*}
			Z_M(((1, m), \dots, (k, m)); ((j - k + 1, 1), \dots (j, 1))) 
			=   Z_{\Phi(M)}(((1, 1), \dots, (k , k)); ((j - k + 1, 1), \dots (j, 1))),
		\end{align*}
		See Section \ref{SS:geometric-RSK} for a more explicit construction of $\Phi$.
		The geometric RSK correspondence and its variants can be constructed from this map $\Phi$, together with rotations and reflections. For the purposes of our next theorem, we let $P^s:\Gamma_{n, m} \to (0, \infty)$ be the output of first applying a 180-degree rotation to $M$ and then applying $\Phi$, and let $Q:\Gamma_{m, n} \to (0, \infty)$ be the output of first reflecting $M$ across the diagonal $x = m + 1 - y$, and then applying $\Phi$. We also set
		$$
		\lambda_i = Z_{P^s}(i, i; n, i) = Z_Q(i, i; m, i).
		$$
		The map $M \mapsto (P^s, Q)$ is invertible. This is not quite the standard geometric RSK correspondence, but rather an application of geometric RSK followed by a Sch\"utzenberger involution on $P$. This is also known as the geometric Burge correspondence, see \cite{bisi2020geometric} for background and probabilistic properties. The standard RSK/Burge correspondences are recovered from this description if we replace every instance of $+$ and $\times$ above with $\max$ and $+$, respectively. The functions $P^s$ and $Q$ are then essentially two Young tableaux, and $(\lambda_1, \dots, \lambda_{n \wedge m})$ is their shared shape.
		
		Finally, for $x \le y$ and $k \ge 1$ we define the \textit{shadow partition function}
		$$
		\check Z_M(x, 1; y, k) = \sum_{\tau:(x, 1) \to (y, k)} \prod_{u \in C(\tau)} M_u,
		$$
		where the sum is over all lattice paths $\tau = (\tau_0, \dots, \tau_{y-x})$ from $(x, 1)$ to $(y, k)$ with $\tau_{i+1} - \tau_i \in \{(1, 0), (1, 1)\}$ for all $i$, and $C(\tau) = \{\tau_i :  i \ge 1, \tau_i - \tau_{i-1} = (1, 0)\}$.
		
		\begin{thm}\label{T:DTM_Discrete_Informal}
			In the above setup, for all $x \in \II{1, n}$ and $y \in \II{1, m}$ we have the identity
			\begin{align*}
				Z_M(x, m; n, y) &= \sum_{k = 1}^{x \wedge y} \check Z_{(P^s)^{-1}}(n - x + 1, 1; n, k) \cdot \lambda_k \cdot \check Z_{Q^{-1}}(m - y + 1, 1; m, k). 
			\end{align*}
			Here $(P^s)^{-1}, Q^{-1}$ are obtained from $P^s, Q$ by applying the inverse map $x \mapsto 1/x$ entry-wise.
		\end{thm}
		An identical statement holds for the usual RSK correspondence if we replace the $(+, \times)$-algebra with the $(\max, +)$-algebra by taking a zero-temperature limit.

		After proving Theorem \ref{T:DTM_Discrete_Informal} (or rather its generalization in Theorem \ref{thm:RSK_Representation_discrete}), we found a weaker version of the theorem in \cite[Equation 3.93]{noumi2002}. However our proof method, which is based on the Desnanot-Jacobi determinant identity (Lemma \ref{lemma:desnanot_jacobi}), differs significantly and generalizes better to continuum settings. In particular, our argument extends to give a direct proof that a DtM Representation holds between the KPZ sheet and the KPZ line ensemble. Again, a more general version of this theorem is given in the body of the paper as Theorem \ref{thm:KPZ_Case}. All objects in the forthcoming theorem are introduced precisely in Section \ref{sec:KPZ_Sheet}.
		
		\begin{thm} \label{thm:KPZ_Case_simply}
			Let $\{ \mathcal{H}^{(t)}_n \}_{n \in \N}$ be the KPZ line ensemble (normalized with the homogeneous Gibbs property), and let $\mathcal{H}^{(t)}(x, y)$ be the KPZ sheet (see Section \ref{sec:KPZ_Sheet} for precise definitions). Let
			\begin{align*}
				&\mathcal{Z}_n^{(t)} =  e^{\mathcal{H}_n^{(t)}}  , \qquad M^{(t)} (x, y) = e^{\mathcal{H}^{(t)} (x, y)}
			\end{align*}
			denote the corresponding stochastic heat (SH) line ensemble and SH sheet. Then there exists a coupling between the SH line ensemble and the SH sheet such that, for every $x, y \geq 0$,
			\begin{align}
				\label{E:KPZ-simple-intro}
				M^{(t)}(x, y) &=  \sum_{k \in \N} \tilde{\mathcal{Z}}^{(t)} \{ (0, k) \to (x, 1)\} \cdot \mathcal{Z}^{(t)}_{k}(0)   \cdot  \mathcal{Z}^{(t)} \{ (0, k) \to (y, 1)\},
			\end{align}
			where $\tilde{\mathcal{Z}}^{(t)}_n(x) = \mathcal{Z}_{n}^{(t)}(-x)$, and $\mathcal{Z}^{(t)} \{(\cdots) \to (\cdots)\}$ denotes a partition function across the SH line ensemble.
		\end{thm}

		\subsection{Geodesic disjointness and melon widths in the directed landscape}
		
		Unlike in the original construction of the Airy sheet in \cite{dauvergne2022directed}, in the DtM representation, path lengths only depend on a compact region of the Airy line ensemble. This is useful when studying properties of geodesics in $\mathcal L$, such as coalescence. Coalescence of geodesics and polymers is an important feature of the KPZ geometry, and has been heavily studied, e.g. see \cite{BSS, ZhangOE, Hammond_2019}. It is also closely connected to KPZ scaling exponents, e.g. see \cite{gu2026integration}.

		In Section \ref{SS:turnaround-geometry} we show that coalescence of geodesics in the landscape is equivalent (almost surely) to coalescence of the corresponding paths in the Airy line ensemble. Hence, coalescence probabilities in the directed landscape can be reduced to events that only depend on a finite region of the Airy line ensemble, whose probabilities can be understood using the Brownian Gibbs Property. To demonstrate this, we identify geodesic coalescence exponents for the directed landscape.

		\begin{thm} \label{thm:Hammond_conj}
			For $\epsilon > 0$, $k \in \N$, let $M(\epsilon)$ be the maximum number of disjoint geodesics in the directed landscape that start in the set $[0, \epsilon] \times \{0\}$ and end in the set $[0, \epsilon] \times \{1\}$. Then for all $k \in \N$, there exists a constant $c_k > 0$ such that for all $\epsilon \in (0, 1)$ we have the lower bound
			$$
			\mathbb P(M(\epsilon) \ge k) \ge c_k \epsilon^{\frac{k^2 - 1}{2}}.
			$$
		\end{thm}
		This verifies a conjecture of Hammond \cite{Hammond_2019}, who proved the matching upper bound in \cite[Theorem 1.1]{Hammond_2019} using Gibbs resampling methods developed in \cite{CH, HamBr}. Up to $\epsilon^{o(1)}$ factors, the $k = 2$ case of the above theorem was proven in \cite{bates2019hausdorff}, and the $k = 3$ case can be extracted from \cite{dauvergne202327}. Those proofs are based on the existence of $k$-tuples of disjoint geodesics with the same endpoints. Since no such $k$-tuples exist for $k \ge 4$, the proofs do not extend to such $k$. Note that Hammond's conjecture is for Brownian LPP, rather than the directed landscape, but the theorem above implies that the same bound holds asymptotically for Brownian LPP since the event $\{M(\epsilon) \ge k\}$ is open with respect to a topology where Brownian LPP converges to the directed landscape, see \cite[Section 8]{dauvergne2021scaling}.

	The DtM representation also gives access to quick estimates on the width of geodesic watermelons in the directed landscape. Here recall that a $k$-geodesic watermelon in $\cL$ from $(0,0)$ to $(0, 1)$ is a collection of paths $\pi_1 \le \cdots \le \pi_k$ from $(0,0)$ to $(0, 1)$ which are strictly ordered on $(0, 1)$, and maximize the sum of their path lengths. Geodesic watermelons exist and are almost surely unique, see \cite{dauvergne2021disjoint}. The paper \cite{basu_interlacing} studied the widths of geodesic watermelons in solvable models of LPP, using soft geometric methods. We give a complementary approach to this problem using the DtM representation. 
	
	\begin{thm}
		\label{T:melon_width}
		Let $\pi = (\pi_1 , \ldots, \pi_k)$ be a $k$-geodesic watermelon in the directed landscape from $(0, 0)$ to $(0, 1)$. Then there exist constants $0 < c < C$ and $d > 0$ so that with probability $1 - e^{-dk}$ as $k \to \infty$, we have the bounds
		\begin{align*}
			c k^{1/3} \leq \max_{0 < t < 1} \pi_k(t) & \leq C k^{1/3}.
		\end{align*}
	\end{thm}	
	
	By combining the results of geodesic watermelon widths from \cite{basu_interlacing} with the convergence of geodesic watermelons from \cite{dauvergne2021disjoint}, we could alternately obtain Theorem \ref{T:melon_width} with $\pi_k(t)$ replaced by $|\pi_k(t)|$. Note that the lower bound in \cite{basu_interlacing} also holds with a stronger $1 - e^{-d k^2}$ probability.

	\subsection{Related work and future directions}
	
	We give a brief discussion of closely related work and potential future applications of the down-the-middle isometry. For more background on the KPZ, we refer the reader to 
	\cite{quastel2011introduction, corwin2012kardar, romik2015surprising, borodin2016lectures, takeuchi2018appetizer, zygouras2022some, ganguly2021random} and references therein.
	
	Since \cite{dauvergne2022directed}, many other models have been shown to converge to the directed landscape: integrable LPP models \cite{dauvergne2021scaling}, the KPZ equation \cite{wu2023kpz}, the stochastic six vertex model and colored ASEP \cite{ACH}, the log-gamma polymer \cite{zhang2025convergence}, and web distances and various couplings of TASEP \cite{DZ24}. We expect our method can be used to give new convergence proofs for last passage and polymer models where there is an underlying version of RSK, though we have not pursued this here. Note that \cite{virag2025actions} recently gave a unified framework for understanding convergence of last passage and polymer models using the original Busemann description for the Airy sheet, which covers the same suite of models. From the above list, it would also be interesting to understand whether there is a DtM-type theorem related to the Yang-Baxter equation for the coloured stochastic six-vertex model, with the aim of giving an alternate proof of the main results of \cite{ACH}. Indeed, the proof in \cite{ACH} goes through a line ensemble representation (see also \cite{aggarwal2024colored}), and the coloured stochastic six-vertex model satisfies shift invariance \cite{borodin2022shift, galashin2021symmetries}, which are the two key ingredients in the DtM approach.

	Another natural direction of inquiry is to understand the Busemann function limit of down-the-middle isometries. This should exist in all LPP and polymer models, but for concreteness we focus on the directed landscape. The function
	$$
	(x, \theta) \mapsto \cL(x, 0; t \theta, t) - \cL(0, 0; t \theta, t)
	$$
	has a two-parameter scaling limit known as the stationary horizon, e.g.\ see \cite{BuS, BSS24, seppalainen2023global}, and \cite{FPAMJB, FSj} for earlier results on prelimiting models. Alternately, this scaling limit can be viewed by simply rescaling and centering the Airy sheet. The Airy line ensemble-Airy sheet representation from \cite{dauvergne2022directed} has an interesting scaling limit in terms of a family of independent two-sided Brownian motions, identified and studied in \cite{dauvergne2024directed}. The same scaling of the DtM isometry should produce a different representation, which should involve a family of one-sided Brownian motions paired with a process of a different character. 
	
	The directed landscape has also been recently constructed in half-space \cite{dauvergne2026directed}. The scaling limit is parametrized by $\rho \in \R \cup \{-\infty\}$ which controls the boundary strength. In this setting, the original full-space program from \cite{dauvergne2022directed} appears to be difficult to implement (though see \cite{dimitrov2025half, dimitrov2026pinned, das2026pinning} for progress in this direction on constructing half-space line ensembles), and a different technique was used to construct the scaling limit from the half-space KPZ fixed point \cite{Zhang}. In forthcoming work of the second author \cite{Syed2026+}, a version of the DtM Representation is used to construct the half-space directed landscape at the scale-invariant critical point $\rho = 0$. The method therein also has relevance for polymer models and the KPZ equation, which lie beyond the scope of \cite{dauvergne2026directed}.
	
		\subsection{Outline of the paper and a word on the proofs}
	Section \ref{sec:discrete} addresses discrete models. After some preliminary setup, the proof of Theorem \ref{T:DTM_Discrete_Informal} (and its full generalization in Theorem \ref{thm:RSK_Representation_discrete}) is quite short. The key ingredient is Lemma \ref{lemma:discrete_nonintpaths}, which gives a series expansion of a partition function value $Z_M(p; q)$ in terms of certain multi-path partition functions. The lemma is proven through iterated applications of the Desnanot-Jacobi identity. We prove the LPP analogue of Theorem \ref{T:DTM_Discrete_Informal} by exchanging the $(+, \times)$-algebra for the $(\max, +)$-algebra in Corollary \ref{cor:RSK_discrete_LPP}. A final short section (Section \ref{SS:part-tableaux}) sets up a relative of the DtM isometry which will ultimately be more suited for taking a scaling limit to semi-discrete LPP and then to the directed landscape.
	
	Section \ref{sec:DTM-KPZ} proves Theorem \ref{thm:KPZ_Case_simply}, and its generalization, Theorem \ref{thm:KPZ_Case}. We first show that if we truncate the series on the right-hand side of \eqref{E:KPZ-simple-intro} at a level $n \in \N$, then the difference between this truncation and $M^{t}(x, y)$ is given by an explicit remainder term $R^{(n)}(x, y)$. This is done in Lemmas \ref{lemma:KPZSheet_rep1} and \ref{lemma:LPP_values_KPZLE}. While we expect that these lemmas could potentially be proven by taking a scaling limit from a discrete model as in \cite{corwin2017intermediate, Nica_2021}, some of the necessary multi-point partition functions have not been analyzed previously and so this would be a technical undertaking. Rather, we prove these results directly in the continuum, using Karlin-McGregor formulas for the multiplicative stochastic heat equation as a starting point. The proof of Theorem \ref{thm:KPZ_Case} then follows by showing that $R^{(n)}(x, y)$ converges to $0$ with $n$. This turns out to be a fairly difficult problem, and requires non-trivial estimates on the KPZ line ensemble which we import from our prior work \cite{Syed_2025}.
	
	Section \ref{sec:Airy_case} proves Theorem \ref{T:shadow-sheet-simple} and its generalization Theorem \ref{T:shadow-sheet}. As mentioned above, the first step is proving a version of the DtM isometry for semi-discrete LPP, which follows by taking a scaling limit of the results from Section \ref{SS:part-tableaux}. The remaining elements of the proof are discussed in Remark \ref{R:ingredients}. Section \ref{S:applications} proves all applications of our main result to the directed landscape.
	
	Finally, we have included a short appendix, which contains an approach to Theorem \ref{T:shadow-sheet-simple} based on path-switching arguments. We were not able to extend this approach to give a complete proof of the stronger Theorem \ref{T:shadow-sheet}, but have nonetheless included it as we feel it is enlightening.
	
	\textbf{AI usage disclosure.} \qquad AI tools were used for literature review, proofreading, and to check for mathematical errors. They were not used to develop any mathematical ideas or proofs, or in the writing of the manuscript. All suggested changes were reviewed by the authors, who take full responsibility for the final content of the paper.
	
	\section{DtM Representation for Discrete Polymer and Percolation Models} \label{sec:discrete}
	
	In this section, we first prove a version of our RSK formula for discrete positive temperature models/lattices. We then obtain the analogous 0-temperature result by switching the algebra from $(+, \times)$ to $(\max, +)$. We also take a limit to LPP across lines, which we will use later to construct the DtM representation for the Airy sheet.  The main result of this section can be seen as a generalization of Equation 3.94 of \cite{noumi2002}, although our proofs differ. As we will see in Section \ref{sec:KPZ_Sheet}, the main idea of our proof extends to the continuum case of the KPZ sheet.

	\subsection{Discrete Last Passage Percolation and Partition Functions}
	\label{S:disclpp}

	\begin{defn} \label{defn:LPP_Partn}
		Let $M:\Z^2 \to \R, x \mapsto M_x$ be a lattice of weights. As before, we will picture $M$ and $\Z^2$ with the vertical coordinates $\II{1, n}$ increasing as we go down the page. Let $\Z^4_\uparrow = \{u= (p, q) \in \Z^4 : p_1 \le q_1, p_2 \ge q_2 \}$. For any $u = (p, q) \in \Z^4_\uparrow$, we define
		\begin{equation}
			\label{E:ZMLM}
			L_M(u) = L_M(p, q) = \max_{\pi: p \to q} \sum_{x \in \pi} M_x, \qquad Z_M(u) = Z_M(p, q) = \sum_{\pi: p \to q} \prod_{x \in \pi} M_x.
		\end{equation}
		Here, the maximum and the sum in both settings are over all \textbf{up-right paths} $\pi$ from $p$ to $q$, i.e. lattice paths $\pi = (\pi_1, \dots, \pi_\ell)$ with $\pi_1 = p$ and $\pi_\ell = q$ where $\pi_i - \pi_{i-1} \in \{(1, 0), (0, -1)\}$ for all $i \in \II{2, \ell}$. The quantity $L_M$ is called the \textbf{last passage} value in $M$ from $p$ to $q$, and $Z_M$ is the \textbf{(polymer) partition function} in $M$ from $p$ to $q$. We extend these definitions to multi-point last passage values and multi-point partition functions.
		
		Let $\bu = (\bp ; \bq)$ for vectors $\bp, \bq \in (\Z^2)^k$ and set $u_i = (p_i, q_i)$ for $i = 1, \dots, k$. We assume that there is at least one $k$-tuple of \textit{disjoint} up-right lattice paths $\pi = (\pi_1, \dots, \pi_k)$ where each $\pi_i$ is an up-right path from $p_i$ to $q_i$. We call such $\pi$ a \textbf{disjoint $k$-tuple} from $\bp$ to $\bq$. Then define
		\begin{equation}
			L_M(\bu) = \max_{\pi: \bp \to \bq} \sum_{x \in \cup \pi} M_x, \qquad Z_M(\bu) = \sum_{\pi: \bp \to \bq} \prod_{x \in \cup \pi} M_x,
		\end{equation}
		where the maximum/sum is over all disjoint $k$-tuples from $\bp$ to $\bq$, and $\cup \pi \sset \Z^2$ is the set of all vertices in one of the constituent paths of $\pi$.
		
		Certain collections of partition functions and last passage values where points start and end in clusters will be especially relevant to us. Because of this, for $p \in \Z^2$, $k \in \N$, we define the clusters $p^{ k \uparrow}, p^{k \downarrow}, p^{k \rightarrow}, p^{k \leftarrow} \in (\Z^2)^k$ by
		\begin{align*}
			p^{k \uparrow} &\defeq (p - (0, k-1), \dots, p - (0, 1), p), \qquad p^{k \downarrow} \defeq (p, p + (0, 1), \dots, p + (0, k-1)),\\
			p^{k \rightarrow} &\defeq (p, p + (1, 0), \dots, p + (k-1, 0)), \qquad p^{k \leftarrow} \defeq (p - (k-1, 0), \dots, p - (1, 0), p).
		\end{align*}
		In addition, for a vector $u = (p ; q) \in \Z^{4}_{\uparrow}$, we define
		\begin{align*}
			u^k \defeq (p^{k \rightarrow} ; q^{k \leftarrow})
		\end{align*}
		For $o = (1, n)$ and $\hat{o} = (m, 1)$, we often suppress the arrow notation in the superscript as the direction does not affect the partition function/LPP values. For instance, for clusters $\mathbf{p}$, $Z_M(\mathbf{p}, \hat{o}^{k \downarrow}) =Z_M(\mathbf{p}, \hat{o}^{k \leftarrow}) $, $Z_M(o^{k \uparrow}, \mathbf{p} ) =Z_M( o^{k \rightarrow}, \mathbf{p}) $, and similarly for the LPP values. 
	\end{defn}
	
	\subsection{Partition functions vs. Last passage values}
	
	Algebraically, studying the field of partition function values is similar to studying the field of multi-point last passage values, where we exchange the usual $(+, \times)$-algebra for the $(\max, +)$-algebra. Hence, we can often prove identities for last passage values by first proving identities for polymer partition functions and then simply exchanging the algebra. The following lemma allows us to do this.
	
	\begin{lemma}(\cite[Proposition 1.9]{noumi2002})\label{lemma:temp_change}
		We say a rational function $f: \R^n \to \R$ is a subtraction-free rational function if there exist two polynomials with positive coefficients $a(x)$, $b(x)$ so that
		\begin{align*}
			f(x) &= \frac{a(x)}{b(x)}. 
		\end{align*}
		For any subtraction-free rational function $f$, define the piece-wise linear function $M(f): \R^n \to \R$ via the following limit
		\begin{align*}
			(Mf)(x_1, \ldots, x_n) &= \lim_{\beta \to \infty} \frac{1}{\beta} \log f(e^{\beta x_1} , \ldots, e^{\beta x_n}). 
		\end{align*}
		Then for any two subtraction free rational functions $f, g$
		\begin{align*}
			M(fg) &= M(f) + M(g) , \qquad M(f / g) = M(f) - M(g) , \qquad M(f + g) = \max \left( M(f), \, M(g)\right).
		\end{align*}
		That is to say, $(Mf)$ can be equivalently retrieved from $f$ by formally switching the operations of $(+, \times)$ to $(\max, +)$. 
	\end{lemma}
	
	The advantage of working with partition functions is that the Lindström–Gessel–Viennot (LGV) theorem allows us to express certain multi-point partition functions as determinants of single-point partition functions. An analogous formula does not hold for LPP since the operation $\max$ is not invertible. We have stated this lemma only in the restricted setting in which we use it.
	
	\begin{lemma}[{\bf Lindstrom-Gessel-Viennot}] \label{LGV_thm}
		Let $G = (V, E)$ be a directed, finite, acyclic graph with edge weights $(w_e, e \in E)$. For any directed path $\pi = (e_1 , \ldots , e_m)$, let 
		\begin{align*}
			w(\pi) &= \prod_{e_i \in \pi} w_{e_i}.
		\end{align*}
		Furthermore, for any two vertices $a, b \in V$, define the partition function $Z_{G}(a, b)$ by
		\begin{align*}
			Z_{G}(a, b) &= \sum_{\pi: a \to b} w(\pi).
		\end{align*}
		Consider base vertices $A = \{a_1 , \ldots, a_n\}$ and destination vertices $B = \{ b_1 , \ldots , b_n\}$. Suppose additionally that if $(\pi_1, \dots, \pi_n)$ is any disjoint $n$-tuple of paths such that $\pi_i$ starts at $a_i$ and ends in $B$ for all $i$, then $\pi_i$ ends at $b_i$ for all $i$. 
		Then 
		\begin{align*}
			\det_{\substack{1 \le i ,j \le n}} Z_G(a_i, b_j) &= \sum_{(\pi_1 , \ldots , \pi_n) : A \to B} \prod_{i = 1}^n w(\pi_i),
		\end{align*}
		where the sum is over all disjoint $n$-tuples from $A$ to $B$.
	\end{lemma}
	Note that in the setting of Section \ref{S:disclpp}, our partition functions are defined using vertex weights, rather than edge weights. It is not difficult to check that the LGV lemma also accommodates vertex weights, or that we can equivalently set up our framework using edge weights. 
	
	One of the advantages of working with determinants is access to certain identities. For our purposes, we require the Desnanot-Jacobi determinant identity.
	
	\begin{lemma}[Desnanot-Jacobi Identity]
		\label{lemma:desnanot_jacobi}
		Let $M$ be a $k \times k$ matrix. For $\ell < k$,  $i_1 < \cdots < i_{\ell}  \in \{ 1, \ldots k\}$  and $j_1 < \ldots < j_{\ell} \in \{ 1, \ldots, k\}$, let $M^{i_1 , \ldots, i_{\ell}}_{j_1, \ldots, j_{\ell}}$ be the $(k - \ell) \times (k - \ell)$ matrix obtained by removing rows $i_1, \ldots, i_{\ell}$ and columns $j_1, \ldots, j_{\ell}$. Then,
		\begin{align*}
			\det(M) \det(M^{i_1, i_2}_{j_1, j_2}) = \det(M^{i_1}_{j_1}) \det(M^{i_2}_{j_2}) - \det(M^{i_1}_{j_2}) \det(M^{i_2}_{j_1})
		\end{align*} 
	\end{lemma}
	
	\subsection{RSK for polymer models} \label{SS:geometric-RSK}
	
	We focus first on discrete polymer models, which correspond to the choice of $(+, \times)$-algebra. We prove analogous results for discrete LPP models in the following section by switching the algebra to $(\max, +)$. 
	
	\begin{defn}[The geometric RSK correspondence]
		For a set $S \subset \Z^2$, let $F(S, T)$ denote the set of functions $M: S \to T$. 
		Also, for $m, n \in \N$, define the set
		\begin{align*}
			\Gamma_{m, n} = \{(x, y) \in \II{1, m} \times \II{1, n} : x \ge y\},
		\end{align*}
		and, setting $u_i = (1, n; i, 1)$, define the map
		$\Phi: F(\II{1, m} \times \II{1, n}, (0, \infty)) \to F(\Gamma_{m, n}, (0, \infty))$ by
		\begin{align*}
			\Phi(M)(i, j) &= 
			\begin{cases}
				\dfrac{Z_M(u_i^{j})}{Z_M(u_{i}^{j-1})}, & j = i, \\[6pt]
				\dfrac{Z_M(u_i^{j}) \, Z_M(u_{i-1}^{j-1})}{Z_M(u_{i-1}^{j}) \, Z_M(u_{i}^{j-1})}, & j < i,
			\end{cases}
		\end{align*}
		where we use the convention that $Z_M(u_i^0) = 1$. The geometric RSK correspondence is the map $M \mapsto (P, Q)$, where
		\begin{align} \label{defn:tropical_tableaux}
			P &= \Phi(M), \qquad Q = \Phi(M^F).
		\end{align}
		We refer to $P$ and $Q$ as \textbf{tableaux}, as they correspond to the Young tableaux that appear in the usual RSK correspondence.
		Here $M^F$ is the array obtained by flipping $M$ across its anti-diagonal, so that  
		$M^F(i, j) = M_{m - j + 1 , n - i + 1}$. 
		Finally, we define `patterns' $\mu = (\mu_{ij})_{1 \leq i \leq m,\,1 \leq j \leq i\wedge n}$ and  
		$\nu = (\nu_{ij})_{1 \leq i \leq n,\,1 \leq j \leq i\wedge m}$ corresponding to $P$ and $Q$, respectively, as
		\begin{align}
			\mu_{ij} &= \prod_{k = j}^i P_{k j}, \qquad \nu_{ij} = \prod_{k = j}^i Q_{k j}.
		\end{align}
		In the zero-temperature limit, $\mu$ and $\nu$ become Gelfand-Tsetlin patterns.
		Since $Z_{M} (\cdot; \hat{o}^{k \leftarrow }) = Z_{M}(\cdot; \hat{o}^{k \downarrow})$, it follows that the bottom rows of the two patterns are equal. We call this bottom row the shape, $\sh (P) = \sh(Q) = (\lambda_1, \ldots, \lambda_{n\wedge m})$ which is given by
		\begin{align}
			\lambda_j &= \prod_{k = j}^m P_{k j} = \prod_{k = j}^n Q_{k j}.
		\end{align}
		
	\end{defn}
	
	Remarkably, the tableaux $P$ and $Q$ share many of the same partition functions as the environment $M$. This was proven in the form below in \cite{dauvergne2022hidden, corwin2020invariance}, though versions of the theorem are essentially contained in \cite{noumi2002}.
	
	\begin{thm}[Theorem 3.3, \cite{dauvergne2022hidden}]
		\label{T:encoding-lpp}
		Fix $m, n \in \N$ and let $M \in F(\II{1, m} \times \II{1, n}, (0, \infty))$. 
		For every point of the form $p = (i, n)$ with $i \in \II{1, m}$, let $\bar{p} = p$ if $i \ge n$ and $\bar{p} = (i, i)$ if $i < n$. When $\bp = (p_1, \dots, p_k)$  we write ${\bar \bp} = (\bar p_1, \dots, \bar p_k)$. 
		Then for any $k$-tuples $(\bp, \bq)$ with each $p_i \in \II{1, m} \times \{n\}$ and each $q_i \in \II{1, m} \times \{1\}$, we have the isometry
		\begin{align*}
			Z_M(\bp, \bq) &= Z_P(\bar{\bp}, \bq).
		\end{align*}
		Similarly, for any $k$-tuples $(\bp, \bq)$ with each $p_i \in \{1\} \times \II{1, n}$ and each $q_i \in \{m\} \times \II{1, n}$,
		\begin{align*}
			Z_{M}(\bp, \bq) &= Z_Q(\bar{\bp}^F, \bq^F).
		\end{align*}
	\end{thm}
	
	Another important transform on tableaux is given by the (geometric) Schützenberger involution. 
	The definition we provide of the Schützenberger involution in terms of non-intersecting paths is not the standard definition in terms of jeu-de-taquin, but it is equivalent. This is shown for geometric RSK in Theorems 2.8 and 2.9 of \cite{noumi2002}.

	\begin{figure}[htbp] 
		\centering
		\includegraphics[width=0.4\textwidth]{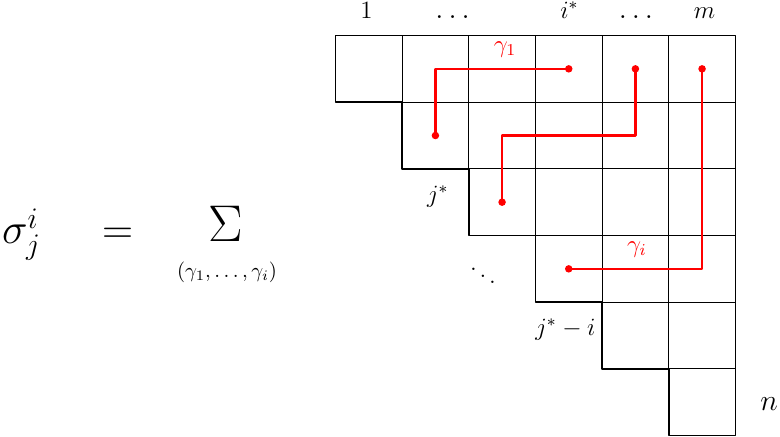} 
		\caption{A graphical illustration of $\sigma^i_j$ from Definition \ref{defn:Schutzenberger}.}
		\label{fig:Schutzenberger}
	\end{figure} 
	
	First, we define the $180^{\circ}$-rotation of a vector. For $p = (i, j) \in \II{1, m} \times \II{1, n}$, let $p^* = (m - i + 1, n - j + 1)$. For a tuple of vectors $\mathbf{p} = (p_1 , \ldots , p_k)$, let $\mathbf{p}^* = (p_k^{*} , \ldots , p_1^*)$. Finally, for a vector $u = (\mathbf{p}, \mathbf{q})$ where $\mathbf{p}, \mathbf{q} \in (\II{1, m} \times \II{1, n})^k$, let $u^* = (\mathbf{q}^*, \mathbf{p}^*)$. Finally, for a matrix $M \in F (\II{1, m} \times \II{1, n} , (0, \infty))$ define its $180^{\circ}$ rotation $RM$ as:
	\begin{align*}
		RM_{i, j} \defeq M_{i^*,j^*} = M_{m - i + 1, n - j + 1}.
	\end{align*}

	\begin{defn} \label{defn:Schutzenberger}
		Fix a tableau $P = (p^i_j)_{1 \leq j \leq m,\,1 \leq i \leq j\wedge n}$. 
		For $i \in \II{1, n}$, $j \in \II{1, m}$ and $i \leq j$, let 
		\begin{align*}
			\sigma^{i}_j &= Z_P \left( \overline{(m - j + 1, n)^{i \rightarrow}} ;  (m, 1)^{i \leftarrow}  \right) .
		\end{align*}
		
		Set $\sigma_j^0=1$ for $0\leq j\leq m$. That is, $\sigma^{i}_j$ is obtained from the partition function illustrated in Figure \ref{fig:Schutzenberger}. 
		Then the geometric Schützenberger involution is the tableau $P^s = (\tilde{p}^i_j)_{1 \leq j \leq m,\,1 \leq i \leq j\wedge n}$, where $\tilde{p}^i_j$ is given by
		\begin{align*}
			\tilde{p}^i_i &= \frac{\sigma^{i}_i}{\sigma^{i-1}_i}, \qquad  
			\tilde{p}^i_j = \frac{\sigma^i_j \sigma^{i-1}_{j-1}}{\sigma^{i}_{j-1} \sigma^{i-1}_j}\quad (i<j). 
		\end{align*}
		Furthermore, this tableau satisfies $\sh(P) = \sh(P^s)$. 
	\end{defn}
	It turns out that if one first rotates the matrix by $180^{\circ}$ and then applies the RSK correspondence, the output is given by Schützenberger-involuted tableaux. 
	That is, if a matrix $M \in F(\II{1, m} \times \II{1, n}, (0, \infty))$ corresponds to a pair of tableaux $(P, Q)$, then by Theorem \ref{T:encoding-lpp}, we have
	\begin{align*}
		\sigma^{i}_j &=  Z_{M} \left( (m - j + 1, n)^{i \rightarrow} ;  (m, 1)^{i \leftarrow}  \right) = Z_{RM} ((1, n)^{i \rightarrow} ; (j, 1)^{i \leftarrow} ),
	\end{align*}
	and hence
	\begin{align*}
		M \mapsto (P, Q) \iff RM \mapsto (P^s, Q^s).
	\end{align*}
	The fact that $\sh(P) = \sh(P^s)$ follows from the observation that $(o^k, \hat{o}^k)^* = (o^k, \hat{o}^k)$.
	
	To state the main theorem of this section, we will first need to define the concept of a \emph{shadow path}.
	
	\begin{defn}
		Consider $p = (m, n)$ and $q = (m', n')$ with $m \le m'$ and $n \le n'$, and assume that $(m' - m) \ge (n' - n)$. 
		A \emph{shadow path} from $p$ to $q$ is a sequence of boxes $\tau = (\tau_0, \dots, \tau_{m' - m})$ where
		\begin{itemize}
			\item $\tau_0 = p$, $\tau_{m' - m} = q$, and for every $i \in \II{1, m' - m}$ we have $\tau_{i} - \tau_{i-1} \in \{(1, 0), (1, 1)\}$.
		\end{itemize}
		For a shadow path $\tau = (\tau_0, \dots, \tau_{m' - m})$, define its \textbf{core} by
		\begin{align*}
			C(\tau) = \{\tau_i : i \in \II{1, m' - m}, \tau_{i} - \tau_{i-1} = (1, 0)\}.
		\end{align*}
		Note that the core always omits the starting point $p$. 
		Moreover, the core contains exactly $m' - m - (n' - n)$ boxes. 
		Observe also that we can reconstruct a shadow path from its core alone, given the starting and ending vertices $p$ and $q$. 
	\end{defn}

	\begin{figure}[htbp] 
		\centering
		\includegraphics[width=\textwidth]{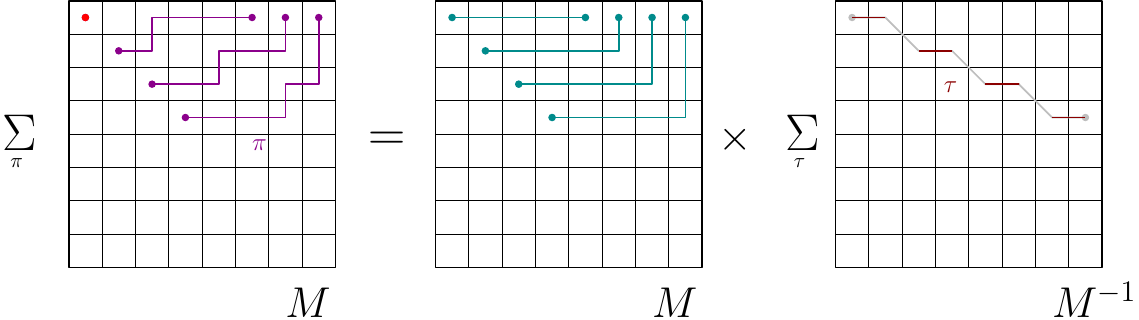} 
		\caption{An illustration of the shadow path identity given in Equation \eqref{E:shadow-identity}}
		\label{fig:shadow_identity}
	\end{figure} 
	
	We define shadow last-passage and shadow partition functions as follows. 
	For $p = (m, n)$ and $q = (m', n')$ with $m \le m'$ and $n \le n'$, let
	\begin{equation}
		\label{E:ZMLMZ}
		\check L_M(p, q) = \max_{\tau: p \to q} \sum_{x \in C(\tau)} M_x, 
		\qquad 
		\check Z_M(p, q) = \sum_{\tau: p \to q} \prod_{x \in C(\tau)} M_x,
	\end{equation}
	where the maximum and sum are taken over all shadow paths from $p$ to $q$. 
	Note that we only take sums and products over vertices in the core. 
	These definitions are set up to satisfy the following relationships:
	\begin{equation}
		\label{E:shadow-identity}
		\begin{split}
			L_M \Big((m, n)^{(n' - n + 1) \searrow}; &((m, n), (m', n)^{(n' - n) \leftarrow} ) \Big)  \\
			&= L_M \left((m, n)^{(n' - n + 1) \searrow}; (m', n)^{(n' - n + 1)\leftarrow } \right) + \check L_{-M} \left(m, n; m', n' \right), \\
			Z_M \Big((m, n)^{(n' - n + 1) \searrow}; & ((m, n), (m', n)^{(n' - n) \leftarrow} ) \Big) \\
			&= Z_M((m, n)^{(n' - n + 1) \searrow}; (m', n)^{(n' - n + 1) \leftarrow} ) \, \check Z_{M^{-1}}(m, n; m', n'),
		\end{split}
	\end{equation}
	where $-M$ and $M^{-1}$ denote the array $M$ with each entry $M_x$ replaced by $-M_x$ or $M_x^{-1}$, respectively, and for $p \in \Z^2$
	\begin{align*}
		p^{k \searrow} \defeq  (p, p + (1, 1), \ldots, p + (k - 1, k - 1)). 
	\end{align*}
	An illustration of this identity is given in Figure \ref{fig:shadow_identity}.

	We can also define multi-point shadow last-passage and multi-point shadow partition functions. 
	For $k$-tuples $\bp, \bq$, define
	\begin{equation}
		\label{E:ZMLMZ-multi}
		\check L_M(\bp, \bq) = \max_{\tau: \bp \to \bq} \sum_{x \in \cup C(\tau)} M_x, 
		\qquad 
		\check Z_M(\bp, \bq) = \sum_{\tau: \bp \to \bq} \prod_{x \in \cup C(\tau)} M_x,
	\end{equation}
	where the maximum and sum are over all disjoint $k$-tuples of shadow paths $\tau = (\tau_1, \dots, \tau_k)$, 
	with each $\tau_i$ a shadow path from $p_i$ to $q_i$. 
	Here, $\cup C(\tau) := \bigcup_{i=1}^k C(\tau_i)$. 
	As with last passage, we only use these definitions if a disjoint $k$-tuple of shadow paths exists. 
	Note that in certain cases, $\cup C(\tau)$ may be empty. 
	In this setting, we use the conventions that the empty sum is $0$ and the empty product is $1$.
	
	To prove the main theorem of this section, we will need a version of the Lindström–Gessel–Viennot (LGV) lemma for shadow partition functions. 
	This is a simple variant of the usual LGV lemma, but we must take a bit of extra care in the proof since we only take products of weights that lie in the cores.
	
	\begin{lemma}
		\label{L:Lindstrom-Gessel-Viennot}
		Let $\bp = ((m_1, n_1), \dots, (m_k, n_k))$ and $\bq = ((m_1', n_1'), \dots, (m_k', n_k')) \in (\Z^2)^k$ satisfy:
		\begin{itemize}[nosep]
			\item There exists at least one disjoint $k$-tuple of shadow paths from $\bp$ to $\bq$.
			\item For all $i \in \II{1, k-1}$, we have $m_i \le m_{i+1}$, $m_i' \le m'_{i+1}$, $n_i \ge n_{i+1}$, and $n_i' \ge n_{i+1}'$.
		\end{itemize}
		Then, for any environment $M:\Z^2 \to \R$, we have
		\begin{equation}
			\label{E:LGV}
			\check Z_M(\bp, \bq) = \det_{1 \le i, j \le k} \check Z_M(p_i, q_j).
		\end{equation}
	\end{lemma}

	\begin{proof}
		To simplify the setup, we assume that $M$ is never equal to $0$. 
		The case when $M$ may equal $0$ follows by continuity of both sides of \eqref{E:LGV}.
		
		Consider the directed graph $G = (V := \Z^2 \cup (\Z^2 + (1/2, 1/2)), E)$, 
		where $(u, v) \in E$ whenever $v - u = (1/2, 1/2)$ or $v - u = (1, 0)$ with $u \in \Z^2$. 
		
		Given a collection of nonzero vertex weights $M$ on $\Z^2$, define edge weights $M^*$ on $G$ by
		\begin{align*}
			M^*(u, v) = 
			\begin{cases}
				M(v), & v \in \Z^2, \\
				M(v + (1/2, 1/2))^{-1}, & v \notin \Z^2.
			\end{cases}
		\end{align*}
		
		Let $p, q \in \Z^2$ be such that a shadow path exists from $p$ to $q$. 
		There is a bijection $f$ between shadow paths from $p$ to $q$ and directed paths in $G$ from $p$ to $q$: 
		
		\begin{itemize}
			\item Starting from a directed path, forgetting the vertices in $\Z^2 + (1/2, 1/2)$ yields a shadow path.  
			\item Starting from a shadow path $\tau = (\tau_0, \dots, \tau_\ell)$, adding all vertices $\tau_{i-1} + (1/2, 1/2)$ whenever $\tau_i - \tau_{i-1} = (1, 1)$ yields a directed path.  
		\end{itemize}
		
		These operations are inverses. Moreover, our definitions are set up so that this bijection is weight-preserving, in the sense that
		\begin{align*}
			\prod_{v \in C(\tau)} M_v = \prod_{e \in f(\tau)} M^*_e.
		\end{align*}
		Finally, two shadow paths $\tau, \tau'$ are disjoint if and only if $f(\tau), f(\tau')$ are disjoint. 
		The lemma then follows by applying the usual Lindström–Gessel–Viennot lemma to compute the partition function of all non-intersecting paths in $M^*$ from $\bp$ to $\bq$.
	\end{proof}

	The main theorem of this section is that shadow-paths along tableaux $P^{-1}, Q^{-1}, (P^s)^{-1}, (Q^s)^{-1}$ encode partition functions across $M$ that start on the left/bottom edge and end at the top/right edge.

	\begin{thm}[DtM Isometry] \label{thm:RSK_Representation_discrete}
		Let $M$ be a positive matrix on $\II{1,m}\times\II{1,n}$,
		with tableaux $P=\Phi(M)$, $Q=\Phi(M^F)$ and common shape
		$\lambda=(\lambda_1,\ldots,\lambda_r)$, where $r=m\wedge n$.
		Let $p=(i,n)$ with $1\le i\le m$, or $p=(1,n-i+1)$ with
		$1\le i\le n$. Similarly, let $q=(m,j)$ with $1\le j\le n$, or
		$q=(m-j+1,1)$ with $1\le j\le m$, and 
		suppose that $p$ can reach $q$ by an up-right path. 
		
		Let $c \in \{ 0, 1\}$ be given by 
		\begin{align*}
			c &= \begin{cases}
				0 ,& p =(i,n) \, \& \, q = (m, j) \textrm{ or } p = (1 , n - i + 1) \, \& \, q = (m - j + 1, 1) \\
				1 , & \textrm{otherwise}.
			\end{cases}
		\end{align*}
		In addition, define two tableaux $U,V$ as
		\begin{align*}
			U &= \begin{cases}
				P^s, & p = (i, n) \\
				Q^s, & p = (1, n - i + 1)
			\end{cases}
			, \qquad
			V = \begin{cases}
				Q, & q = (m, j) \\
				P, & q = (m - j + 1, 1)
			\end{cases}
		\end{align*}
		and let $d_U, d_V$ be the number of columns in $U$ and $V$, respectively.
		Then
		\[
		Z_M(p,q)=\sum_{h=1}^{\min(i,j, m , n)}(-1)^{c(h-1)}\lambda_h
		\check Z_{U^{-1}}(d_U-i+1,1;d_U,h)
		\check Z_{V^{-1}}(d_V-j+1,1;d_V,h).
		\]
	\end{thm}
	\begin{figure}[htbp] 
		\centering
		\includegraphics[width=0.5\textwidth]{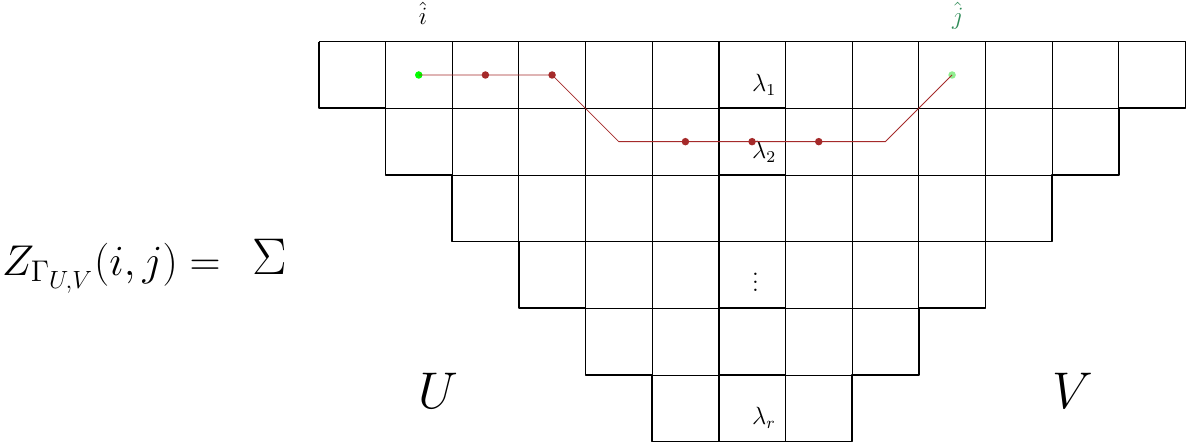} 
		\caption{An illustration of the graph $\Lambda_{U,V}$ and the partition function values $Z_{\Lambda_{U,V}} (i, j)$.}
		\label{fig:directed_graph}
	\end{figure} 
	
	\begin{remark}
		A diagram of this representation is shown in Figure \ref{fig:directed_graph}. It may be interpreted as arising from stitching together two tableaux and computing the associated shadow paths, beginning at the top of each tableau and proceeding downward to a common midpoint. For this reason, we refer to this construction as the down-the-middle (DtM) representation.
	\end{remark}
	
	\begin{remark}
		Since the partition functions from bottom to top (respectively, from left to right) across \( M \) coincide with those across \( P \) (respectively, \( Q \)) by Theorem \ref{T:encoding-lpp}, we obtain the following relation. Applying Theorem \ref{thm:RSK_Representation_discrete} with \( p = (i,n) \) and \( q = (m - j + 1, 1) \), we deduce that
		\begin{align*}
			Z_{P} (\bar p, q) &= \sum_{k = 1}^{\min(i,j,m,n)} (-1)^{(k - 1)}\lambda_{k} \check{Z}_{(P^s)^{-1}} ( m - i + 1 , 1 ; m, k) \check{Z}_{P^{-1}} (m - j + 1 , 1 ; m, k).
		\end{align*}
		Here $\bar p$ is the projection defined in Theorem \ref{T:encoding-lpp}. A similar result also holds for left to right partition functions across $Q$. Hence, in these two cases, this theorem simply corresponds to an alternating series expansion of single-path partition functions across $P$ (respectively, $Q$). 
	\end{remark}
	
	The proof of theorem \ref{thm:RSK_Representation_discrete} follows from the following two lemmas.
	
	\begin{lemma} \label{lemma:discrete_nonintpaths}
		Fix an $n \times m$ matrix of positive entries $M$ and let $i, j \in \N$. Let $p = (i,n)$ or $p = (1,n-i+1)$, and $q = (m-j+1,1)$ or $q = (m,j)$, with both endpoints in $M$, and suppose that $p$ can reach $q$ by an up-right path. For any $k \le \min(i, j, m, n)$, define the $k$-tuples $u_k, v_k \in ( \II{1, m} \times \II{1, n})^k$ and the constant $c \in \{0, 1 \}$ based on the following cases:
		
		\begin{align*}
			u_k &= \begin{cases}
				( o^{k-1}, p), & p = (i, n) \\
				(p, o^{k - 1}), & p = (1,n-i+1)
			\end{cases} , \qquad 
			v_k = \begin{cases}
				(\hat o^{k - 1}, q), & q = (m, j) \\
				(q, \hat o^{k - 1}), & q = (m-j+1,1)
			\end{cases}, 
			\\
			c &= \begin{cases}
				0 ,& p =(i,n) \, \& \, q = (m, j) \textrm{ or } p = (1 , n - i + 1) \, \& \, q = (m - j + 1, 1) \\
				1 , & \textrm{otherwise}
			\end{cases}
		\end{align*}
		
		For all $\ell < r_0 := \min (i, j, m, n)$,
		\begin{align}
			Z_M(p,q) &= \sum_{k=1}^\ell (-1)^{c (k - 1)}
			\frac{ Z_M( u_k, \hat o^k) \, Z_M(o^k, v_k )}{ Z_M(o^k, \hat o^k)^2 } 
			\cdot \frac{Z_M(o^k, \hat o^k)}{Z_M(o^{k-1}, \hat o^{k-1})} 
			+ (-1)^{c \ell}\frac{ Z_M( u_{\ell + 1}, v_{\ell + 1} ) }{ Z_M(o^\ell, \hat o^\ell) }. \label{eqn:alternating_sum}
		\end{align}
		When $\ell = r_0 -1$, the remainder term in \eqref{eqn:alternating_sum} takes on the same form as the  first $r_0 - 1$ terms because either
		$u_{r_0}=o^{r_0}$ or $v_{r_0}=\hat o^{r_0}$. Therefore in this case,
		\begin{align*}
			Z_M(p,q) &= \sum_{k=1}^{\min(i,j, m, n)} (-1)^{c (k - 1)} 
			\frac{ Z_M( u_k, \hat o^k) \, Z_M(o^k, v_k) }{ Z_M(o^k, \hat o^k)^2 } 
			\cdot \frac{Z_M(o^k, \hat o^k)}{Z_M(o^{k-1}, \hat o^{k-1})}.
		\end{align*}
		
	\end{lemma}

	\begin{proof}
		We first consider the case when $p = (i, n)$, $q = (m, j)$, in which case $c = 0$. We proceed by induction. For the base case, by LGV (Lemma \ref{LGV_thm}) we have
		\begin{align*}
			Z_M(\{ o, p\}, \{ \hat o, q\}) &= Z_M(o, \hat o) Z_M(p, q) - Z_M(o, q) Z_M(p, \hat o)
		\end{align*}
		Rearranging and using that $Z_M(o, \hat o) > 0$ gives the identity for $\ell=1$:
		\begin{align*}
			Z_M(p,q) &= \frac{Z_M(o,q) Z_M(p,\hat o)}{Z_M(o,\hat o)} + \frac{Z_M(\{ o, p\}, \{ \hat o, q\})}{Z_M(o,\hat o)}
		\end{align*}
		
		Assume the result holds for some $\ell < \min(i,j)$. By the Desnanot-Jacobi Identity (Lemma \ref{lemma:desnanot_jacobi}):
		\begin{align*}
			Z_M(\{ o^{\ell+1}, p\}, \{ \hat o^{\ell+1}, q\}) Z_M(o^\ell,\hat o^\ell) 
			&= Z_M(\{ o^\ell, p\}, \{ \hat o^\ell, q\}) Z_M(o^{\ell+1}, \hat o^{\ell+1}) - Z_M(\{ o^\ell, p\}, \hat o^{\ell+1}) Z_M(o^{\ell+1}, \{ \hat o^\ell, q\}).
		\end{align*}
		Rearranging gives:
		\begin{align*}
			\frac{Z_M(\{ o^\ell, p\}, \{ \hat o^\ell, q\})}{Z_M(o^\ell,\hat o^\ell)} 
			&= \frac{ Z_M(\{ o^\ell, p\}, \hat o^{\ell+1}) Z_M(o^{\ell+1}, \{ \hat o^\ell, q\}) }{ Z_M(o^\ell, \hat o^\ell) Z_M(o^{\ell+1}, \hat o^{\ell+1}) }  + \frac{ Z_M(\{ o^{\ell+1}, p\}, \{ \hat o^{\ell+1}, q\})}{ Z_M(o^{\ell+1}, \hat o^{\ell+1}) }
		\end{align*}
		Substituting this expression for the remainder term in the identity for $\ell$ proves the $\ell + 1$ case. 
		
		The proof for the case when $p$ is on the left edge and $q$ is on the top edge is almost identical, except that the ordering of $u_k$ and $v_k$ is altered to ensure that $Z(u_k, \hat{o}^k), Z(o^k, v_k) > 0$, and correspond to matrices of single-point partition functions as per the LGV Theorem. When $p$ and $q$ are on opposite facing edges of $M$, a negative sign is introduced at each induction step. Supposing that $p$ is on the left edge and $q$ is on the right edge, the Desnanot-Jacobi Identity now gives:
		\begin{align*}
			Z_M \left(\{ p, o^{\ell + 1}\}  , \{ \hat{o}^{\ell + 1}, q\} \right) Z_M(o^{\ell} , \hat{o}^{\ell}) &= Z_{M} \left( \{ p, o^{\ell}\} , \hat{o}^{\ell + 1} \right) Z_{M} \left( o^{\ell + 1} , \{ \hat{o}^{\ell}, q \} \right) - Z_{M}(o^{\ell + 1} , \hat{o}^{\ell + 1}) Z_M \left( \{ p, o^{\ell}\}  , \{ \hat{o}^{\ell} , q\} \right). 
		\end{align*}
		Re-arranging now gives
		\begin{align*}
			\frac{Z_{M} (\{p,  o^{\ell}\} , \{ \hat{o}^{\ell} , q \})}{Z_M(o^\ell, \hat{o}^{\ell})} &= \frac{Z_{M} \left( \{ p, o^{\ell}\} , \hat{o}^{\ell + 1} \right) Z_{M} \left( o^{\ell + 1} , \{ \hat{o}^{\ell}, q \} \right)}{Z_M (o^{\ell} , \hat{o}^{\ell}) Z_M (o^{\ell + 1} , \hat{o}^{\ell + 1}) } -  \frac{	Z_M \left(\{ p, o^{\ell + 1}\}  , \{ \hat{o}^{\ell + 1}, q\} \right) }{Z_M(o^{\ell + 1} , \hat{o}^{\ell + 1})}.
		\end{align*}
		The minus sign in the error term now results in the alternating sum. The induction step for when $p$ is at the bottom and $q$ is at the top is identical. 
	\end{proof}
	
	\begin{lemma} \label{lemma:rsk_formulas_discrete}
		Fix an $n \times m$ matrix of positive entries $M$. Define tableaux $P = \Phi(M)$ and $Q = \Phi(M^F)$ with shape $\sh(P) = \sh(Q) = \lambda$. Then $\lambda_k = \frac{Z_M(o^k, \hat o^k)}{Z_M(o^{k-1}, \hat o^{k-1})}$ for $1\leq k\leq m\wedge n$, and the following four identities hold:
		\begin{enumerate}
			\item If $p = (i,n)$ with $1\leq i\leq m$, then for $1\leq k\leq i\wedge n$:
			\begin{align*}
				\frac{Z_M(\{o^{k-1},p\},\hat o^k)}{Z_M(o^k,\hat o^k)} &= \check Z_{(P^s)^{-1}}(m-i+1,1;m,k)
			\end{align*}
			\item If $p = (1,n-i+1)$ with $1\leq i\leq n$, then for $1\leq k\leq i\wedge m$:
			\begin{align*}
				\frac{Z_M(\{p, o^{k-1}\},\hat o^k)}{Z_M(o^k,\hat o^k)} &= \check Z_{(Q^s)^{-1}}(n-i+1,1 ;n,k)
			\end{align*}
			\item If $q = (m,j)$ with $1\leq j\leq n$, then for $1\leq k\leq j\wedge m$:
			\begin{align*}
				\frac{Z_M(o^k, \{\hat o^{k-1}, q\})}{Z_M(o^k,\hat o^k)} &= \check Z_{Q^{-1}}(n-j+1,1; n,k)
			\end{align*}
			\item If $q = (m-j+1, 1)$ with $1\leq j\leq m$, then for $1\leq k\leq j\wedge n$:
			\begin{align*}
				\frac{Z_M(o^k, \{q, \hat o^{k-1}\})}{Z_M(o^k,\hat o^k)} &= \check Z_{P^{-1}}(m-j+1,1; m,k).
			\end{align*}
		\end{enumerate}
	\end{lemma}
	
	\begin{proof}
		The identity for $\lambda_k$ is standard, and follows from the definition of $\Phi$, which allows us to identify $\lambda_1 \cdots \lambda_k$ as a telescoping product, equal to $Z_M(o^k, \hat o^k)$.  The proofs of the remaining identities are all similar, so we only handle case $1$. We have
		\begin{align*}
			Z_M(\{ o^{k-1}, p\}, \hat o^{k}) 
			&= Z_{RM}(o^{k}, \{p^*, \hat o^{k-1} \}) 
			= Z_{P^s}(\overline{o^{k}}, \{p^*, \hat o^{k-1} \})
		\end{align*}
		where in the last equality we used Theorem \ref{T:encoding-lpp} and the fact that $P^s = \Phi(RM)$. Here the $R$-notation and $*$-notation for rotations is as in Definition \ref{defn:Schutzenberger}. Now recalling the shadow-path identity \eqref{E:shadow-identity}, we have 
		\begin{align*}
			Z_{P^s}(\overline{o^{k}}, \{ p^*, \hat o^{k-1} \}) &=  Z_{P^s}(\overline{o^k},\hat o^k) \check Z_{(P^s)^{-1}}(p^*; m, k) = Z_{M}(o^k,\hat o^k) \check Z_{(P^s)^{-1}}(m-i+1, 1;m, k).
		\end{align*}
		The proofs of the other three cases are similar. 
	\end{proof}
	
	\begin{proof}[Proof of Theorem \ref{thm:RSK_Representation_discrete}]
		Substituting the relations from Lemma \ref{lemma:rsk_formulas_discrete} into the result of Lemma \ref{lemma:discrete_nonintpaths} proves the theorem. 
	\end{proof}

	\subsection{Extension to the multi-path case}
	
	Using Theorem \ref{thm:RSK_Representation_discrete}, we may write partition functions from the bottom to right edge of $M$, or from the left to top edge of $M$, in terms of a single partition function across a directed graph $\Lambda$. Doing so allows us to apply the LGV theorem and immediately extend Theorem \ref{thm:RSK_Representation_discrete} to the multi-path case.
	\begin{lemma}[DtM Multi-path Partition Functions] \label{lemma:RSK_discrete_multipt}
		Let $\mathbf{i} = (i_1 , \ldots, i_k) \in \N^k_{<}$, $\mathbf{j} = (j_1 , \ldots, j_k) \in \N^k_{<}$, and set $\ell = \min(i_k, j_k,m,n)$. Let $\bp = (\mathbf{i},n)$ or $\bp = (1,n-\mathbf{i}+1)$, and $\bq = (m-\mathbf{j}+1,1)$ or $\bq = (m,\mathbf{j})$, with all endpoints in $M$. Assume that every $p_a$ can reach every $q_b$ by an up-right path. Let $c \in \{ 0, 1\}$ be given by 
		\begin{align*}
			c &= \begin{cases}
				0 ,& \bp =(\mathbf{i},n) \, \& \, \bq = (m, \mathbf{j}) \textrm{ or } \bp = (1 , n - \mathbf{i} + 1) \, \& \, \bq = (m - \mathbf{j} + 1, 1) \\
				1 , & \textrm{otherwise}.
			\end{cases}
		\end{align*}
		Write $\widehat{\bq}=\bq$ when $c=0$, and $\widehat{\bq}=(q_k,\ldots,q_1)$ when $c=1$.
		In addition, define two tableaux $U,V$ as
		\begin{align*}
			U &= \begin{cases}
				P^s, & \bp = (\mathbf{i}, n) \\
				Q^s, & \bp = (1, n - \mathbf{i} + 1)
			\end{cases}
			, \qquad
			V = \begin{cases}
				Q, & \bq = (m, \mathbf{j}) \\
				P, & \bq = (m - \mathbf{j} + 1, 1)
			\end{cases}
		\end{align*}
		and let $d_U, d_V$ be the number of columns in $U$ and $V$ respectively.
		Then
		\begin{align}
			\label{E:multipath-part}
			Z_M( \mathbf{p}, \widehat{\mathbf{q}})
			=(-1)^{c\binom{k}{2}} \sum_{I \in \II{1, \ell}^k_{<} }
			\Big( \prod_{h \in I} (-1)^{c(h-1)}\lambda_h \Big)
			\check{Z}_{U^{-1}} (d_U-\mathbf{i}+1,1;d_U,I)
			\check{Z}_{V^{-1}} (d_V-\mathbf{j}+1,1;d_V,I).
		\end{align}
		On the right hand side above, we use the notation $(d_U-\mathbf{i}+1,1)$ for the vector $((d_U-i_x+1,1),x=1,\dots,k)$, and similarly for $(d_U,I)$. Partition functions of empty path families are taken to be zero.
	\end{lemma}
	
	\begin{proof} Letting $U, V$ be the two tableaux that appear in Theorem \ref{thm:RSK_Representation_discrete}, define the subset $\Gamma = \Gamma_{U,V} \subset \Z^2$ given by
		\begin{align*}
			\Gamma & \defeq \{ (x, y) \in \Z^2 \, : \, y \in \II{1, \ell}, \, -(d_U-y+1) \leq x \leq d_V-y+1 \} ,
		\end{align*} 
		and define $\Lambda_{U,V}: \Gamma \to \R$ by:
		\begin{align*}
			\Lambda_{U,V}(x,y) &= \begin{cases}
				(U(d_U+x+1,y))^{-1} & x \leq -1 \\
				(-1)^{c(y-1)}\lambda_y & x=0 \\
				(V(d_V-x+1,y))^{-1} & x \geq 1
			\end{cases}. 
		\end{align*}
		We consider $\Lambda_{U,V}$ as a directed graph with directed edges from vertex $(x,y)$ to vertex $(x',y')$ in the following cases:
		\begin{itemize}
			\item $x,x' \leq -1$, $x'-x=1$, $y'-y\in\{0,1\}$,
			\item $x,x' \geq 1$, $x'-x=1$, $y'-y\in\{0,-1\}$,
			\item $x' = 0$, $x' - x = 1$, $y = y'$,
			\item $x' = 1$, $x' - x = 1$, $y = y'$.
		\end{itemize}
		For a path $\tau = (\tau_1 , \ldots, \tau_r)$ in $\Lambda$, we extend the concept of its core as:
		\begin{align*}
			C(\tau)
			&= \{\tau_i=(x,y):2\le i\le r,\ x\le-1,\ \tau_i-\tau_{i-1}=(1,0)\}\\
			&\quad{}\cup\{\tau_i=(0,y):1\le i\le r\}\\
			&\quad{}\cup\{\tau_i=(x,y):1\le i<r,\ x\ge1,\ \tau_{i+1}-\tau_i=(1,0)\}.
		\end{align*}
		Thus the core includes the middle vertex, and on the right it collects the departure vertices of horizontal steps, since this side is a reversed shadow path.
		We then define the partition function for single points $p$ to $q$ across $\Lambda_{U,V}$ as:
		\begin{align*}
			Z_{\Lambda_{U,V}}(p, q) &= \sum_{\tau: p \to q} \prod_{s \in C(\tau)} \Lambda_{U,V}(s) .
		\end{align*}
		A diagram of $\Lambda = \Lambda_{U,V}$ is illustrated in Figure \ref{fig:directed_graph}. The partition function $Z_{\Lambda_{U,V}}$ corresponds to going down and right along the left side of $\Lambda$, starting at $(-i, 1)$ and turning around at $(0,h)$ for some $h\in\II{1,\ell}$ and then going up and right along $\Lambda$ to the point $(j, 1)$. Since going down and right along $\Lambda$ corresponds to a shadow path across $U^{-1}$ (and similarly for the right side of $\Lambda$ after reversal), we can also see $Z_{\Lambda}$ as a weighted sum of shadow paths across $U^{-1}$ and $V^{-1}$. More precisely, for $i\in\{i_1,\ldots,i_k\}$ and $j\in\{j_1,\ldots,j_k\}$,
		\begin{align}
			\label{E:LambdaUV}
			Z_{\Lambda_{U,V}}(-i,1;j,1)
			=\sum_{h=1}^{\ell}(-1)^{c(h-1)}\lambda_h
			\check{Z}_{U^{-1}}(d_U-i+1,1;d_U,h)
			\check{Z}_{V^{-1}}(d_V-j+1,1;d_V,h)
			=Z_M(p,q),
		\end{align}
		where the last equality is the result from Theorem \ref{thm:RSK_Representation_discrete}, and $p \in \{(i,n),(1,n-i+1)\}$, $q \in \{(m,j),(m-j+1,1)\}$, with the exact choice of $p,q$ matching the choice of $\bp,\bq$ in the lemma.
		
		Now, arguing as in Lemma \ref{L:Lindstrom-Gessel-Viennot}, the Lindstr\"om-Gessel-Viennot lemma holds for partition functions across $\Lambda$. Indeed, assign each horizontal edge on the left the weight at its arrival vertex, each horizontal edge on the right the weight at its departure vertex, and each edge into $(0,h)$ the weight $(-1)^{c(h-1)}\lambda_h$; assign all other edges weight $1$. This gives the same weights for the paths in question, and LGV applies also to signed edge weights. Thus
		\begin{align}
			\label{E:LGV-shadows}
			Z_{\Lambda_{U,V}}(-\mathbf{i},1;\mathbf{j},1)
			&=\det\left(Z_{\Lambda_{U,V}}(-i_x,1;j_y,1)\right)_{x,y=1}^{k}.
		\end{align}
		Expanding out the multi-path partition function on the left-hand side of \eqref{E:LGV-shadows} according to its increasing tuple of middle vertices gives the sum on the right-hand side of \eqref{E:multipath-part}. On the other hand, plugging \eqref{E:LambdaUV} into the right-hand side of \eqref{E:LGV-shadows} and applying the LGV lemma gives
		\[
		Z_{\Lambda_{U,V}}(-\mathbf{i},1;\mathbf{j},1)
		=(-1)^{c\binom{k}{2}}Z_M(\bp,\widehat{\bq}),
		\]
		since the compatible endpoint order in $M$ is reversed exactly when $c=1$. This proves \eqref{E:multipath-part}.
	\end{proof}
	
	\subsection{Discrete LPP Models}
	
	By formally exchanging the $(+, \times)$-algebra for the $(\max, +)$-algebra, we can deduce that the same results hold for last passage percolation. To state the corollary, we first define the last passage analogue $\Psi$ of the map $\Phi$.
	\begin{defn} (The RSK correspondence)
		\label{D:LPRSK}
		Let $M:\II{1, m} \times \II{1, n} \to \R$, and recall that $u_i = (1, n; i, 1)$. Define the map $\Psi:F(\II{1, m} \times \II{1, n}, \R) \to F(\Gamma_{m, n}, \R)$ by:
		$$
		\Psi(M)(i, j) = \begin{cases}
			L_M(u_i^{j}) - L_M(u_{i}^{j-1}), \qquad &j = i \\
			L_M(u_i^{j}) +  L_M(u_{i-1}^{j-1}) - L_M(u_{i-1}^{j}) - L_M(u_{i}^{j-1}), \qquad &j < i
		\end{cases}
		$$
		where we use the convention that $L_M(u_i^0) = 0$. Set $P=\Psi(M)$ and $Q=\Psi(M^F)$. We define Gelfand–Tsetlin patterns $\mu = (\mu_{ij})_{1 \leq i \leq m,\,1 \leq j \leq i\wedge n}$ and  
		$\nu = (\nu_{ij})_{1 \leq i \leq n,\,1 \leq j \leq i\wedge m}$ corresponding to $P$ and $Q$, respectively, as
		\begin{align}
			\mu_{ij} &= \sum_{k = j}^i P_{k j}, \qquad \nu_{ij} = \sum_{k = j}^i Q_{k j}.
		\end{align}
		Since $L_{M} (\cdot, \hat{o}^{k \leftarrow }) = L_{M}(\cdot, \hat{o}^{k \downarrow})$, it follows that the bottom rows of the two Gelfand-Tsetlin patterns are equal. We call this bottom row the shape, $\sh (P) = \sh(Q) = (\lambda_1, \ldots, \lambda_{n\wedge m})$ which is given by
		\begin{align}
			\lambda_j &= \sum_{k = j}^m P_{k j} = \sum_{k = j}^n Q_{k j}.
		\end{align}
		The RSK correspondence is the map
		\begin{align*}
			\mathrm{RSK}: M \mapsto (P, Q).
		\end{align*}
		Again, we refer to $P, Q$ as \textbf{tableaux}, since when $M$ is positive integer-valued, they are simple transformations of the semi-standard Young tableaux that appear in the usual RSK correspondence.
	\end{defn}
	By exchanging the underlying algebra in Definition \ref{defn:Schutzenberger}, one recovers the classical Schützenberger involution from its geometric counterpart. We use the same notation for both involutions, as the intended meaning will be clear from context.
	
	\begin{defn} \label{defn:Schutzenberger_LPP}
		Fix a tableau $P = (p^i_j)_{1 \leq j \leq m,\,1 \leq i \leq j\wedge n}$. 
		For $i \in \II{1, n}$, $j \in \II{1, m}$ and $i \leq j$, let 
		\begin{align*}
			\sigma^{i}_j &= L_P \left( \overline{(m - j + 1, n)^{i \rightarrow}} ;  (m, 1)^{i \leftarrow}  \right) .
		\end{align*}
		Set $\sigma_j^0=0$ for $0\leq j\leq m$. That is, $\sigma^{i}_j$ is obtained from the multi-path LPP value illustrated in Figure \ref{fig:Schutzenberger}. 
		Then the classical Schützenberger involution is the tableau $P^s = (\tilde{p}^i_j)_{1 \leq j \leq m,\,1 \leq i \leq j\wedge n}$, where $\tilde{p}^i_j$ is given by
		\begin{align*}
			\tilde{p}^i_i &= \sigma^{i}_i - \sigma^{i-1}_i, \qquad  
			\tilde{p}^i_j = \sigma^i_j +  \sigma^{i-1}_{j-1} -\sigma^{i}_{j-1} - \sigma^{i-1}_j\quad (i<j). 
		\end{align*}
		We have $\sh(P) = \sh(P^s)$. 
	\end{defn}
	
	We now exchange the algebras in Theorem \ref{thm:RSK_Representation_discrete}/Lemma \ref{lemma:RSK_discrete_multipt} to obtain a DtM representation for the RSK correspondence. Our first corollary corresponds to the cases in those statements when $c = 0$. It is immediate from Lemma \ref{lemma:RSK_discrete_multipt} and Lemma \ref{lemma:temp_change}.
	
	\begin{cor} \label{cor:RSK_discrete_LPP}
		Let $\mathbf{i} = (i_1 , \ldots, i_k) \in \N^k_{<}$, $\mathbf{j} = (j_1 , \ldots, j_k) \in \N^k_{<}$, and set $\ell = \min(i_k, j_k, m, n)$. Suppose that $\mathbf{p}, \mathbf{q} \in (\II{1, m} \times \II{1, n})^k$ are given by either
		\begin{align*}
			\mathbf{p} &= (\mathbf{i}, n) \,,\, \mathbf{q}  = (m , \mathbf{j}) \quad \textrm{OR} \quad \mathbf{p} = (1, n - \mathbf{i} + 1) \, , \, \mathbf{q} = (m - \mathbf{j} + 1, 1). 
		\end{align*}
		Let $M$ be an $m \times n$ matrix of real entries with corresponding tableaux $P = \Psi(M)$, $Q = \Psi(M^F)$. Define two tableaux $U,V$ as
		\begin{align*}
			U &= \begin{cases}
				P^s, & p = (i, n) \\
				Q^s, & p = (1, n - i + 1)
			\end{cases}
			, \qquad
			V = \begin{cases}
				Q, & q = (m, j) \\
				P, & q = (m - j + 1, 1)
			\end{cases}
		\end{align*}
		where $P^s, Q^s$ correspond to the classical Schützenberger involutions of $P$ and $Q$ respectively. Let $\operatorname{col}(U), \operatorname{col}(V)$ be the number of columns in $U$ and $V$ respectively. 
		Then
		\begin{align*}
			L_M( \mathbf{p}, \mathbf{q}) = \max_{I \in \II{1, \ell}^k_{<} }  ( \sum_{k \in I} \lambda_{k})  + \check{L}_{-U} (\operatorname{col}(U) - \mathbf{i} + 1 , 1 ; \operatorname{col}(U), I)  + \check{L}_{-V} (\operatorname{col}(V) - \mathbf{j} + 1 , 1 ; \operatorname{col}(V) , I).
		\end{align*}
	\end{cor}
	
	For LPP across opposite sides of $M$ (the $c=1$ case), some care must be taken due to the negative signs that appear in the right-hand side of \eqref{eqn:alternating_sum} and \eqref{E:multipath-part}. These can be dealt with by first moving the negative signs to the other side of the equality so that all signs are positive, and then making the algebra switch. Doing so gives a sequence of non-trivial equalities which happen to be equivalent to the RSK Isometry. We record the result for single-path LPP from bottom to top of $M$ below. An analogous, more complicated result can be obtained for the multi-path case.
	
	\begin{cor} Let $M : \II{1, m} \times \II{1, n} \to \R$. For $i, j \in \II{1, m}$ so that $i + j \leq m$, let $p = (i, n)$ and $q = (m - j + 1, 1)$. Let $\ell < \min(i,j)$. If $\ell$ is even,
		\begin{align*}
			&\max \left\{  L_{M}(p,q), \max_{k \leq \ell, k \operatorname{ even}} \left\{ L_{M}(u_k, \hat{o}^k) + L_{M}(o^k, v_k) - L_{M}(o^{k - 1} , \hat{o}^{k - 1}) - L_{M} (o^k , \hat{o}^k)  \right\}    \right\} \\
			&= \max \left\{   L_{M} (u_{\ell + 1} , v_{\ell+  1} ) - L_{M} (o^{\ell} , \hat{o}^{\ell}) , \,   \max_{k \leq \ell, k \operatorname{ odd}} \left\{ L_{M}(u_k, \hat{o}^k) + L_{M}(o^k, v_k) - L_{M}(o^{k - 1} , \hat{o}^{k - 1}) - L_{M} (o^k , \hat{o}^k)  \right\}   \right\}.
		\end{align*}
		If $\ell$ is odd,
		\begin{align*}
			&\max \left\{  L_{M}(p,q), \, L_{M} (u_{\ell + 1} , v_{\ell+  1} ) - L_{M} (o^{\ell}, \hat{o}^{\ell}) ,  \,\max_{k \leq \ell, k \operatorname{ even}} \left\{ L_{M}(u_k, \hat{o}^k) + L_{M}(o^k, v_k) - L_{M}(o^{k - 1} , \hat{o}^{k - 1}) - L_{M} (o^k , \hat{o}^k)  \right\}    \right\} \\
			&=    \max_{k \leq \ell, k \operatorname{ odd}} \left\{ L_{M}(u_k, \hat{o}^k) + L_{M}(o^k, v_k) - L_{M}(o^{k - 1} , \hat{o}^{k - 1}) - L_{M} (o^k , \hat{o}^k)  \right\}.
		\end{align*} 
	\end{cor}
	
	To help understand the corollary, it helps to look at the simplest case when $\ell = 1$, which reduces to
	$$
	\max \{  L_{M}(p,q), \, L_{M} ((o, p); (q, \hat o)) - L_{M} (o,  \hat{o}) \} = L_M(p, \hat o) + L_M(o, q) - L_M(o,\hat o).
	$$
	It is easy to check this by hand. Indeed, if there exist disjoint geodesics from $o$ to $q$ and $p$ to $\hat o$, then $L_{M} ((o, p); (q, \hat o)) = L_M(o, q) + L_M(p, \hat o)$, and the equality follows from the quadrangle inequality 
	$$
	L_M(p, \hat o) + L_M(o, q) - L_M(o,\hat o) - L_M(p, q) \ge 0.
	$$
	If all pairs of geodesics from $o$ to $q$ and $p$ to $\hat o$ overlap, then the quadrangle inequality above becomes an equality, and we also have the strict inequality $L_{M} ((o, p); (q, \hat o)) < L_M(o, q) + L_M(p, \hat o)$. Together with some algebra, this gives the same identity.

	\subsection{Partition Functions Across the Tableaux}
	\label{SS:part-tableaux}

	In the previous sections, we considered an $n \times m$ matrix $M$ and computed partition functions starting from the bottom/left edge and ending at the right/top edge in terms of shadow paths across tableaux. In this section, we will prove a related identity for partition functions across the tableaux as opposed to the original environment. Namely, for a tableau $P = \Phi M$ (or $P = \Psi M$ in the LPP case) we will consider partition/percolation values starting at a point $\overline{(i, n)}$ on the bottom edge of $P$ and ending at a point $(m, j)$ on the right edge of $P$. This identity will be important for passing to the limit of LPP across lines.
	
	For the case when $j = 1$ we obtain such an identity by two applications of Theorem \ref{T:encoding-lpp}:
	\begin{equation*}
		Z_{\Phi M}(\overline{(i, n)} ; m, 1) = Z_M(i, n ; m, 1) = Z_{RM}(1, n ; m + 1 - i , 1) = Z_{\Phi RM} (1, 1 ; m + 1 - i , 1).
	\end{equation*}
	In anticipation of the generalization of this formula to the $j > 1$ case, we make the following manipulation. Recognizing that for any $i > 1$, there is a unique up-right path from $(1, 1)$ to $(m + 1 - i , 1)$, for $i \ge 2$ we can write
	\begin{align}
		\nonumber
		Z_{\Phi RM} (1, 1 ; m + 1 - i , 1) &= \frac{Z_{\Phi RM} (1, 1 ; m , 1)}{Z_{\Phi RM} (m + 2 - i, 1 ; m , 1)} \\
		\label{E:LP1path}
		&=Z_{(\Phi RM)^{-1}} (m + 2 - i, 1 ; m , 1) Z_{\Phi RM} (1, 1 ; m , 1),
	\end{align}
	where here $(\Phi RM)^{-1}$ denotes the array whose entries are reciprocals of those in $\Phi RM$. When $j > 1$, $Z_{\Phi M}(\overline{(i, n)} ; m, j)$ is equal to an expression similar to \eqref{E:LP1path}, but with the first term in the product replaced by a shadow partition function.

	\begin{thm}
		\label{T:rotation-isometry}
		Fix $m, n \in \N$ and let  $M \in F(\II{1, m} \times \II{1, n}, (0, \infty))$. Let $\bp = (\mathbf{i}, n)$ and $\bq = (m, \mathbf{j})$ be two vectors in $(\Z^2)^k$ such that $1 \le i_1 < \dots < i_k \le m$ and $1 \le j_1 < \dots < j_k \le n$ and $i_\ell \ge j_\ell$ for all $\ell \in \II{1, k}$. Then
		\begin{align*}
			Z_{\Phi M}(\overline{\bp}, \bq) = \check Z_{(\Phi RM)^{-1}}(m + 1 - \mathbf{i}, 1 ; m, \mathbf{j}) \prod_{\ell=1}^k Z_{\Phi RM}(j_\ell, j_\ell ; m, j_\ell).
		\end{align*}
	\end{thm}
	
	\begin{remark}
		It is easy to check from the definitions that for $i \ge 2$ we have 
		$$
		\check Z_{(\Phi RM)^{-1}}(m + 1 - i, 1 ;m, 1) = Z_{(\Phi RM)^{-1}}(m + 2 - i, 1 ;m, 1),
		$$
		and so this theorem reduces to \eqref{E:LP1path} in the $k = 1, i_1 \ge 2, j_1 = 1$ case.
	\end{remark}
	
	\begin{proof}[Proof of Theorem \ref{T:rotation-isometry}]
		We start by proving the claim for single points. Let $p = (i, n)$ and $q = (m, j)$ for some $i \in \II{1, m}, j \in \II{1, n}$ with $i \ge j$. Define $\bu = ( \tilde \bp,  \tilde \bq) \in (\Z^4_\uparrow)^j$ by letting
		$$
		\tilde \bp = ((1, n)^{ (j-1) \rightarrow}, p), \qquad \tilde \bq = (m, 1)^{j \leftarrow}.
		$$
		Then by Theorem \ref{T:encoding-lpp}, we have that
		\begin{equation}
			\label{E:ZM-rotation}
			Z_M( \tilde \bp;  \tilde \bq) = Z_{\Phi M}(\overline{ \tilde \bp};  \tilde \bq).
		\end{equation}
		
		\begin{figure}[htbp]
			\centering
			\includegraphics[width=0.8\textwidth]{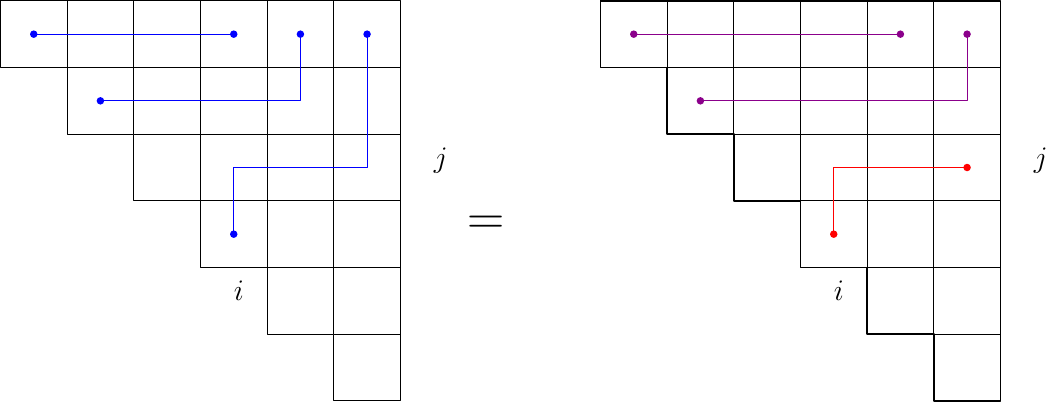} 
			\caption{A sketch of the identity given in \eqref{E:ZPhi-1} for $i = 4$, $j = 3$ and $m = n = 6$.}
			\label{fig:tableaux_paths}
		\end{figure} 
		
		Now, any disjoint $j$-tuple $\tau$ from $\overline{\tilde \bp}$ to $ \tilde\bq$ simply consists of the unique $(j-1)$-tuple going from $\overline{((1, n)^{ (j-1) \rightarrow })}$ to $(m-1, 1)^{(j - 1) \leftarrow}$ and a single path $\tau_j$ going from $\overline{(i, n)}$ to $(m, 1)$ which is forced to use all the boxes $(m, j), (m, j-1), \dots, (m, 1)$. We can reassign the boxes in the constituent paths of $\tau$ to obtain the unique $(j-1)$-tuple from $\overline{((1, n)^{ (j-1 ) \rightarrow})}$ to $(m, 1)^{(j - 1) \leftarrow}$ and a single unconstrained path from $\overline{(i, n)}$ to $(m, j)$, see Figure \ref{fig:tableaux_paths}. Therefore we can factor the partition function as follows,
		\begin{equation}
			\label{E:ZPhi-1}
			\begin{split}
				Z_{\Phi M}(\overline{ \tilde \bp};  \tilde \bq) &= Z_{\Phi M}(\overline{((1, n)^{ (j-1) \rightarrow})}; (m, 1)^{(j - 1) \leftarrow}) Z_{\Phi M}(\overline{p}; q) \\
				&= Z_{\Phi RM}(\overline{((1, n)^{ (j-1 ) \rightarrow })}; (m, 1)^{(j - 1) \leftarrow}) Z_{\Phi M}(\overline{p}, q).
			\end{split}
		\end{equation}
		Here the second equality uses Theorem \ref{T:encoding-lpp} twice, to go from $\Phi M$ to $M$, then from $RM$ to $\Phi RM$.
		
		We perform a similar computation in the environment $\Phi R M$. Again by Theorem \ref{T:encoding-lpp}, we have that
		\begin{equation}
			\label{E:ZM2}
			Z_{\Phi M}(\overline{ \tilde \bp},  \tilde \bq) =Z_M( \tilde \bp,  \tilde \bq) = Z_{RM}( \tilde \bq^*,  \tilde \bp^*) =  Z_{\Phi R M}( \overline{ \tilde \bq^*}, \tilde \bp^*).
		\end{equation}
		
		Now, $\overline{ \tilde \bq^*} = \overline{((1, n)^{j \rightarrow})} = ((1, 1), \dots, (j, j))$ and $\tilde \bp^* = ((m + 1 - i, 1), (m, 1)^{(j - 1) \leftarrow})$. Examining Figure \ref{fig:tableaux_shadowpaths}, for $i < m$ we see that any disjoint $j$-tuple of paths from $\overline{ \tilde \bq^*}$ to $ \tilde \bp^*$ is equivalent to the unique $j$-tuple of paths from $\overline{((1, n)^{j \rightarrow})}$ to $(m - i, 1)^{j \searrow}$ concatenated with any $j$-tuple of paths from $(m + 1 - i, 1)^{j \searrow}$ to $((m + 1 - i, 1), (m, 1)^{(j -1 ) \leftarrow} )$. Therefore using \eqref{E:shadow-identity} we have
		\begin{equation}
			\label{E:Zcalc}
			\begin{split}
				&Z_{\Phi R M}(\overline{ \tilde \bq^*}, \tilde \bp^*) \\
				=\; &Z_{\Phi R M}(\overline{((1, n)^j)} ; (m - i, 1)^{j \searrow}) Z_{\Phi RM}((m - i + 1, 1)^{j \searrow} ; (m + 1 - i, 1), (m, 1)^{(j - 1) \leftarrow} ) \\
				=\; &Z_{\Phi R M}(\overline{((1, n)^j)} ; (m - i, 1)^{j \searrow}) Z_{\Phi RM}((m - i + 1, 1)^{j \searrow} ; (m, 1)^{j \leftarrow}) \\
				&\quad \qquad \times \quad \check Z_{(\Phi RM)^{-1}}(m + 1 - i, 1 ; m, j) \\
				=\; &Z_{\Phi R M}(\overline{((1, n)^{j \rightarrow})} ; (m, 1)^{j \leftarrow}) \check Z_{(\Phi RM)^{-1}}(m + 1 - i, 1 ; m, j).
			\end{split}
		\end{equation}

		Here the third equality follows by noting that the unique $j$-tuple from $\overline{((1, n)^{j \rightarrow})}$ to $(m, 1)^{j \leftarrow}$ covers the exact same set of boxes as the union of the unique $j$-tuples from $\overline{(1, n)^{j \rightarrow}}$ to $(m - i, 1)^{j \searrow}$ and from $(m + 1 - i, 1)^{j \searrow}$ to $(m, 1)^{j \leftarrow}$. Note that when $i = m$, the equality of the first and last lines in \eqref{E:Zcalc} is immediate from \eqref{E:shadow-identity}.
		Now, we also have the immediate computation
		\begin{equation}
			Z_{\Phi RM}(\overline{((1, n)^{j \rightarrow})}, (m, 1)^{j \leftarrow} ) = Z_{\Phi RM}(\overline{((1, n)^{(j-1 ) \rightarrow} )}, (m, 1)^{(j - 1) \leftarrow} ) Z_{\Phi RM}((j, j), (m, j)),
		\end{equation}
		which follows by inspection. 
		\begin{figure}[htbp]
			\centering
			\includegraphics[width=0.5\textwidth]{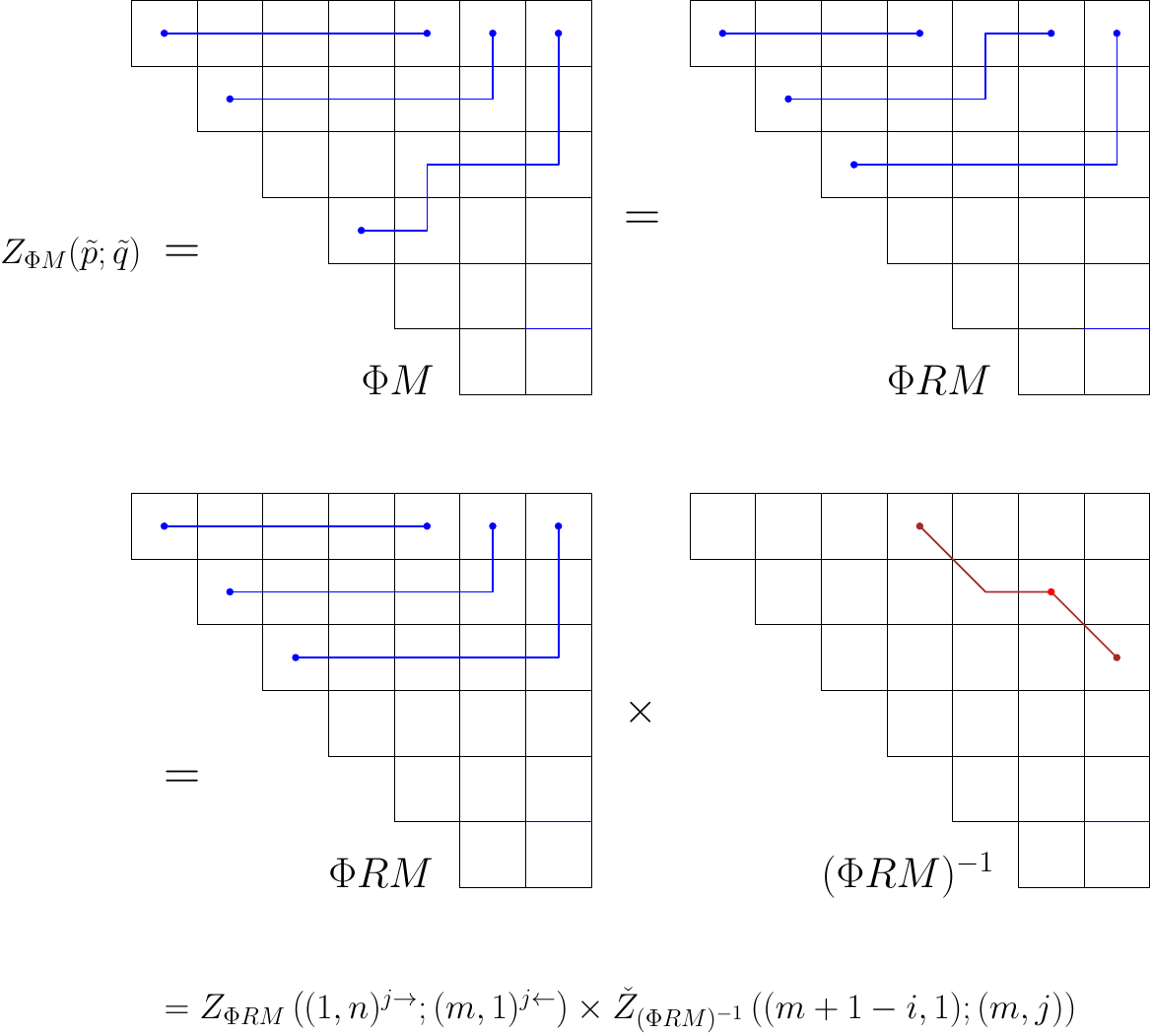}
			\caption{A sketch of the identity given in \eqref{E:Zcalc}}
			\label{fig:tableaux_shadowpaths}
		\end{figure} 
		Putting this together with \eqref{E:ZPhi-1}, \eqref{E:ZM2}, and \eqref{E:Zcalc} gives that
		\begin{align}
			\label{E:one-path-form}
			Z_{\Phi M}(\overline{p}, q) = Z_{\Phi RM}((j, j), (m, j)) \check Z_{(\Phi RM)^{-1}}((m + 1 - i, 1), (m, j)).
		\end{align}
		We now extend this formula to vectors of points $\bp = ((i_1, n), \dots, (i_k, n)), \bq = ((m, j_1), \dots, (m, j_k))$ as in the statement of the lemma. Indeed, by the usual LGV lemma, we have that
		$$
		Z_{\Phi M}(\overline{\bp}, \bq) = \det_{1 \le i , j \le k} Z_{\Phi M}(\overline{p_i}, q_j).
		$$
		On the other hand, taking a determinant of \eqref{E:one-path-form} over indices $i_1 < \dots < i_k$ and $j_1 < \dots < j_k$, we have the following factorization:
		$$
		\left(\det_{1 \le \ell, \ell' \le k} \check Z_{(\Phi RM)^{-1}}((m + 1 - i_\ell, 1), (m, j_{\ell'})) \right) \prod_{\ell=1}^k Z_{\Phi RM}((j_\ell, j_\ell), (m, j_\ell)),
		$$
		which by Lemma \ref{L:Lindstrom-Gessel-Viennot} is equal to
		$$
		\check Z_{(\Phi RM)^{-1}}((m + 1 - \mathbf{i}, 1), (m, \mathbf{j})) \prod_{\ell=1}^k Z_{\Phi RM}((j_\ell, j_\ell), (m, j_\ell)),
		$$
		as desired.
	\end{proof}
	
	By formally exchanging the $(+, \times)$-algebra for the $(\max, +)$-algebra using Lemma \ref{lemma:temp_change}, we can deduce that the same results hold true for last passage percolation. 
	\begin{cor}
		\label{C:LPP-result}
		Fix $m, n \in \N$ and let  $M \in F(\II{1, m} \times \II{1, n}, \R)$. Let $\bp = (\mathbf{i}, n)$ and $\bq = (m, \mathbf{j})$ be two vectors in $(\Z^2)^k$ such that $1 \le i_1 < \dots < i_k \le m$ and $1 \le j_1 < \dots < j_k \le n$ and $i_\ell \ge j_\ell$ for all $\ell$. Then
		$$
		L_{\Psi M}(\overline{\bp}, \bq) = \check L_{-\Psi RM}((m + 1 - \mathbf{i}, 1), (m, \mathbf{j})) + \sum_{\ell=1}^k L_{\Psi RM}((j_\ell, j_\ell), (m, j_\ell)).
		$$ 
	\end{cor}

	\section{DtM Representation for the KPZ Sheet} \label{sec:KPZ_Sheet}

	Introduced by Kardar, Parisi and Zhang \cite{PhysRevLett.56.889}, the KPZ Equation corresponds to the following stochastic differential equation
	\begin{align} \label{E:KPZ_Eqn}
		\partial_t \mathcal{H} &= \frac{1}{2} \partial_x^2 \mathcal{H} + \frac{1}{2} (\partial_x \mathcal{H})^2 + \xi.  
	\end{align}
	Here the random function $\mathcal{H} : \R_{+} \times \R \to \R$ corresponds to a model for a randomly growing interface. The driving noise of the system is given by the spacetime white noise $\xi$.  The KPZ equation is often understood through a Cole-Hopf transformation of the multiplicative stochastic
	heat equation (SHE). That is, the solution to the KPZ equation is given by
	$\mathcal{H}(t, x) = \log \mathcal{Z}(t, x)$, where $\mathcal{Z}(t, x)$ satisfies the multiplicative SHE
	\begin{align} \label{E:StochasticHeatEqn}
		\partial_t \mathcal{Z} &= \frac{1}{2} \partial_x^2 \mathcal{Z} + \mathcal{Z} \xi.
	\end{align}
	While it is easy to check the formal relationship between \eqref{E:KPZ_Eqn} and \eqref{E:StochasticHeatEqn}, making this rigorous is
	extremely difficult and was finally understood in \cite{hairer2012solvingkpzequation}. 
	
	The solution to \eqref{E:StochasticHeatEqn} from arbitrary initial data can be defined from a four-parameter Green's function $\mathcal{Z}:\{(x, s; y, t) \in \R^4 : s \le t\} \to \R$, where each of the processes $(y, t) \mapsto \mathcal{Z}(x, s; y, t)$ solves the equation \eqref{E:StochasticHeatEqn} on the time interval $t \in [s, \infty)$ starting from the narrow-wedge initial data $\mathcal{Z}(x, s; \cdot, s) = \delta_x$. Here $\delta_x$ is a Dirac-$\delta$ mass at the point $x$. The field $\mathcal{Z}$ is a continuous positive function when we restrict to the set where $s < t$. Moreover, by linearity, for a (sufficiently nice) function $f:\R \to \R$ and $s \in \R$, we can solve \eqref{E:StochasticHeatEqn} from the initial condition $Z(\cdot) = f$ by
	\begin{equation}
		\label{E:ZintZ}
		\cZ(t, y) = \int \cZ(x, s; y, t) f(x) dx.
	\end{equation}
	See \cite{alberts2014continuum, alberts2022green} for basic properties of this field. 
	
	For $s < r < t$, the independence of $\xi$ implies that the increments $\cZ(\cdot, s; \cdot, r)$ and $\cZ(\cdot, r, \cdot, t)$ are independent, and by the Chapman-Kolmogorov equations, we can build the increment $\cZ(\cdot, s; \cdot, t)$ from $\cZ(\cdot, s; \cdot, r)$ and $\cZ(\cdot, r, \cdot, t)$. Therefore much of the information about the law of the entire process $\mathcal Z$ is contained in the law of the individual increments $\cZ(\cdot, s; \cdot, t)$. By time-stationarity, the law of this increment only depends on $t-s$. Hence, to understand the entire process $\cZ$, we can focus on understanding the family of functions
	$$
	\cZ^t(x, y) := \cZ(x, 0; y, t),
	$$
	known as the \textbf{SH sheet at time $t$}, and its logarithm $\mathcal H^t = \log \cZ^t$ is the \textbf{KPZ sheet}. This is the motivation for constructing a DtM representation in this setting.
	
	The Feynman-Kac formula suggests we may write 
	\begin{align} \label{E:Feynmann_Kacs}
		\mathcal{Z}^t(x ,y) &= p(t, y-x) \E_x^{y} :\exp  : \left( \int_{0}^{t} \xi(r, B_r) \, dr \right). 
	\end{align}
	where in the above $p(t, x) = \frac{1}{\sqrt{2 \pi t}} e^{-x^2/2t}$ is the Brownian Motion transition probability and $\E_x^y$ is the law of a Brownian bridge $(B_s)_{s = 0}^t$ starting at $x$ and ending at $y$. In other words, we may view $\mathcal{Z}^t(x, y)$ as the partition function of a Brownian polymer in the background white noise field $\xi$. Since $\xi$ is distribution-valued, the integral in \eqref{E:Feynmann_Kacs} is only formal, so to make this rigorous, one first starts with an approximation from the discrete polymer models of Section \ref{S:disclpp}. Alternately, the above notation can be viewed as shorthand for the Wiener chaos expansion  \begin{align}\label{E:Weiner_Chaos_single}
		p(t, y-x) \E_{x}^{y}  :\exp: \big( \int_0^t \xi(r, B_r) \, dr \big) \defeq \sum_{n = 0}^{\infty} \int_{0 = t_0 \leq \cdots \leq t_{n + 1} = t} \int_{\R^n}  \prod_{i = 1}^{n+1} p(t_{i}  - t_{i - 1} , z_{i} - z_{i - 1})  \prod_{i = 1}^{n} \xi(dt_i, dz_i),
	\end{align}
	where in the above $z_0 = x$,$z_{n + 1} = y$.

	The SH sheet has a multi-path extension, introduced in \cite{O_Connell_2015}. Let $\Lambda_{k} = \{ (x_1, \ldots, x_k) \in \R^k \, : \, x_1 < x_2 < \ldots < x_k\}$. For $\mathbf{x}, \mathbf{y} \in \Lambda_k$, let
	\begin{align}
		K^t(\x, \y) = p^{*}_k(t, \x, \y)  \mathbb{E}_{\mathbf{x}}^{\mathbf{y}} :\exp :\big( \sum_{i = 1}^k \int_0^t \xi(s, B_s^i) \, ds \big). \label{eqn:K_defn}
	\end{align}
	Here $p^{*}_k$ corresponds to the Karlin-McGregor formula for Brownian Motions killed on intersection,
	\begin{align*}
		p^*_k(t, \x, \y) = \det (p(t, x_i - y_j))_{i, j = 1}^k,
	\end{align*}
	and similarly to \eqref{E:Weiner_Chaos_single}, expression \eqref{eqn:K_defn} is shorthand for the following Wiener chaos expansion:
	\begin{align}\label{E:Weiner_Chaos}
		K^t(\mathbf{x},\mathbf{y}) = p_k^*(t,\mathbf{x},\mathbf{y}) \sum_{n = 0}^{\infty} \int_{0 = t_0 \leq \cdots  \leq t_{n + 1} = t} \int_{\R^n} R^{(k)}_n \big( (t_1, z_1) , \ldots  , (t_n, z_n) \big) \prod_{i = 1}^{n} \xi(dt_i, dz_i)
	\end{align}
	where $R^{(k)}_n$ is the n-point correlation function for $k$ non-intersecting Brownian bridges all starting at $\mathbf{x} = (x_1, \ldots, x_k)$ and ending at $\mathbf{y} = (y_1, \ldots, y_k)$, and $\E^{\mathbf{y}}_{\mathbf{x}}$ is the law of such non-intersecting Brownian bridges. We will not use the above definition directly, but rather a consequence, \cite[Theorem 5.3]{O_Connell_2015}, which can be viewed as a Lindstr\"om-Gessel-Viennot or Karlin-McGregor theorem in this setting:
	\begin{align} \label{E:Karlin_Mcgregor_SHE}
		K^t(\x, \y) = \det \big(  K^t(x_i, y_j )\big)_{i, j = 1}^k
	\end{align}
	Continuity of $\cZ^t$ implies that $K^t$ vanishes along the boundary of $\Lambda_k$. However, one can re-normalize $K^t$ in the following way so we retain non-trivial information at the boundary. Let $\Delta(\bx) = \prod_{1 \le i < j \le k} (x_j - x_i)$ denote the Vandermonde determinant, and define the \textbf{extended SH sheet}
	\begin{align}
		M^t(\x, \y) = \frac{K^t(\mathbf{x} , \mathbf{y})}{\Delta(\mathbf{x}) \Delta (\mathbf{y})}.
	\end{align}

	\begin{lemma}[Theorem 1.3 of \cite{lun2020}] \label{L:M-cty}
		Almost surely, $M^t$ has a continuous, strictly positive extension to $\overline{\Lambda}_n \times \overline{\Lambda}_n$. 
	\end{lemma}
	
	Now, as mentioned previously, one approach to constructing all of the objects above is by taking scaling limits of directed polymers on $\Z^2$. There is also a scaling limit of the output of the geometric RSK correspondence. The two tableaux that form the output of geometric RSK become a single \emph{line ensemble} in the scaling limit.

	For $x \in \R$, let $x^n = (x, \dots, x) \in \R^n$. We simultaneously define the SHE$^{(t)}$ line ensemble, $\{ \mathcal{Z}^{(t)}_n(x) \}_{n \in \N , \, x \in \R}$ and the KPZ$^{(t)}$ line ensemble, $ \{ \mathcal{H}^{(t)}_n(x) \}_{n \in \N , \, x \in \R}$  by setting
	\begin{align*}
		\mathcal{H}^{(t)}_n(x)  = \log \left( \left( (n - 1)! \right)^2 \frac{M^t(0^n, x^n)}{M^t(0^{n-1}, x^{n-1})} \right) = \log \left( \mathcal{Z}^{(t)}_{n}(x)  \right).
	\end{align*}
	Here and throughout, we adopt the convention that $M^t(0^0, x^0) = 1$.
	Moving forward, we will suppress the superscripts $^t, ^{(t)}$ when they are clear from context. Note that the KPZ$^{(t)}$ line ensemble is often defined as the scaling limit of a line ensemble (tableaux pair) associated to a discrete or semi-discrete polymer. It was proven in \cite{Nica_2021} that the definition we give here is equivalent. 
	
	Note that in many applications, it is easiest to work with the KPZ line ensemble since this ensemble satisfies a natural Gibbs resampling property. However, in what follows we will focus on the SH line ensemble, which is more natural from the point of view of the DtM representation. Of course, all theorems about the SH equation can be transferred to theorems about the KPZ equation by taking logarithms.
	
	\subsection{The DtM representation and a proof strategy}
	\label{sec:DTM-KPZ}

	We can now precisely state the DtM representation for the SH line ensemble. For this, we need to define semi-discrete polymers analogously to how we defined semi-discrete LPP in the introduction. Let $g = (g_i:\R \to (0, \infty), i \in \Z)$ be a sequence of continuous functions. 
	For $x \le y$ and $m \ge n$ define
	$$
	Q[(x, m) \to (y, n)] = \{\pi:[x, y] \to \{n, \dots, m\}, \pi \text{ cadlag, nonincreasing}\},
	$$
	which is the set of all up-right paths from $(x, m)$ to $(y, n)$. We can encode a path $\pi$ by jump times $x = t_m \le \cdots \le t_{n-1} = y$, where $t_i := \inf \{z \in [x, y] : \pi(z) \le i\}$, and we can define its weight with respect to the environment $g$ by
	$$
	g(\pi) = \prod_{i = n}^m \left( g_i(t_{i - 1})  / g_i(t_i) \right). 
	$$
	Next, define the \textbf{partition function} from $(x, m)$ to $(y, n)$ by
	$$
	g\{(x, m) \to (y, n)\} = \int_{Q[(x, m) \to (y, n)]}   g(\pi) d \mu(\pi)
	$$
	Here the measure $\mu$ on $Q[(x, m) \to (y, n)]$ is simply Lebesgue measure on the set of potential jump times in $\R^{m - n}$ of paths $\pi$.
	We also define multi-path partition functions. Suppose that $\bp = ((x_1, m_1), \dots, (x_k, m_k)), \bq = ((y_1, n_1), \dots, (y_k, n_k)) \in (\R \times \Z)^k$ satisfy the ordering constraints $x_i \le y_i, m_i \ge n_i$, and $x_1 < \cdots < x_k$, $y_1 < \cdots < y_k$ and $m_1 < \cdots < m_k$, $n_1 < \cdots < n_k$. We also assume that there is at least one $k$-tuple of paths $\pi = (\pi_1, \dots, \pi_k)$ with $\pi_i \in Q[(x_i, m_i) \to (y_i, n_i)]$ such that $\pi_i < \pi_{i+1}$ on $(x_i, y_i) \cap (x_{i+1}, y_{i+1})$ for all $i = 1, \dots, k-1$, and let $Q[\bp \to \bq]$ be the set of such $k$-tuples. Then set
	$$
	g\{\bp \to \bq\} = \int_{Q[\bp \to \bq]} g(\pi_1) \cdots g(\pi_k) d \mu(\pi),
	$$
	where again, $\mu$ is Lebesgue measure on jump times.

	\begin{thm} \label{thm:KPZ_Case}
		Let $M^t$ be the extended SH sheet, and define
		\begin{align*}
			\mathcal{Z}_n^{(t)}(x) &=  \begin{cases}
				\left( (n - 1)! \right)^2 \frac{M^t (0^n, x^n)}{M^t (0^{n-1} , x^{n-1} )  }, \qquad &x \ge 0, \\
				\left( (n - 1)! \right)^2 \frac{M^t (|x|^n, 0^n)}{M^t (|x|^{n-1}, 0^{n-1})  }, \qquad &x < 0,  
			\end{cases}  
		\end{align*}
		Then $\{\cZ^{(t)}_n : n \in \N\}$ is an SH line ensemble, and for every $(\bx, \by) \in \big( \bigcup_{n \geq 1} \R^n_{<} \times \R^n_{<} \big)$ with $x_1, y_1 \geq 0$ 
		\begin{align}
			\label{E:SHZZZ}
			K^t(\mathbf{x}, \mathbf{y}) &=  \sum_{I = \{ i_1 , \ldots , i_n \} \in  \N^n_{<}}     \Big( \prod_{k = 1}^n \mathcal{Z}^{(t)}_{i_k}(0) \Big) \cdot \tilde{\mathcal{Z}}^{(t)} \{  (0, I) \to (\mathbf{x}, 1)  \}  \cdot \mathcal{Z}^{(t)} \{ (0, I) \to (\mathbf{y}, 1) \} ,
		\end{align}
		where $\tilde{\mathcal{Z}}^{(t)}_n(x) = \mathcal{Z}_{n}^{(t)}(-x)$ and $\mathcal{Z}^{(t)} \{ (\cdots)  \to  ( \cdots )  \}$ denotes the partition function across the SH line ensemble. 
	\end{thm}
	
	Similarly to the proof of Theorem \ref{thm:RSK_Representation_discrete}, the multi-point case follows from the single-point result due to the Lindstr\"om-Gessel-Viennot theorem. The proof essentially follows verbatim the proof of Lemma \ref{lemma:RSK_discrete_multipt}. The one change is that in order to apply the LGV theorem (which holds for graphs), we need to first discretize the environment and then take a limit. This is straightforward and we leave it as an exercise for the reader.
	
	\begin{remark}
		\label{rem:SHZZZ}
		Note that we have stated the multi-point version for $K^t$, rather than $M^t$. Of course, a similar representation holds for $M^t$ after dividing by $\Delta(\bx) \Delta(\by)$ on both sides. It is not difficult to check that after this division-by-Vandermonde, each term on the right-hand side of \eqref{E:SHZZZ} has a continuous extension to all of $\R^n_{\le} \times \R^n_{\le}$, which can be viewed as a partition function with repeated endpoints.
	\end{remark}
	
	The rest of this section is devoted towards the proof of Theorem \ref{thm:KPZ_Case_simply}, which we break up into three steps. The first step is a continuous analogue of Theorem \ref{thm:RSK_Representation_discrete}. Namely, by spacing out the starting and ending points to be a distance $\epsilon > 0$ apart, we may apply the same argument used in the proof of Theorem \ref{thm:RSK_Representation_discrete} involving the Desnanot-Jacobi Identity. Taking $\epsilon  \to 0 $ then gives the following lemma. 
	
	\begin{lemma} \label{lemma:KPZSheet_rep1} Almost surely for all $n \in \N$ and $x, y \geq 0$
		\begin{align*}
			M^t(x, y)
			&= \sum_{k = 1}^{n}  x^{k - 1} y^{k - 1}   \frac{M^t(0^k, 0^k)}{M^t (0^{k - 1}, 0^{k - 1})}  \frac{M^t( \{ 0^{k - 1}, x\}, 0^k)}{M^t(0^{k } , 0^{k })}  \frac{M^t( 0^k, \{ 0^{k - 1} , y\})}{ M^t(0^k, 0^k)} + R^{(n)}(x,y).
		\end{align*}
		where 
		\begin{align}
			R^{(n)}(x,y) = x^n y^n \frac{M(\{ 0^n, x \} , \{0^n, y \})}{M(0^n, 0^n)}.
		\end{align}
	\end{lemma}
	
	\noindent In the next step of the proof, we rewrite the ratios appearing in the above summation in terms of partition functions through the SH line ensemble. 
	
	\begin{lemma} \label{lemma:LPP_values_KPZLE}
		Let $M^t$ be an extended SH Sheet, and define
		\begin{align*}
			\mathcal{Z}_n^{(t)}(x) &=  \begin{cases}
				\left( (n - 1)! \right)^2 \frac{M^t (0^n, x^n)}{M^t (0^{n-1} , x^{n-1} )  }, \qquad x \ge 0, \\
				\left( (n - 1)! \right)^2 \frac{M^t (|x|^n, 0^n)}{M^t (|x|^{n-1}, 0^{n-1})  }, \qquad x < 0,  
			\end{cases}  
		\end{align*}
		Then for all $k \in \N$ and $x \geq 0$,
		\begin{align*}
			\tilde{\mathcal{Z}}^{(t)} \{ (0, k) \to (x, 1) \}  &= \frac{x^{k - 1}}{(k - 1)!} \frac{M^t(\{ 0^{k - 1}, x\} , 0^k)}{M^t(0^k, 0^k)} ,\\
			\mathcal{Z}^{(t)}  \{ (0, k) \to (x, 1) \} &= \frac{x^{k - 1}}{(k - 1)!} \frac{M^t(0^k, \{ 0^{k - 1}, x\})}{M^t(0^k, 0^k)}.
		\end{align*}
	\end{lemma}
	
	In proving Lemma \ref{lemma:LPP_values_KPZLE}, the only properties of the SH Sheet used are the Karlin-McGregor formula in \eqref{E:Karlin_Mcgregor_SHE} and continuity of $M^t$ on $\R^n_{\leq} \times \R^n_{\leq}$. Consequently, Lemma \ref{lemma:LPP_values_KPZLE} (together with Lemma \ref{lemma:LPP_SHE_Sheet}) extends \cite[Theorem 3.4]{O_Connell_2015}, which establishes the corresponding result only in the case of smooth background noise.

	The final step of the proof is to show that the remainder term vanishes as one takes $n \to \infty$. 
	
	\begin{lemma} \label{lemma:KPZ_remainderterm} Almost surely, $R^{(n)} \to 0$, uniformly on compact subsets of $[0, \infty)^2$.
	\end{lemma}
	
	This is the key technical step of this section and is handled in Section \ref{subsec:KPZ_remainderterm}. The key idea is to use the relationship between the SH line ensemble and SH sheet to bound the remainder term by the height of the $n$th line in the KPZ line ensemble at the origin, $\mathcal{H}_n(0)$. The necessary estimate on $\mathcal{H}_n(0)$ is then supplied by \cite{Syed_2025}.
	
	\begin{proof}[Proof of Theorem \ref{thm:KPZ_Case_simply}]
		The fact that $\cZ^{(t)}$, as defined in Theorem \ref{thm:KPZ_Case}, is an SH line ensemble follows from results of \cite{dauvergne2022hidden}. We give a sketch here.
		
		First, working in the fully discrete setting of Section \ref{sec:discrete}, let $M:\Z^2 \to \R$ be an environment of i.i.d.\ inverse-gamma random variables. Fix $n, m \in \N$, let $o = (1, n)$ and $\hat o = (m, 1)$, and define the two processes
		\begin{align*}
			L(k, i) &= Z_M((o; i, 1)^{k \wedge i}), \qquad \qquad \qquad k \in  \II{1, n}, i \in \N, \\
			\hat L(k, i) &= \begin{cases}
				Z_M((m + 1 - i, n; \hat o)^{k \wedge i}), \quad  &k\in  \II{1, n}, i \in \II{1, m}, \\ 
				Z_M((o; i, 1)^{k \wedge i}), \qquad &k \in  \II{1, n}, i \ge m + 1.
			\end{cases} 
		\end{align*}
		By \cite[Theorem 4.2 and Proposition 4.13(1)]{dauvergne2022hidden}, the processes $L$ and $\hat L$ are equal in law, jointly in all $i, k$. 
		
		Next, we can take two limit transitions. First, in the `long box' limit, where we fix $n$ and take $m \to \infty$, and diffusively scale the i.i.d.\ environment so that we obtain partition functions across Brownian motions, we recover a similar identity for the O'Connell-Yor polymer. Finally, we take the intermediate-disorder scaling limit of the O'Connell-Yor polymer where $n$ and $m$ tend to infinity, as in  \cite[Theorem 1.2, Corollary 1.7, and Sections 2.4--2.6]{Nica_2021}. The identity above then finally scales to the following. 
		Write $x_-=\max\{-x,0\}$ and $x_+=\max\{x,0\}$. Then
		\begin{equation}\label{E:SH_tower_symmetry}
			\bigl(M^t(x_-^k,x_+^k):x\in\R,\ k\geq1\bigr)
			\stackrel{d}{=}
			\bigl(M^t(0^k,x^k):x\in\R,\ k\geq1\bigr).
		\end{equation}
		This implies that $\cZ^{(t)}$ is an SH line ensemble.
		
		For the identity \eqref{E:SHZZZ}, substitute the expressions for the ratio of partition functions in Lemma \ref{lemma:LPP_values_KPZLE} into the summation in Lemma \ref{lemma:KPZSheet_rep1}. The theorem then follows by taking $n \to \infty$ and applying Lemma \ref{lemma:KPZ_remainderterm}.
	\end{proof}

	\subsection{Proof of Lemma \ref{lemma:KPZSheet_rep1}} \label{subsec:lemma_KPZSheet_rep1}
	
	We start with a version of Lemma \ref{lemma:KPZSheet_rep1} with the endpoints spaced apart.
	
	\begin{lemma} \label{lemma:SHE_finitesum_K}
		Let $n \in \N$. For any $x, y > 0$ and $0 < \epsilon < \frac{1}{n} \min (x, y)$, the following identity holds:
		\begin{align} \label{E:induction_step}
			M^t(x, y) = \sum_{k = 1}^{n}  \frac{K^t(  \{ \epsilon_{k - 1}, x  \} , \epsilon_k) K^t(\epsilon_k, \{ \epsilon_{k - 1}, y \})}{K^t(\epsilon_{k - 1}, \epsilon_{k - 1}) K^t (\epsilon_k, \epsilon_k)} + \frac{K^t(\{ \epsilon_n, x\} , \{ \epsilon_n, y\})}{K^t(\epsilon_n, \epsilon_{n})},
		\end{align}
		where $\epsilon_k = (0, \epsilon, \ldots , (k - 1) \epsilon)$.
	\end{lemma}
	
	\begin{proof}
		We apply an induction argument on $n$. First, let us prove the base case. By the Karlin-McGregor formula \eqref{E:Karlin_Mcgregor_SHE}, for all $x, y > 0$:
		\begin{align*}
			K^t(\{ 0, x\}, \{ 0, y \}) &= K^t(0, 0) K^t(x, y) - K^t(0, y) K^t(x, 0)
		\end{align*}
		For $n = 1$, we have that $K^t = M^t$. Re-arranging gives us:
		\begin{align*}
			M^{0,t}(x, y) &=  \frac{K^{0,t}(x, 0) K^{0,t}(0, y)}{K^{0,t}(0, 0)}  + \frac{K^{0,t}(\{ 0, x\} , \{ 0, y\})}{K^{0,t}(0, 0)} 
			=  \frac{K^{0,t}(x, \epsilon_1) K^{0,t}(\epsilon_1, y)}{K^{0,t}(\epsilon_1, \epsilon_1)}  + \frac{K^{0,t}(\{ \epsilon_1, x\} , \{ \epsilon_1, y\})}{K^{0,t}(\epsilon_1, \epsilon_1)}.
		\end{align*}
		Suppose now we have proven the induction step \eqref{E:induction_step} for $n$, and consider $0 < \epsilon < \frac{\min (x, y)}{n + 1}$. Define an $(n + 2) \times (n + 2)$ matrix
		\begin{align*}
			A &=	\begin{pmatrix}
				K(0, 0) & \cdots & K(0, (n - 1)\epsilon) & K(0, n\epsilon) & K(0, y) \\
				\vdots & \ddots & \vdots & \vdots & \vdots \\
				K((n - 1)\epsilon, 0) & \cdots & K((n - 1) \epsilon, (n - 1) \epsilon) & K((n - 1)\epsilon, n \epsilon) & K((n - 1)\epsilon, y) \\
				K(n \epsilon, 0) & \cdots & K(n \epsilon, (n - 1) \epsilon) & K(n \epsilon, n \epsilon) & K(n \epsilon, y) \\
				K(x, 0) & \cdots & K(x, (n - 1) \epsilon) & K(x, n \epsilon) & K(x, y) \\
			\end{pmatrix} 
		\end{align*}
		By the Karlin-McGregor formula \eqref{E:Karlin_Mcgregor_SHE} we have that
		\begin{align*}
			K(\{ \epsilon_{n + 1}, x\} , \{ \epsilon_{n + 1}, y\}) &= \det(A) , \qquad K(\epsilon_n, \epsilon_n)  = \det(A^{(n+1, n+2)}_{(n+1, n + 2)})
		\end{align*}
		where $A^{(n+1, n+2)}_{(n+1, n + 2)}$ is the $n \times n$ submatrix obtained from $A$ by removing the last two rows and columns. By the Desnanot-Jacobi identity,
		\begin{align} \label{E:det_equality}
			K(\epsilon_n, \epsilon_n) K(\{ \epsilon_{n + 1}, x\} , \{ \epsilon_{n + 1}, y\}) &= \det(A^{(n+1, n+2)}_{(n+1, n + 2)})  \det(A) 
			= \det \begin{pmatrix}
				\det (A^{(n + 1)}_{(n+ 1)}) & \det (A^{(n + 2)}_{(n+ 1)}) \\ \det (A^{(n + 1)}_{(n+ 2)}) & \det (A^{(n + 2)}_{(n+ 2)})
			\end{pmatrix}.
		\end{align}
		Applying the Karlin-McGregor formula again, we obtain the following set of identities:
		\begin{align*}
			K^{0, t} (\{ \epsilon_n, x\} , \{ \epsilon_n, y \}) &= \det ( A^{(n+1)}_{(n+1)}) \\
			K^{0, t} (\{ \epsilon_n, x\} ,  \epsilon_{n + 1}) &=  \det(A^{(n+1)}_{(n+2)} ) \\ 
			K^{0, t} ( \epsilon_{n+1} ,  \{ \epsilon_{n}, y\}) &=  \det( A^{(n+2)}_{(n+1)} ) \\
			K^{0, t} (\epsilon_{n + 1}, \epsilon_{n + 1}) &= \det( A^{(n + 2)}_{(n + 2)} )
		\end{align*}
		Substituting these identities into \eqref{E:det_equality} and re-arranging proves that 
		\begin{align} \label{E:remainder_term}
			K^{0,t}(\{ \epsilon_n, x\}, \{ \epsilon_n, y\}) &= \frac{K^{0,t}(\epsilon_n, \epsilon_n) K^{0,t}( \{\epsilon_{n + 1}, x \} , \{ \epsilon_{n + 1}, y\})}{K^{0,t}(\epsilon_{n + 1}, \epsilon_{n + 1})} + \frac{K^{0,t}(\{ \epsilon_n, x\}, \epsilon_{n + 1}) K^{0,t}(\epsilon_{n + 1}, \{ \epsilon_n, y\})}{K^{0,t}(\epsilon_{n + 1}, \epsilon_{n + 1})}
		\end{align}
		Plugging \eqref{E:remainder_term} into the remainder term of \eqref{E:induction_step} concludes the induction step. 
	\end{proof}
	
	\begin{proof}[Proof of Lemma \ref{lemma:KPZSheet_rep1}]
		We now take $\epsilon \to 0$ in each of the terms in Lemma \ref{lemma:SHE_finitesum_K}. First on the remainder term,
		\begin{align*}
			\frac{K^t(  \{ \epsilon_{n} , x \}, \{ \epsilon_{n},  y\})}{K^t(\epsilon_{n}  , \epsilon_{n})} &= \frac{K^t(  \{ \epsilon_{n} , x \}, \{ \epsilon_{n},  y\}) / \Delta(\epsilon_{n})^2}{K^t(\epsilon_{n}  , \epsilon_{n}) / \Delta(\epsilon_{n})^2} \frac{\prod_{j = 0}^{n-1} (x - j \epsilon) (y - j \epsilon)}{\prod_{j = 0}^{n-1} (x - j \epsilon) (y - j \epsilon)} \\
			&= \frac{K^t(  \{ \epsilon_{n} , x \}, \{ \epsilon_{n},  y\}) / (\Delta(\epsilon_{n} , x) \Delta(\epsilon_{n} , y)  )}{K^t(\epsilon_{n}  , \epsilon_{n}) / \Delta(\epsilon_{n})^2}   \prod_{j = 0}^{n-1} (x - j \epsilon) (y - j \epsilon) \\
			& \xto{\epsilon \to 0} x^{n} y^{n} \frac{M^t( \{ 0^{n}, x \}  , \{ 0^{n} , y\})  }{M^t(0^{n} , 0^{n})}
		\end{align*}
		Now for the terms in the sum,
		\begin{align*}
			\frac{K^t (\{ \epsilon_{k - 1} , x\}  , \epsilon_k)  K^t(\epsilon_{k} , \{ \epsilon_{k - 1}, y\})}{K^t(\epsilon_{k - 1} , \epsilon_{k - 1}) K^t(\epsilon_k, \epsilon_k)} &= \frac{K^t (\{ \epsilon_{k - 1} , x\}  , \epsilon_k)  K^t(\epsilon_{k} , \{ \epsilon_{k - 1}, y\})}{K^t(\epsilon_{k - 1} , \epsilon_{k - 1}) K^t(\epsilon_k, \epsilon_k)}  \cdot  \frac{\Delta(\epsilon_k)^2}{\Delta(\epsilon_k)^2} \frac{\Delta(\epsilon_{k - 1})^2}{\Delta(\epsilon_{k - 1})^2} \frac{\prod_{j = 0}^{k - 2} (x - j \epsilon) (y - j \epsilon)}{\prod_{j = 0}^{k - 2} (x - j \epsilon) (y - j \epsilon)} \\
			&= \frac{M^t(\{ \epsilon_{k - 1}, x\} , \epsilon_k) M^t(\epsilon_k, \{ \epsilon_{k - 1}, y\})}{M^t(\epsilon_{k - 1} , \epsilon_{k - 1})  M^t(\epsilon_k , \epsilon_k)} \cdot \prod_{j = 0}^{k - 2} (x - j \epsilon) (y - j \epsilon) \\
			& \xto{\epsilon \to 0} x^{k - 1} y^{k - 1} \frac{M^t( \{ 0^{k - 1}, x \} , 0^k) M^t(0^k, \{ 0^{k - 1} , y\})}{M^t(0^{k - 1} , 0^{k - 1}) M^t(0^k, 0^k)}
		\end{align*}
		Therefore, we conclude for all $n \in \N$ and $x, y > 0$
		\begin{align*}
			M^t(x, y) &= \sum_{k = 1}^n  x^{k - 1} y^{k - 1} \frac{M^t( \{ 0^{k - 1}, x \} , 0^k) M^t(0^k, \{ 0^{k - 1} , y\})}{M^t(0^{k - 1} , 0^{k - 1}) M^t(0^k, 0^k)} + x^{n} y^{n} \frac{M^t( \{ 0^{n}, x \}  , \{ 0^{n } , y\})  }{M^t(0^{n} , 0^{n})} 
		\end{align*}
		Since both sides above are continuous for all $x, y \ge 0$, the identity extends to all $x, y \ge 0$, as desired.
	\end{proof}

	\subsection{Partition Functions Across the SH line ensemble and the proof of Lemma \ref{lemma:LPP_SHE_Sheet}} \label{subsec:lemma_LPP_values_KPZLE}

	Note that since the spatial fluctuations of the KPZ equation are locally Brownian, almost surely $M^t$ is nowhere partially differentiable. Despite this, surprisingly, certain ratios of $M$ are differentiable. The next lemma identifies the derivative of one such ratio.
	
	\begin{lemma} Fix $t > 0$, and for $\eta \geq 0$ and $x \in \R^n_{<}$, define the function $f^{\eta, x} : \R^n_{<} \to \R$ by:
		\begin{align*}
			f^{\eta, x}(z) = \frac{M \big( x + \eta e_n  ,  z  \big)}{M \big(  x   , z    \big)}
		\end{align*}
		where $e_i$ is the $i$th standard unit vector in $\R^n$. Then $f^{\eta, x}(z)$ is a continuously differentiable function of $z$ whose partial derivatives are given by:
		\begin{align*}
			\partial_i f^{\eta, x} (z) &= \eta \frac{M(\hat{x}_n, \hat{z}_i)  M\big( (x_1 , \ldots , x_n , x_n + \eta), (z_1, \ldots, z_{i}, z_{i}, z_{i + 1} \ldots , z_n ) \big)    }{M(x, z)^2} 
		\end{align*}
		where $\hat{x}_n = (x_1 , \ldots , x_{n - 1})$ and $\hat{z}_i = (z_1 , \ldots , z_{i - 1}, z_{i + 1} , \ldots , z_{n})$.
	\end{lemma}
	
	\begin{remark}
		The analogous statement for the Airy sheet $\cS$ is that the difference profile $y \mapsto \cS(z, y) - \cS(x, y)$ is monotone, and constant off of a measure $0$ set, see \cite{basu2019fractal}. In fact, this process is locally absolutely continuous with respect to the running maximum of a  Brownian motion, see \cite{ganguly2021local, dauvergne2021isometries}.
	\end{remark}
	
	\begin{proof} The lemma follows from a simple application of the Karlin-McGregor formula and the Desnanot-Jacobi identity. We have that
		\begin{align*}
			\partial_i f^{\eta,x}(z)
			&= \lim_{\delta \to 0} \frac{1}{\delta} \Big( \frac{M \big( x + \eta e_n  ,  z + \delta e_i \big)}{M \big( x  , z+ \delta e_i    \big)}  - \frac{M \big( x + \eta e_n  ,  z  \big)}{M \big( x  , z    \big)}
			\Big)  \\
			&= \lim_{\delta \to 0} \frac{1}{\delta} \frac{M(x , z ) M(x + \eta e_n , z + \delta e_i) - M(x + \eta e_n , z) M(x, z + \delta e_i) }{M(x , z) M(x, z + \delta e_i)}  \\
			&= \lim_{\delta \to 0} \frac{1}{\delta}  \frac{1}{\Delta (x) \Delta(z) \Delta (x +  \eta e_n) \Delta(z + \delta e_i)}\frac{K(x , z ) K(x + \eta e_n , z + \delta e_i) - K(x + \eta e_n , z) K(x, z + \delta e_i) }{M(x , z) M(x, z + \delta e_i) } 
		\end{align*}
		Now consider the following $(n + 1) \times (n + 1)$ matrix:
		{\small\begin{align*}
				A &= \begin{pmatrix}
					K(x_1, z_1) & \cdots & K( x_1, z_i) & K (x_1, z_{i} + \delta) & K(x_1, z_{i + 1}) & \cdots & K (x_1, z_n) \\
					K(x_2, z_1) & \cdots & K( x_2, z_i) & K (x_2, z_{i} + \delta) & K(x_2, z_{i + 1}) & \cdots & K (x_2, z_n) \\
					\vdots & \cdots & \vdots & \vdots & \vdots & \cdots & \vdots \\
					K(x_n, z_1) & \cdots & K( x_n, z_i) & K (x_n , z_{i} + \delta) & K(x_n, z_{i + 1}) & \cdots & K (x_n, z_n) \\
					K(x_n + \eta , z_1) & \cdots & K( x_n + \eta, z_i) & K (x_n + \eta , z_{i} + \delta) & K(x_n + \eta , z_{i + 1}) & \cdots & K (x_n + \eta, z_n)
				\end{pmatrix}
		\end{align*}}

		Notice that the numerator in the limit above can be written in terms of the minor determinants of $A$ in the following way. 
		\begin{align*}
			K(x , z ) K(x + \eta e_n , z + \delta e_i) - K(x + \eta e_n , z) K(x, z + \delta e_i)  &= \det( A^{n+1}_{i+1}) \cdot \det(  A^n_i)  - \det( A^n_{i + 1} ) \cdot \det  (A^{n+1}_i) \\
			&= \det \begin{pmatrix} \det (A_i^n) & \det (A_i^{n + 1})  \\ \det (A_{i + 1}^n ) & \det (A_{i + 1}^{n+1}) \end{pmatrix} \\
			&= \det(A) \det (A^{(n, n + 1)}_{(i , i+ 1)})
		\end{align*}
		where in the last line we used the Desnanot-Jacobi identity. Hence, letting $\hat{z}_i = (z_1 , \ldots, z_{ i - 1} , z_{i + 1} , \ldots, z_n)$ and $\hat{x}_n = (x_1 , \ldots, x_{n - 1})$ we have:
		\begin{align*}
			\partial_i f^{\eta,x}(z) 
			&= \lim_{\delta \to 0} \frac{1}{\delta}  \frac{ K\big( \hat{x}_n, \hat{z}_i \big)    K\big( (x_1 , \ldots, x_n, x_n + \eta), (z_1, \ldots, z_{i}, z_{i} + \delta, z_{i + 1} \ldots , z_n ) \big)  }{\Delta (x) \Delta(z) \Delta (x +  \eta e_n) \Delta(z + \delta e_i) M(x , z) M(x, z + \delta e_i)}   \\
			&=  \frac{M(\hat{x}_n, \hat{z}_i)  M\big( (x_1 , \ldots , x_n , x_n + \eta), (z_1, \ldots, z_{i}, z_{i}, z_{i + 1} \ldots , z_n ) \big)    }{M(x, z)^2} \\
			& \quad \cdots \times  \frac{\Delta (\hat{x}_{n })  \Delta (x_1 , \ldots , x_n , x_n + \eta  )}{\Delta (x)  \Delta (x +  \eta e_n)  } \times \lim_{\delta \to 0 } \frac{1}{\delta} \frac{ \Delta (\hat{z}_i) \Delta (z_1 , \ldots, z_i, z_i + \delta, z_{i + 1}, \ldots, z_n)}{\Delta(z)\Delta(z + \delta e_i)  }.
		\end{align*}
		It can be easily verified that the above limit in $\delta$ is exactly equal to one, and that
		\begin{align*}
			\frac{\Delta (\hat{x}_{n })  \Delta (x_1 , \ldots , x_n , x_n + \eta  )}{\Delta (x)  \Delta (x +  \eta e_n)  } = \eta.
		\end{align*}
		This gives the desired derivative identity. The fact that $\partial_i f^{\eta, x}$ is continuous follows from Lemma \ref{L:M-cty}.
	\end{proof}
	
	\noindent Now setting $x=\epsilon_n$ and taking the directional derivative in the direction $(1, \ldots, 1)$ yields the following corollary. 
	
	\begin{cor} \label{cor:g_defn} Fix $\epsilon>0$. For $\eta \geq 0$, define the functions $g_n^{\eta}: \R_{+} \to \R_{+}$ by:
		\begin{align*}
			g_n^{\eta}(z) &= \frac{M ( \epsilon_n + \eta e_n, z^n) }{M (\epsilon_n,  z^n)}.
		\end{align*}
		where $\epsilon_n = (0, \epsilon, \ldots, (n - 1) \epsilon)$. In particular, $g_n^0\equiv1$. Then $g^\eta_n$ is differentiable on $(0, \infty)$ and
		\begin{align*}
			(g_n^{\eta})'(z) =  n \eta \frac{ M(\epsilon_{n - 1} , z^{n - 1}) M(\epsilon_{n + 1} + (\eta - \epsilon) e_{n + 1} ,z^{n + 1}  ) }{M ( \epsilon_n  , z^n)^2}.
		\end{align*}
	\end{cor}
	\noindent In particular, for any $\ell \in \N$, we have:
	\begin{align*}
		\frac{(g^{\ell \epsilon}_n  )' }{(g^{\epsilon}_n )'} &=  \ell g^{(\ell - 1) \epsilon}_{n + 1}.
	\end{align*}
	The next lemma is proven in \cite{O_Connell_2015} using the Cauchy-Binet formula.
	\begin{lemma}[Lemma 3.5 of \cite{O_Connell_2015}] \label{lemma:cauchy-binet-identity} Suppose $y \in \R^n_<$.
		If $f_1, f_2, \ldots $ is a sequence of continuously differentiable functions with $f_1 \equiv 1$, then,
		\begin{align*}
			\det [f_i(y_j)]_{i,j = 1}^n = \int_{z \prec y} \det [f_{i + 1}' (z_j)]_{i,j = 1}^{n - 1} dz_1 \cdots dz_{n-1}.
		\end{align*}
		Here the integral is over the set of all $z \in \R^{n-1}_<$ which interlace with $y$: $y_1 < z_1 < y_2 <\cdots < y_{n-1} < z_{n-1} < y_n$. Moving forward we use the interlacing notation $z \prec y$.
	\end{lemma}
	
	We are now ready to prove the following analogue of \cite[Theorem 3.4]{O_Connell_2015} for the SH case. 
	
	\begin{lemma} \label{lemma:LPP_SHE_Sheet} For $y \in \R^{n}_{<}$ with $y_1\geq0$,
		\begin{align*}
			M^t (0^n , y) 
			&= (\prod_{j = 0}^{n - 1}j! )\frac{M^t(0^n , 0^n) }{\Delta (y) } Z  \{ (0, (1, \dots, n)) \to (y, 1) \}.
		\end{align*}
	\end{lemma}
	
	\begin{remark} This formula actually disagrees with Theorem 3.4 of \cite{O_Connell_2015}, which has a normalization constant of $(\prod_{j = 0}^{n - 1} j! )^2$ instead. This discrepancy is due to a typo in equation (39) of \cite{O_Connell_2015}, where for smooth background noise, the authors claim that $M_{n}(t, x^n, \mathbf{y}) = \Delta (\mathbf{y})^{-1} \det [\partial_x^{i - 1} u(t, x , y_j)]_{i, j = 1}^{n}$. However, the normalization is slightly incorrect, and instead should be given by
		$$M_{n}(t, x^n, \mathbf{y}) = (\prod_{j = 0}^{n - 1} j! )^{-1} \Delta (\mathbf{y})^{-1} \det [\partial_x^{i - 1} u(t, x , y_j)]_{i, j = 1}^{n}$$
		Using this normalization makes the result of \cite{O_Connell_2015} consistent with Lemma \ref{lemma:LPP_SHE_Sheet}. 
	\end{remark}
	
	\begin{proof} First, for any $y \in \R^n_{<}$, we have by definition:
		\begin{align*}
			M^t \big( \epsilon_n, y \big) &= \frac{K (\epsilon_n, y)}{\Delta (\epsilon_n) \Delta(y)} \\
			&= \frac{1}{\Delta (\epsilon_n) \Delta(y)} \det [K(\epsilon (i-1), y_j)]_{i, j = 1}^n \\
			&= \frac{\prod_{j = 1}^n K(0, y_j) }{\Delta (\epsilon_n) \Delta(y)}
			\det \left[g^{\epsilon(i-1)}_1(y_j)\right]_{i, j = 1}^n 
		\end{align*}
		where in the last step we substituted in the definition of $g$ from Corollary \ref{cor:g_defn}, and used the fact that $M^t = K$ for $n = 1$. Now by Lemma \ref{lemma:cauchy-binet-identity},
		\begin{align*}
			M^t(\epsilon_n, y) &= \frac{\prod_{j = 1}^n K(0, y_j) }{\Delta (\epsilon_n) \Delta(y)}  \int_{z \prec y} \det \left[(g^{\epsilon i}_1)'(z_j)\right]_{i, j = 1}^{n-1}  dz_1 \cdots dz_{n-1} \\
			&=  \frac{\prod_{j = 1}^n K(0, y_j) }{\Delta (\epsilon_n) \Delta(y)}  \int_{z \prec y}
			\prod_{i = 1}^{n - 1} (g_1^{\epsilon})'(z_i)  \det \left[\frac{(g^{\epsilon i}_1)'}{(g^{\epsilon}_1)'}(z_j)\right]_{i, j = 1}^{n-1}  dz_1 \cdots dz_{n-1}
			\\
			&=  \frac{\prod_{j = 1}^n K(0, y_j) }{\Delta (\epsilon_n) \Delta(y)}  ( n - 1)! \int_{z \prec y} \prod_{i = 1}^{n - 1} (g_1^{\epsilon})'(z_i)  \det \left[g_2^{(i-1)\epsilon}(z_j)\right]_{i, j = 1}^{n-1}  dz_1 \cdots dz_{n-1},
		\end{align*}
		where in the final equality we have used Corollary \ref{cor:g_defn}. Applying Lemma \ref{lemma:cauchy-binet-identity} and Corollary \ref{cor:g_defn} again yields
		\begin{align*}
			M^t(\epsilon_n, y)	&=  \frac{\prod_{j = 1}^n K(0, y_j) }{\Delta (\epsilon_n) \Delta(y)}  ( n - 1)! (n-2)! \int_{z^2 \prec z^1 \prec y}  \left(\prod_{i = 1}^{n - 1} (g_1^{\epsilon})'(z_i^1) \right) \left(\prod_{i=1}^{n-2} (g_2^{\epsilon})'(z_i^2) \right)  \det \left[g_3^{(i-1) \epsilon}(z_j^2)\right]_{i, j = 1}^{n-2} \prod_{i, j} d z^i_j.
		\end{align*}
		Repeating this procedure another $n-3$ times removes the determinant term and gives
		\begin{align*}
			M^t (\epsilon_n, y) &= \frac{\prod_{j = 1}^n K(0, y_j)}{\Delta (\epsilon_n)  \Delta(y)} (\prod_{j = 0}^{n - 1} j! ) \int_{z \in \operatorname{GT(y)}} \prod_{k = 1}^{n - 1} \prod_{i = 1}^{n - k} (g^{\epsilon}_k)'(z^k_i) dz,
		\end{align*}
		where $\operatorname{GT(y)}$ is the set of all Gelfand-Tsetlin patterns $z$ of depth $n$ whose top row is equal to $y$, i.e. $z = (z^{n-1} \prec z^{n-2} \prec \cdots \prec z^1 \prec y)$, with $z^i \in \R^{n-i}$, and the integral is over Lebesgue measure on this set. We can use Corollary \ref{cor:g_defn} to write this more explicitly as
		\begin{align*}
			\frac{\prod_{j = 1}^n K(0, y_j)}{\Delta (\epsilon_n)  \Delta(y)} (\prod_{j = 1}^{n - 1} j! )  \int_{z \in \operatorname{GT}(y)} \prod_{k = 1}^{n - 1} \prod_{j = 1}^{n - k} \big( \frac{k \epsilon  M^t(\epsilon_{k - 1}, (z^k_j)^{k - 1})  M^t(\epsilon_{k + 1} , (z^k_j)^{k + 1} ) }{M^t(\epsilon_k, (z^k_j)^k)^2 }  \big) \, d z.
		\end{align*}
		where $(z^i_j)^k = (z^i_j, \dots, z^i_j) \in \R^k$.
		Now since $\Delta (\epsilon_n) = \epsilon^{n(n - 1)/2} \prod_{j = 1}^{n - 1} j!$, we can further simplify to get
		\begin{align*}
			M^t (\epsilon_n, y) &=  \frac{\prod_{j = 1}^n K(0, y_j)}{  \Delta(y)} (\prod_{j = 1}^{n - 1} j! )  \int_{z \in \operatorname{GT}(y)} \prod_{k = 1}^{n - 1} \prod_{j = 1}^{n - k} \big( \frac{M^t(\epsilon_{k - 1}, (z^k_j)^{k - 1})  M^t(\epsilon_{k + 1} , (z^k_j)^{k + 1} ) }{M^t(\epsilon_k, (z^k_j)^k)^2 }  \big) \, dz
		\end{align*}
		Next, taking $\epsilon \to 0$ and using continuity of $M^t$ on its boundary gives:
		\begin{align}
			M^t (0^n , y)
			&= \frac{\prod_{j = 0}^{n - 1} j!}{\Delta (y)} \prod_{j = 1}^n K(0, y_j )  \int_{z^k_j \in \operatorname{GT}(y)} \prod_{k = 1}^{n - 1} \prod_{j = 1}^{n - k} \big( \frac{M^t(0^{k - 1}, (z^k_j)^{k-1})  M^t(0^{k + 1} , (z^k_j)^{k+1}) }{M^t(0^k, (z^k_j)^k)^2}  \big) \, d z. \label{E:M_identity_intmstep}
		\end{align}
		The lemma is concluded by observing that the RHS may be written as a partition function over the SH line ensemble, $\{ \mathcal{Z}_n \}_{n \in \N}$, where $\mathcal{Z}_n(x) = \frac{M^t(0^n , x^n)}{M^t(0^{n - 1}, x^{n - 1})}$. Indeed, each term inside the double product equals $\cZ_{k+1}(z^k_j)/\cZ_k(z^k_j)$, and the set $\operatorname{GT}(y)$ encodes the set of all $n$-tuples of disjoint paths from $(0, (1, \dots, n))$ to $(y, 1)$ in terms of their jump times: for an $n$-tuple of disjoint paths $\pi = (\pi_1, \dots, \pi_n)$ from $(0, (1, \dots, n))$ to $(y, 1)$, letting 
		$$
		z^i_j = \inf \{t \in [0, y_{i + j}] : \pi_{i + j}(t) = i\}
		$$
		gives the encoding.
	\end{proof}
	
	\begin{proof}[Proof of Lemma \ref{lemma:LPP_values_KPZLE}]
		We prove the second statement of the lemma, as the first can be proven via a symmetric argument. For $y > 0$, let $y_{\epsilon} = (\epsilon_{n-1}, y)$. Then by Lemma \ref{lemma:LPP_SHE_Sheet}, we have
		\begin{align*}
			\frac{M^t(0^n, y_{\epsilon})}{M^t(0^n, 0^n)} &= (\prod_{j = 0}^{n - 1}j!) \frac{\mathcal{Z} \{ (0, (1, \dots, n)) \to  (y_{\epsilon} , 1) \}  }{\Delta (y_{\epsilon})}
		\end{align*}
		Taking $\epsilon \to 0$ gives:
		\begin{align*}
			\frac{M^t(0^n, \{ 0^{n-1}, y\})}{M^t(0^n, 0^n)} &= \lim_{\epsilon \to 0} (\prod_{j = 0}^{n - 1}j!) \frac{\textrm{Vol} (\operatorname{GT}(\epsilon_{n - 1})) \mathcal{Z} \{ (0, n ) \to (y, 1) \} }{ \Delta (\epsilon_{n - 1} )y^{n - 1}} 
			= \frac{(n - 1)!}{y^{n - 1}} \mathcal{Z} \{ (0 ,  n) \to (y , 1) \} 
		\end{align*}
		where in the last step we used that $\textrm{Vol} (\operatorname{GT}(z)) = \Delta (z) / (\prod_{j = 0}^{n - 1}j!)$. Rearranging concludes the proof. 
	\end{proof}

	\subsection{The remainder term: proof of Lemma \ref{lemma:KPZ_remainderterm}} \label{subsec:KPZ_remainderterm}
	
	For $n \in \N$ and $x, y > 0$, recall that the remainder term $R^{(n)}(x, y)$ is given by
	\begin{align*}
		R^{(n)}(x, y) =  x^{n} y^{n} \frac{M^{t}( \{ 0^{n}, x \}  , \{ 0^{n } , y\})  }{M^{t}(0^{n} , 0^{n})} 
	\end{align*}
	where we suppress the superscript when the index $n$ is clear. We will start by splitting $R^{(n)}(x, y)$ into two pieces by using the Chapman-Kolmogorov equation for $M^t$. For this lemma, we define $M^{s, t}$ in the same way as we defined $M^t := M^{0, t}$ (i.e. using the Karlin-McGregor formula \eqref{E:Karlin_Mcgregor_SHE}), but instead of starting with the SH sheet $Z^t(x, y) := Z(x, 0; y, t)$, we start with the shifted sheet $Z^{s, t}(x, y) := Z(x, s; y, t)$.

	\begin{lemma}[Corollary 5.4 of \cite{O_Connell_2015}] \label{lemma:OW_cor}
		Almost surely, for all $\mathbf{x} , \mathbf{y} \in \R^n_{\leq}$, and $0 < s < t$,
		\begin{align*}
			M^{0, t}(\mathbf{x}, \mathbf{y}) &= \int_{\Lambda_n} M^{0, s}(\mathbf{x}, \mathbf{z}) M^{s, t}(\mathbf{z}, \mathbf{y}) \Delta(\mathbf{z})^2 \, d \mathbf{z}. 
		\end{align*}
	\end{lemma}

	We also need a simple inequality.
	
	\begin{lemma}\label{lemma:K_and_M_basic_ineq} Almost surely, for any $s < t$, $\mathbf{x}, \mathbf{y} \in  \R^n_{\le}$ and $x_{n + 1}  > x_n$, $y_{n + 1} > y_n$, the following inequalities hold:
		\begin{align*}
			M^{s , t}(\{ \mathbf{x} , x_{n + 1}\} , \{ \mathbf{y}, y_{n + 1}\}) & \leq M^{s, t} (\mathbf{x} , \mathbf{y}) \frac{M^{s, t} (x_{n + 1}, y_{n + 1})}{\prod_{i = 1}^{n} (x_{n + 1} - x_i )   (y_{n + 1} - y_i )  }.
		\end{align*}
	\end{lemma}
	
	\begin{proof}
		It is enough to show that for $\bx, \by \in \Lambda_n$, we have
		$$
		K^{s, t} (\{ \mathbf{x} , x_{n + 1}\} , \{ \mathbf{y}, y_{n + 1}\}) \leq K^{s,t} (\mathbf{x} , \mathbf{y}) K^{s, t}(x_{n + 1}, y_{n+1}).
		$$
		Indeed, this immediately implies the desired formula when $\bx, \by \in \Lambda_n$. The extension to the boundary of $\Lambda_n$ follows by continuity. The inequality above is a standard Fischer determinant inequality for totally nonnegative matrices, e.g. see \cite[equation (1.5)]{skandera2022barrett}.
	\end{proof}

	Using these two lemmas, we obtain the following upper bound on $R^{(n)}(x,y)$.
	
	\begin{lemma} \label{lemma:remainder_firstestimate}
		For all $n \in \N, \eta \in \R$ and $x, y \geq 0$, 
		\begin{align*}
			R^{(n)}(x,y) &\leq M^{0,t}(x, y) [R_1^{(n, \eta)}+  R_2^{\eta} (x, y)],
		\end{align*}
		where
		\begin{align}
			R_1^{(n, \eta)}(x, y) &= \frac{1}{M^{0, t} (0^n, 0^n)} \int_{\Lambda_n \cap \{ z_n \leq \eta \}} M^{0, \frac{t}{2}} (0^n, \mathbf{z}) M^{\frac{t}{2}, t} (\mathbf{z}, 0^n) \Delta(\mathbf{z})^2 \, d \mathbf{z}, \label{eqn:R_1_defn}\\
			R_2^{\eta}(x,y) &=  \frac{1}{M^t(x, y)}\int_\eta^\infty M^{0 , \frac{t}{2}} (x, z) M^{\frac{t}{2}, t} (z, y) \, d z. \label{eqn:R_2_defn}
		\end{align}
	\end{lemma}
	
	\begin{proof}
		First suppose $x,y>0$. By Lemma \ref{lemma:OW_cor},
		\begin{align*}
			x^n y^n M^{0,t}(\{0^n,x\},\{0^n,y\})
			&=x^n y^n\int_{\Lambda_{n+1}}
			M^{0,t/2}(\{0^n,x\},\{\mathbf z,z_{n+1}\})
			M^{t/2,t}(\{\mathbf z,z_{n+1}\},\{0^n,y\})
			\Delta(\mathbf z,z_{n+1})^2\,d\mathbf z\,dz_{n+1}.
		\end{align*}
		On this domain $z_{n+1}>z_n$, so Lemma \ref{lemma:K_and_M_basic_ineq} applies to both kernels. After cancellation of $x^n y^n$ and the factors $\prod_{i=1}^n(z_{n+1}-z_i)^2$, the integrand is bounded by
		\[
		F(\mathbf z,z_{n+1})=
		M^{0,t/2}(0^n,\mathbf z)M^{0,t/2}(x,z_{n+1})
		M^{t/2,t}(\mathbf z,0^n)M^{t/2,t}(z_{n+1},y)
		\Delta(\mathbf z)^2.
		\]
		This product is defined and nonnegative on $\Lambda_n\times\R$. Since
		$\Lambda_{n+1}\subset[(\Lambda_n\cap\{z_n\leq\eta\})\times\R]\cup[\Lambda_n\times(\eta,\infty)]$, we may now enlarge the integration domains to obtain
		\begin{align*}
			x^n y^n M^{0,t}(\{0^n,x\},\{0^n,y\})
			&\leq\left(\int_{\Lambda_n\cap\{z_n\leq\eta\}}\int_{\R}
			+\int_{\Lambda_n}\int_{\eta}^{\infty}\right)
			F(\mathbf z,z_{n+1})\,dz_{n+1}\,d\mathbf z\\
			&=M^{0,t}(x,y)\int_{\Lambda_n\cap\{z_n\leq\eta\}}
			M^{0,t/2}(0^n,\mathbf z)M^{t/2,t}(\mathbf z,0^n)\Delta(\mathbf z)^2\,d\mathbf z\\
			&\quad+M^{0,t}(0^n,0^n)\int_{\eta}^{\infty}
			M^{0,t/2}(x,z)M^{t/2,t}(z,y)\,dz.
		\end{align*}
		Dividing by $M^{0,t}(0^n,0^n)$ gives the claim. If $x=0$ or $y=0$, then $R^{(n)}(x,y)=0$, so the inequality follows from nonnegativity.
	\end{proof}
	
	\begin{lemma}
		\label{L:R2}
		The term $R_2^\eta$ converges to $0$ as $\eta \to \infty$, uniformly on compact subsets of $[0, \infty)^2$.
	\end{lemma}
	
	\begin{proof}
		For fixed $x, y$, pointwise convergence of $R_2^\eta$ as $\eta \to \infty$ follows from the dominated convergence theorem. To upgrade to uniform-on-compact convergence, we show that $R_2^\eta$ is monotone increasing in $x, y$. We will check monotonicity in $x$, as monotonicity in $y$ is similar. For this, first observe that for $x < x' \in \R$ and $z < z' \in \R$, we have that
		$$
		M^{0, t/2}(x, z) M^{0, t/2}(x', z') - M^{0, t/2}(x, z') M^{0, t/2}(x', z) > 0.
		$$
		Indeed, the left-hand side above is the (positive) $2 \times 2$ determinant $K^{0, t/2}((x, x'), (z,z'))$. Therefore we can write $M^{0, t/2}(x', z) = f(z) M^{0, t/2}(x, z)$, where $f$ is an increasing function. Now, let $\mu$ be the probability measure on $\R$ with density $M^{0, t/2}(x, z) M^{t/2, t}(z, y)/M^t(x, y)$. The inequality $R_2^\eta(x, y) \le R_2^\eta(x', y)$ is equivalent to the statement that
		$$
		\int f(z) \mathbf{1}(z > \eta) d\mu(z) \ge \int f(z) d\mu(z) \int \mathbf{1}(z > \eta) d\mu(z),
		$$
		which follows since both $f, \mathbf{1}(z > \eta)$ are increasing.
	\end{proof}
	
	To set up the estimate on $R_1^{(n, \eta)}$, let us first write $R_1^{(n, \eta)} = B_n C_{n, \eta}$, where
	$$
	B_n = \frac{\E M^t(0^n, 0^n)}{M^t(0^n, 0^n)}, \qquad C_{n, \eta} = \int_{\Lambda_n \cap \{ z_n \leq \eta \}} \frac{M^{0, \frac{t}{2}} (0^n, z) M^{\frac{t}{2}, t} (z, 0^n) \Delta(z)^2}{\E M^t(0^n, 0^n)} \, dz.
	$$
	The term $C_{n, \eta}$ can be easily controlled. 
	\begin{lemma}
		\label{L:Cneta}
		Let $W = (W_1 \le \dots \le W_n)$ be a collection of $n$ non-intersecting Brownian bridges on $[0, t]$ with $W_i(0) = W_i(t) = 0$ for all $i$. Then for all $n \in \N, \eta > 0$,
		$$
		\E C_{n, \eta} = \mathbb P(W_n(t/2) \le \eta). 
		$$
		In particular, $\E C_{n, \sqrt{tn}/2} \le 2 \exp(-d n^2)$.
	\end{lemma}
	
	\begin{proof}
		By linearity of expectation and independence of $M^{0, t/2}$ and $M^{t/2, t}$, we have that
		$$
		\E C_{n, \eta} = \int_{\Lambda_n \cap \{ z_n \leq \eta \}} \frac{ \E [M^{0, \frac{t}{2}} (0^n, z) ]  \E [M^{\frac{t}{2}, t} (z, 0^n)]  \Delta(z)^2 }{\E M^t(0^n, 0^n)}\, dz.
		$$
		To evaluate the ratio of expectations above, we separate the endpoints at $0$. Indeed, for $x, z \in \Lambda_n$, 
		$$
		\E M^{s, t}(x, z) = \frac{\det[p(t-s, x_i - z_j)]_{{i, j}=1}^n}{\Delta(x) \Delta(z)}.
		$$
		Therefore letting $\epsilon_n = (0, \epsilon, \dots, \epsilon (n-1))$, we have 
		$$
		\frac{ \E [M^{0, \frac{t}{2}} (\epsilon_n, z) ]  \E [M^{\frac{t}{2}, t} (z, \epsilon_n)]  \Delta(z)^2 }{\E M^t(\epsilon_n, \epsilon_n)} = \frac{\det[p(t/2, \epsilon (i-1) - z_j)]_{{i, j}=1}^n \det[p(t/2, \epsilon (i-1) - z_j)]_{{i, j}=1}^n}{\det[p(t, \epsilon (i-j))]_{{i, j}=1}^n}.
		$$
		This is exactly the density at $z$ at time $t/2$ for $n$ non-intersecting Brownian bridges on $[0, t]$ starting and ending at $\epsilon_n$. Taking $\epsilon \to 0$, we recover the density for $W$. Noting also that $M^t$ converges in $L^1$ at the boundary of $\Lambda_n$ (see \cite[Lemma 6.1]{O_Connell_2015}), this allows us to evaluate the integral $\E C_{n, \eta}$, as desired. For the bound on $\E C_{n, \sqrt{t n}/2}$, note that by transforming non-intersecting Brownian bridges to Brownian motions, we have that 
		$\mathbb P(W_n(t/2) \le \eta) = \mathbb P(B_n(1) \le 2\eta/\sqrt{t})$, where $B_n$ is the top path in a family of $n$ non-intersecting Brownian motions, or equivalently, Dyson's Brownian motion with $\beta = 2$. The bound is then simply a large deviation estimate for the Gaussian Unitary Ensemble, e.g. see \cite[Theorem 1]{ledoux2010small}.
	\end{proof}
	
	On the other hand, the term $B_n$ is more difficult to understand. Bounding the probability that $B_n$ is large is essentially equivalent to controlling the heights in the KPZ line ensemble. We tackled this problem in the paper \cite{Syed_2025}.
	
	\begin{prop}[Theorem 1.1 of \cite{Syed_2025}] \label{prop:KPZ_blkheight}  
		Let $\{ \mathcal{H}^{(t)}_n \}_{n \in \N}$ be the KPZ$_t$ line ensemble and
		\begin{align*}
			e_n^{(t)} &= \log (t^{1 - n} (n-1)!) + \frac{t}{24}.
		\end{align*}
		Fix $t_0 > 0, \epsilon > 0$.  There exist constants $d, C > 0$ so that for all $n \in \N, t > t_0$ and $a \ge C (t n^{2\epsilon} + t^{1/4} n^{3/4 + 2\epsilon})$ we have
		\begin{align*}
			\prob \left(  | \mathcal{H}^{(t)}_n(0) - e_n^{(t)} | > a \right) & \leq 2 \exp(- d a n^{\epsilon}  \max(1, t^{-1/4} n^{1/4})).
		\end{align*}
	\end{prop}

	\begin{cor} \label{cor:KPZ_LE_input} Fix $t_0, \epsilon > 0$. There are $(t_0,\epsilon)$-dependent constants $c, d , C > 0$ such that for all $t > t_0$ and $a > C (t n^{2 \epsilon} + t^{1/4} n^{3/4 + 2 \epsilon})$ we have
		\begin{align*}
			\prob \Big( B_n \geq e^{ an} \Big) 	& \leq  e^{- da}. 
		\end{align*}
	\end{cor}
	
	\begin{proof} First, we can evaluate
		\begin{align*}
			\E M^t(\epsilon_n, \epsilon_n) &= \frac{\det [p(t, \epsilon(i-j))]_{i, j = 0}^{n-1}}{\Delta(\epsilon_n)^2} \\
			&= p(t, 0)^n \exp\left(-\sum_{i=1}^n \frac{\epsilon^2 (i-1)^2}{t} \right) \frac{\det [\exp(\epsilon^2 i j/t)]_{i, j = 0}^{n-1}}{\Delta(\epsilon_n)^2} \\
			&= p(t, 0)^n \exp\left(-\sum_{i=1}^n \frac{\epsilon^2 (i-1)^2}{t} \right) \frac{\Delta(1, \exp(\epsilon^2/t), \dots, \exp((n-1) \epsilon^2/t))}{\Delta(\epsilon_n)^2}
		\end{align*}
		where in the final line we have evaluated the Vandermonde determinant. At this point, it is straightforward to find $\E M^t(0^n,0^n)$ by taking $\epsilon \to 0$. We get
		$$
		\E M^t(0^n, 0^n) = \frac{p(t, 0)^n t^{-n (n - 1)/2}}{ \prod_{j = 1}^{ n- 1} j!}.
		$$
		Now, let $\mathcal H$ be a KPZ$_t$ line ensemble, coupled to $M^t$ as introduced at the beginning of Section \ref{sec:KPZ_Sheet}. That is, 
		$$
		M^t(0^n, 0^n) = \frac{\exp \left(\sum_{i=1}^n \mathcal H_i(0)\right)}{ \left( \prod_{i=1}^{n-1} i! \right)^2}.
		$$
		The preceding identities and the definition of $e_i^{(t)}$ give
		\[
		\log\frac{M^t(0^n,0^n)}{\E M^t(0^n,0^n)}
		=\sum_{i=1}^{n}\bigl(\mathcal H_i(0)-e_i^{(t)}\bigr)
		+\frac{nt}{24}-n\log p(t,0).
		\]
		By increasing the constant in the lower bound on $a$, we may assume
		$a+t/24-\log p(t,0)\ge a/2$. Therefore
		\begin{align*}
			\P(B_n\ge e^{an})
			&\le\P\left(\sum_{i=1}^{n}
			(\mathcal H_i(0)-e_i^{(t)})\le-an/2\right)\\
			&\le\sum_{i=1}^{n}
			\P(\mathcal H_i(0)-e_i^{(t)}\le-a/2)
			\le 2n e^{-d_0a/2}\le e^{-da},
		\end{align*}
		where in the last step we used Proposition \ref{prop:KPZ_blkheight} with the assumption that $a > C (t n^{2 \epsilon} + t^{1/4} n^{3/4 + 2 \epsilon})$.
	\end{proof}
	
	We now put everything together to bound $R_1$.
	
	\begin{lemma} \label{lemma:R_1_tails} Fix $t_0 > 0, \epsilon > 0$. There exists a constant $d > 0$  such that for all $t > t_0, n \ge t^{1 + \epsilon}$ we have:
		\begin{align*} 
			\prob \big( R_1^{(n, \sqrt{tn}/2)} \geq  e^{- dn^2}  \big) &\leq e^{-d n}. 
		\end{align*}
		In particular, $R_1^{(n, \sqrt{tn}/2)} \to 0$ almost surely as $n \to \infty$ by the Borel-Cantelli lemma.
	\end{lemma}
	
	\begin{proof}
		Throughout the proof we let all constants $d, d'$ depend on $t_0, \epsilon$. We have 
		$$
		\mathbb P(R_1^{(n, \sqrt{tn}/2)} \ge e^{-d n^2}) \le \mathbb P(B_n > \exp(d n^2)) + \mathbb P(C_{n, \sqrt{tn}/2} > \exp(-2d n^2)).
		$$
		The second term above is bounded by $\exp(-d' n^2)$ by Lemma \ref{L:Cneta} and Markov's inequality, as long as $d$ is sufficiently small. The term $\mathbb P(B_n > \exp(d n^2/2))$ is bounded by $\exp(-d'n)$ by applying Corollary \ref{cor:KPZ_LE_input} with $a = d n/2$. This choice of $a$ satisfies the conditions of that corollary by the bound $n \ge t^{1 + \epsilon}$. This yields the inequality in the lemma. 
	\end{proof}
	
	\begin{proof}[Proof of Lemma \ref{lemma:KPZ_remainderterm}]
		By Lemma \ref{lemma:remainder_firstestimate}, we have $R^{(n)}(x, y) \le M^{0,t}(x, y)[R_1^{(n, \sqrt{tn}/2)} + R_2^{\sqrt{tn}/2}(x, y)]$. The $M^{0, t}$ term is continuous, and so uniformly bounded on compact sets. The term $R_1^{(n, \sqrt{tn}/2)}$ converges to $0$ almost surely by Lemma \ref{lemma:R_1_tails} and the $R_2$ term converges to $0$ uniformly on compact sets by Lemma \ref{L:R2}.
	\end{proof}
	
	\begin{remark}
		While we have not done so explicitly here, it is not difficult to combine the quantitative estimate in Lemma \ref{lemma:R_1_tails} with a similar quantitative estimate on $R_2$ to show that for fixed $x, y, t$, as $n \to \infty$ the remainder $R^{(n)}(x, y)$ is of order $\exp(-d n^2)$ off of an exponentially rare event.    
	\end{remark}

	\section{DtM Representation for the Airy Sheet} \label{sec:Airy_case}
	
	In this section, we establish Theorem \ref{T:shadow-sheet-simple}, which shows that the Airy sheet and the Airy line ensemble can be related by a DtM isometry. Our proof also gives an alternate construction of the Airy sheet from Brownian LPP. We will similarly give a new construction of the extended (multi-path) Airy sheet.
	
	To set up the main theorem, we must define last passage percolation across lines. The setup here is essentially the same as in Section \ref{sec:DTM-KPZ}, with the $(+, \times)$-algebra replaced by the $(\max, +)$-algebra. We use notation from the beginning of that section.
	Consider a sequence of continuous functions $f = (f_i:\R \to \R, i \in \Z)$. For $x \le y$ and $m \ge n$, and an up-right path $\pi \in Q[(x, m) \to (y, n)]$, we define 
	$$
	\|\pi\|_f = \sum_{i=n}^m f_i(t_{i-1}) - f_i(t_i), 
	$$
	where the jump times $x = t_m \le \cdots \le t_{n-1} = y$ are defined by $t_i = \inf \{z \in [x, y] : \pi(z) \le i\}$. Then the \textbf{last passage value} from $(x, m)$ to $(y, n)$ is given by
	$$
	f[(x, m) \to (y, n)] = \sup_{\pi \in Q[(x, m) \to (y, n)]} \|\pi\|_f.
	$$
	Extending this definition to multiple paths, suppose $\bp = ((x_1, m_1), \dots, (x_k, m_k)), \bq = ((y_1, n_1), \dots, (y_k, n_k)) \in (\R \times \Z)^k$ satisfy the ordering constraints $x_i \le y_i, m_i \ge n_i$, and $x_1 \le \cdots \le x_k$, $y_1 \le \cdots \le y_k$ and $m_1 \le \cdots \le m_k$, $n_1 \le \cdots \le n_k$, and that $Q[\bp \to \bq]$ is non-empty. Then we set
	$$
	f[\bp \to \bq] = \sup_{\pi \in Q[\bp \to \bq]} \sum_{i=1}^k \|\pi_i\|_f.
	$$
	Of course, all definitions make sense when the domain of $f$ is restricted to a subset of $\R \times \Z$.
	
	\begin{thm} \label{T:shadow-sheet}
		Let $\{ B_n \}_{n \in \N}$ be a family of independent standard Brownian motions and let $\mathfrak{X} = \bigcup_{m \in \N} \R^m_{\leq} \times \R^m_{\leq}$. For $n \in \N$, define the prelimiting extended Airy sheet $\mathcal{S}_n : \mathfrak{X} \to \R$ so that for all $(\mathbf{x}, \mathbf{y}) \in \R^k_{\leq } \times \R^k_{\leq}$
		\begin{align} \label{E:prelim_AirySheet}
			\mathcal{S}_n(\bx, \by) = n^{-1/3} \lf(B[(2\bx n^{2/3}, n) \to (n + 2 \by n^{2/3}, 1)] - 2k n - 2 n^{2/3} \sum_{i=1}^k (y_i - x_i) \rg) .
		\end{align}
		\begin{enumerate}[label=(\roman*)]
			\item (\cite[Theorem 3.1]{dauvergne2021disjoint}) The sequence of functions $\mathcal{S}_n$ is tight in the topology of uniform convergence on compact sets.
			\item Let $T_c:\mathfrak X \to \mathfrak X$ be the translation operator: $T_c(\bx, \by) = (\bx + c^k, \by + c^k)$ for $\bx, \by\in \R^k$. Then any subsequential limit $\mathcal{S}$ of $\cS_n$ is translation invariant: $\cS \stackrel{d}{=} \cS \circ T_c$.
			
			\item For any subsequential limit $\cS$ of $\cS_n$, there exists a coupling between the parabolic Airy line ensemble $\cA = \{ \mathcal{A}_n \}_{n \in \N}$ and $\mathcal{S}$ so that
			\begin{align} \label{E:AiryLE_retrieval}
				\sum_{i = 1}^k \mathcal{A}_i(x) &= \begin{cases}
					\mathcal{S}(0^k, x^k), & x \geq 0 \\
					\mathcal{S}((-x)^k, 0^k) & x < 0 
				\end{cases} , 
				\qquad \forall x \in \R, k \in \N,
			\end{align}
			and for all $\mathbf{x}, \mathbf{y} \in \R^{k}_{\leq}$ with $x_1, y_1 \geq 0$	
			\begin{align} \label{E:DtM_AirySheet}
				\mathcal{S}(\bx, \by) 
				&=\max_{I \in \N^k_<}   \tilde{\mathcal{A}}[(0,I) \to (\bx, 1)] + \mathcal{A}[(0, I) \to (\by, 1)] + \sum_{i \in I} \mathcal{A}_i(0),
			\end{align}
			where $\tilde{\mathcal{A}}_i(x) = \mathcal{A}_i(-x)$. 
			
			\item The distribution of $\mathcal{S}$ is uniquely determined by (ii) and (iii). In particular, combining (i)-(iii), we can conclude that the sequence $\cS_n$ converges in law to a limit $\cS$, called the \textbf{extended Airy sheet}.
		\end{enumerate}
	\end{thm}
	
	\begin{remark} As mentioned previously, for single points $x, y$, the Airy sheet was constructed in \cite{dauvergne2022directed}. The extension to multiple paths was shown in \cite{dauvergne2021disjoint}. The construction of the directed landscape from the Airy sheet is straightforward and axiomatic. It is analogous to the construction of Brownian motion from the normal distribution, see \cite[Section 10]{dauvergne2022directed} for details.
	\end{remark}
	
	Note that part (ii) of Theorem \ref{T:shadow-sheet} is immediate, and part (iv) is immediate from parts (i)-(iii). 
	Therefore given part (i), the only part of Theorem \ref{T:shadow-sheet} that requires proof is part (iii). This is the goal of the remainder of this section. We start by introducing RSK across semi-discrete environments.

	\subsection{Semi-discrete RSK} 
	\label{SS:semi-discrete-RSK}
	
	\begin{defn}[The melon map]
		Let $f = (f_1, \dots, f_n)$ be a sequence of continuous functions $f_i:[0, \infty) \to \R$. Define the \textbf{$n$-melon} $Wf = (Wf_1, \dots, Wf_n)$ where  $Wf_i:[0, \infty) \to \R$ is a continuous function by the equations
		$$
		\sum_{i=1}^k Wf_i(x) = f[(0^k, n) \to (x^k, 1)], \qquad x \ge 0, k = 1, \dots, n.
		$$
		We have the ordering $Wf_1 \ge \cdots \ge Wf_n$ and $Wf_i(0) = 0$ for all $i$. 
	\end{defn}
	The melon map is a semi-discrete limit of the RSK correspondence (see Section \ref{sec:discrete}). In this limit, only one of the two tableaux remains, and this tableau becomes the melon $Wf$. As with the classical RSK correspondence, the melon map satisfies an isometry.
	
	\begin{lemma}[\cite{dauvergne2022directed}] \label{lemma:RSK_Isometry_semidiscrete}
		Let $f = (f_1, \dots, f_n)$ be a sequence of continuous functions $f_i:[0, \infty) \to \R$.
		Then for all $k = 1,\ldots, n$ and $\mathbf{x}, \mathbf{y} \in \R^k_{\leq}$ with $x_1, y_1 \geq 0$,
		\begin{align*}
			Wf [(\mathbf{x}, n) \to (\mathbf{y}, 1)] = f [(\mathbf{x}, n) \to (\mathbf{y}, 1)].
		\end{align*}
	\end{lemma}
	
	Our key deterministic input in the proof of Theorem \ref{T:shadow-sheet}(iii) concerns the interplay between the RSK isometry and rotations of the environment. Equivalently, this is related to the interplay between the RSK isometry and the Sch\"utzenberger involution. To state the next lemma, for an environment $f = (f_1, \dots, f_n), f_i:[0, t] \to \R$, let $R^t f = ((R^tf)_i:[0, t] \to \R, i = 1, \dots, n)$ denote the rotation of the environment by $\pi$. More precisely, $$
	(R^tf)_i(s) = - f_{n + 1 - i}(t-s) + f_{n + 1 - i}(t). 
	$$
	We also define \textbf{backwards first passage} across a collection of continuous functions in the restricted context in which we use it. Let $f:\II{1, n} \times [0, t] \to \R$, and let $\bx \in [0, t]^k_\ge, I \in \II{1, n}^k_\le$. Then
	$$
	f[(\bx, 1) \to_f (t, I)] = \inf_\pi \sum_{i=1}^k \|\pi_i\|_f = - (\sup_\pi \sum_{i=1}^k \|\pi_i\|_{-f}).
	$$
	Here both the supremum and the infimum above are taken over $k$-tuples of non-decreasing cadlag functions $\pi = (\pi_1, \dots, \pi_k)$ where $\pi_i:[x_i, t] \to \II{1, I_i}$ for all $i$ and, for $i<j$, $\pi_i < \pi_j$ on their common open time interval $(\max\{x_i,x_j\},t)=(x_i,t)$. As with non-increasing cadlag functions, $\|\pi_i\|_f$ is defined as a sum of increments.
	
	\begin{lemma}
		\label{C:f-corollary}
		Let $f = (f_1, \dots, f_n)$ be an environment of continuous lines defined on $[0, t]$. Then for any vectors $\bx \in \R^k_\le, I \in \{  1, \ldots, n\}^k_{<}$ such that there is at least one disjoint $k$-tuple from $(\bx, n)$ to $(t, I)$ we can write
		\begin{equation}
			\label{E:Wfxn}
			\begin{split}
				Wf[(\bx, n) \to (t, I)] &= -WR^tf[(t\mathbf{1} - \bx, 1) \to_f (t, I)] + \sum_{\ell=1}^k  W R^tf_{I_\ell}(t),  \\
				&= \tilde W^tf[(0, I) \to (\bx, 1)] + \sum_{\ell=1}^k  \tilde W^tf_{I_\ell}(0), 
			\end{split}
		\end{equation}
		where $\tilde W^tf(s) = WR^tf(t - s).$
	\end{lemma}
	
	We remark for later use that $\tilde W^t f$ can be described by a formula which is as simple as that of $Wf$:
	\begin{equation}
		\label{E:tildeWf}
		\sum_{i=1}^\ell \tilde W^tf_i(s) = f[(s^\ell, n) \to (t^\ell, 1)]. 
	\end{equation}
	
	When $\bx, \by$ are singletons, Lemma \ref{C:f-corollary} was shown in \cite[Lemma 5.3]{dauvergne2022directed}. As done in previous sections, the easiest way to pass from a single-path LPP identity to a multi-path LPP identity is to first establish the single-path identity in a related polymer model, then pass to the multi-path identity by LGV, and then take a zero temperature limit. In anticipation of Lemma \ref{C:f-corollary}, we proved a polymer version in Theorem \ref{T:rotation-isometry} for fully discrete models, and passed to an LPP limit in the fully discrete setting in Corollary \ref{C:LPP-result}. From here, taking a second limit to the semi-discrete setting is straightforward.

	\begin{proof}
		The second equality in \eqref{E:Wfxn} is immediate from the definition.
		For the first equality, by continuity of both sides of \eqref{E:Wfxn}, it suffices to prove the lemma for $\bx$ with $0 < x_1 < \dots < x_k < t$ and $x_i/t \in \Q$ for all $i$. Fix such an $\bx$ and any $I$.
		
		For every $m$, define the array $M^m: \II{1, m} \times \II{1, n} \to \R$ by letting $M^m_{(i, j)} = f_j(it/m) - f_j((i-1)t/m)$. From the arrays $\Psi M^m$ and $\Psi R M^m$, we define $n$-tuples of cadlag functions $g^m = (g^m_1, \dots, g^m_n), h^m = (h^m_1, \dots, h^m_n)$ by
		\begin{align*}
			g^m_j(x) = \sum_{i \in \II{j, \fl{m x/t}}} (\Psi M^m)_{(i, j)}, \qquad h^m_j(x) = \sum_{i \in \II{j, \fl{m x/t}}} (\Psi R M^m)_{(i, j)},
		\end{align*}
		for $x \in \{0, t/m, \dots, t\}$, with empty sums equal to zero, and by linear interpolation at other times.
		The definition of $\Psi$ implies that for $x \in \{0, t/m, \dots, t\}$ we have
		\begin{align*}
			\sum_{j=1}^k g^m_j(x) &= L_{M^m}((1, n ; m x/t, 1)^{k \wedge (m x/t)}), \qquad \\
			\sum_{j=1}^k h^m_j(x) &= L_{M^m}((m + 1 - m x/t, n ; m, 1)^{k \wedge (m x/t)}),
		\end{align*}
		where $L(u^0) := 0$. Since the functions $f_i$ are uniformly continuous on $[0, t]$, from these formulas we can see that $g^m \to W f$ and $h^m \to W Rf$ uniformly as $m \to \infty$. Now choose a subsequence $Y \sset \N$ so that for all $m \in Y, i \in \II{1, k}$ we have $mx_i/t \in \II{n, m}$. By appealing to Corollary \ref{C:LPP-result} and using the uniform convergence of $g^m \to W f$ and $h^m \to W Rf$, it is enough to show that along this subsequence all three of the differences
		\begin{align*}
			&g^m[(\bx, n) \to (t, I)] - L_{\Psi M^m}(m\bx/t, n; m, I) \\
			&-h^m[(t\mathbf{1} - \bx, 1) \to_f (t, I)] - \check L_{-\Psi RM^m}((m + 1 - m\bx/t, 1), (m, I)) \\
			& \sum_{\ell=1}^k \left(h^m_{I_\ell}(t) - L_{\Psi R M^m}(I_\ell, I_\ell; m, I_\ell)\right)
		\end{align*}
		converge to $0$ with $m$. The third of these differences equals $0$ for all $m$ by definition. Understanding the first and second differences amounts to analyzing the difference between taking a supremum over weights of discrete paths and continuous paths. The difference here is negligible in the limit since the size of our discrete mesh is converging to $0$, and $g^m, h^m$ are uniformly converging sequences of functions.
	\end{proof}
	
	The next lemma is immediate from the metric composition law and Corollary \ref{C:f-corollary} and will be the key step in proving the DtM Representation for the Airy sheet (Theorem \ref{T:shadow-sheet}). It can be viewed as a kind of DtM representation for semi-discrete LPP.
	
	\begin{lemma}
		\label{L:mc-consequence}
		Let $f:\II{1, n} \times [0, \infty) \to \R$ be continuous, and fix $b > 0$. Then for all $k \in \II{1, n}, \bx, \by \in [0, \infty)^k_\le$, with $x_k \le b \le y_1$ we have
		\begin{equation}
			\label{E:MC-Law}
			\begin{split}
				&f[(\bx, n) \to (\by, 1)] \\
				&=\max_{I \in \II{1, n}^k_<} Wf[(b, I) \to (\by, 1)] + Wf[(\bx, n) \to (b, I)]\\
				&= \max_{I \in \II{1, n}^k_<} Wf[(b, I) \to (\by, 1)] + \tilde W^b f[(0, I) \to (\bx, 1)] + \sum_{j \in I} \tilde W^bf_j(0),
			\end{split}
		\end{equation}
		and the set of $I \in \N^k_\le$ that achieve the maximum is the same on the two right-hand sides above.
	\end{lemma}

	\subsection{Brownian LPP}
	
	We now move on to the setting of Theorem \ref{T:shadow-sheet}, when our semi-discrete environment consists of independent Brownian motions. The starting point for studying this model is the observation that the output of the melon map is a system of $n$ non-intersecting Brownian motions.
	
	\begin{thm}[\cite{DNV2}, \cite{OConnell2002}] \label{lemma:Brownian_Burke}
		Let $B^n = (B_1^n , \ldots, B_n^n)$ be $n$ independent standard Brownian motions all started at $0$, and let $WB = WB^n$ be the melon map applied to $B^n$. Then $WB$ is a collection of $n$ independent Brownian motions $(WB_1, \dots, WB_n)$, conditioned on the non-crossing event
		\begin{equation}
			\label{E:WB}
			WB_1 \ge \cdots \ge W B_n.
		\end{equation}
		In other words, $WB$ is an $n$-level Dyson's Brownian motion with $\beta = 2$.
	\end{thm}
	
	The conditioning in Theorem \ref{lemma:Brownian_Burke} needs to be understood in a limiting sense, e.g. we can condition on the ordering \eqref{E:WB} on an interval $[\epsilon, \epsilon^{-1}]$ and then take $\epsilon \to 0$. This description implies that $WB$ has tractable exact formulas and we can take asymptotics. 
	
	\begin{thm}[\cite{CH}]  \label{lemma:AiryLE_defn}
		Let $WB^n$  be as in Theorem \ref{lemma:Brownian_Burke}. Define the rescaled ensemble $A^n = (A^n_1,\ldots, A^n_n)$ as
		\begin{align*}
			A^n_i(y) &= n^{-1/3} \left(  (WB^n)_i (n + 2 y n^{2/3}) - 2n - 2yn^{2/3}  \right).
		\end{align*}
		Then $A^n$ converges in law with respect to the product of uniform-on-compact topologies to a random sequence of continuous functions $\mathcal A = \{ \mathcal{A}_n:\R \to \R \}_{n \in \N}$ known as the \textbf{parabolic Airy line ensemble}. 
	\end{thm}
	
	Now, let us return to \eqref{E:prelim_AirySheet} and suppose that $x_1 \ge 0, y_1 \ge 0$. By Lemma \ref{L:mc-consequence}, we can rewrite the right-hand side of \eqref{E:prelim_AirySheet} as 
	\begin{align*} 
		\max_{I \in \II{1, n}^k_<} n^{-1/3} \lf(W B[(n, I) \to (n + 2 \mathbf{y} n^{2/3}, 1)] + \tilde W^n B[(0, I) \to (2\bx n^{2/3}, 1)] + \sum_{j \in I} \tilde W^n B_j(0) - 2k n - 2 n^{2/3} \sum_{i=1}^k (y_i - x_i) \rg),
	\end{align*}
	where here all melon maps $W, \tilde W^n$ are $n$-line maps. Next, letting $A^n$ be as in Theorem \ref{lemma:AiryLE_defn} with $B^n = (B_1, \dots, B_n)$, and defining $\tilde A^n$ by
	$$
	\tilde A^n_i(x) = n^{-1/3}((\tilde W^n B)_i(2 x n^{2/3}) - 2n + 2 x n^{2/3}), 
	$$
	we can rewrite the above display as
	\begin{align} 
		\label{E:shadow-sheet-pre}
		\cS_n(\bx, \by) = \max_{I \in \II{1, n}^k_<} A^n[(0, I) \to (\mathbf{y}, 1)] + \tilde A^n[(0, I) \to (\mathbf{x}, 1)] + \sum_{j \in I} \tilde A^n_j(0).
	\end{align}
	This is now a prelimiting version of the formula in Theorem \ref{T:shadow-sheet}(iii). To prove that this passes to the limit, the main technical point is to show that the maximizing index above is tight.

	\begin{lemma} \label{lemma:k_tightness}
		For $n \in \N$,  let $I^n = I^n(\mathbf{x}, \mathbf{y})$ denote the largest $I \in \{ 1, \ldots, n\}^k_{<}$ (chosen in the lexicographic order) that attains the maximum in \eqref{E:shadow-sheet-pre}.
		
		Then for any $k \in \N$, and any compact set $K \subset [0, \infty)^k_\le \times [0, \infty)^k_\le$, the random variable $\sup_{(\mathbf{x}, \mathbf{y}) \in K} I^n_k(\mathbf{x}, \mathbf{y})$ is tight. 
	\end{lemma}
	
	We postpone the proof of Lemma \ref{lemma:k_tightness} until the next section, and proceed with the proof of Theorem \ref{T:shadow-sheet}(iii). The main remaining step is to identify the joint limit of $(A^n, \tilde A^n)$ on $[0, \infty)$. For this, define
	$$
	\cA^{n}_i(x) = \begin{cases}
		A^n_i(x), \quad &x \ge 0, \\
		\tilde A^n_i(-x), \quad &x \le 0.
	\end{cases}
	$$
	\begin{lemma}
		\label{L:An-limit}
		The random sequences of functions $\cA^n = (\cA^n_1, \dots, \cA^n_n):[-n^{1/3}/2, \infty) \to \R^n$ and $A^n = (A^n_1, \dots, A^n_n):[-n^{1/3}/2, \infty) \to \R^n$ are equal in law. In particular, $\cA^n$ converges in law to an Airy line ensemble.
	\end{lemma}
	
	\begin{proof}
		This is essentially contained in \cite{dauvergne2022hidden}, but we include a proof for completeness. First, we have that
		\begin{equation*}
			(\cA^n|_{[-n^{1/3}/2, 0]}, \cA^n(0), \cA^n|_{[0, \infty)}) = (\cA^n|_{[-n^{1/3}/2, 0]}, A^n(0), A^n|_{[0, \infty)}).
		\end{equation*}
		Next, $\cA^n|_{[-n^{1/3}/2, 0]} \stackrel{d}{=} A^n|_{[-n^{1/3}/2, 0]}$. This is an immediate consequence of the fact that Brownian motion has time-stationary increments and is invariant under time reversal. Moreover, by the Markov property for Dyson's Brownian motion, $A^n|_{[-n^{1/3}/2, 0]}$ and $A^n|_{[0, \infty)}$ are conditionally independent given $A^n(0)$. Therefore to complete the proof, we just need to show that $\cA^n|_{[-n^{1/3}/2, 0]}$ and $A^n|_{[0, \infty)}$ are conditionally independent given $A^n(0)$. Equivalently, it is enough to show that the conditional distribution of $A^n|_{[0, \infty)}$ given $\cA^n|_{[-n^{1/3}/2, 0]}$ is the same as the conditional distribution of $A^n|_{[0, \infty)}$ given $A^n|_{[-n^{1/3}/2, 0]}$, since the latter equals the conditional distribution of $A^n|_{[0, \infty)}$ given only $A^n(0)$.

		This follows from the stronger property that $\cA^n|_{[-n^{1/3}/2, 0]}$ and $A^n|_{[-n^{1/3}/2, 0]}$ are actually measurable functions of each other (this is a limiting version of the Sch\"utzenberger involution). Indeed, $A^n|_{[-n^{1/3}/2, 0]}$ encodes the melon $W B$ on the interval $[0, n]$ and the $\cA^n|_{[-n^{1/3}/2, 0]}$ encodes the melon $W R^n B$ on the same interval. Both of these objects can be built from LPP values of the form
		$$
		f[(x^k, n) \to (y^k, 1)],
		$$
		where $x, y \in [0, n]$, and therefore by the RSK isometry (Lemma \ref{lemma:RSK_Isometry_semidiscrete}), each of these objects can be constructed from the other. 
	\end{proof}
	
	\begin{proof}[Proof of Theorem \ref{T:shadow-sheet}(iii)]
		Let $\cA^n$ be as in Lemma \ref{L:An-limit}, and let $(\cS, \cA)$ be a joint subsequential limit (in law) of $(\cS_n, \cA^n)$. Here $\cA$ is a parabolic Airy line ensemble by Lemma \ref{L:An-limit}, and tightness of $\cS_n$ is provided by Theorem \ref{T:shadow-sheet}(i).
		
		Next, for $x, y \ge 0$, we have
		\[
		\mathcal S_n(0^k,y^k)=\sum_{i=1}^{k}A_i^n(y),\qquad
		\mathcal S_n(x^k,0^k)=\sum_{i=1}^{k}\tilde A_i^n(x).
		\]
		Passing to the joint subsequential limit gives
		\eqref{E:AiryLE_retrieval}. Finally, \eqref{E:shadow-sheet-pre} becomes \eqref{E:DtM_AirySheet} in the limit. The only thing to check is tightness of the maximizing location $I$, which is provided by Lemma \ref{lemma:k_tightness}.
	\end{proof}

	\subsection{Proof of Lemma \ref{lemma:k_tightness}} \label{subsec:k_tightness}
	
	The proof of Lemma \ref{lemma:k_tightness} follows from combining tail bounds on the random variables $\tilde A^n_j(0), j \in \II{1, n}$ with a modulus of continuity on $A^n, \tilde A^n$ which allows us to control the last passage values. The sequence $(\tilde A^n_j(0), j \in \II{1, n})$ is simply a rescaling of the eigenvalue process for the Gaussian Unitary Ensemble, for which strong tail bounds are classical. The following bound will suffice for us.
	
	\begin{lemma}
		\label{L:GUE-tail}
		There exist absolute constants $c, c', \epsilon_0, \beta_0 > 0$ such that $c - \beta_0 j^{2/3}\le \E \tilde A^n_j(0) \le c - \epsilon_0 j^{2/3}$ for all $1 \le j \le n, n \in \N$, and
		$$
		\mathbb P(\tilde A^n_j(0) \ge - \epsilon_0 j^{2/3}) \le 2 e^{-c' j}.
		$$
	\end{lemma}
	
	This can be extracted from \cite{gustavsson2005gaussian}; we use a bound from \cite{gotze2005rate} for a cleaner argument. 
	
	We will also use the moderate deviation bounds on the top eigenvalue, which we state here since they will also be of future use.
	
	\begin{lemma}[\cite{ledoux2010small}] \label{lemma:ledoux_rider_GUE_small} Let $\lambda_1 = WB^n_1(1)$ be the top eigenvalue of the GUE. There exist constants $c, d > 0$ so that for all $n \geq 1$ and $0 < \epsilon \leq 1$,
		\begin{align*}
			\prob \left( \lambda_1 \leq 2 \sqrt{n} (1 - \epsilon) \right) \leq c e^{-d n^2 \epsilon^{3}}, \qquad \prob \left( \lambda_1 \geq 2 \sqrt{n} (1 + \epsilon) \right) \leq c e^{-d n \epsilon^{3/2}}.
		\end{align*}
		The second bound also holds when $\epsilon > 1$ with the $\epsilon^{3/2}$ replaced by $\epsilon^2$ by a large deviation estimate, see \cite[equation (1.4)]{ledoux2010small}.
	\end{lemma}

	\begin{proof}[Proof of Lemma \ref{L:GUE-tail}]
		Let $(\lambda_1 \ge \cdots \ge \lambda_n)$ be the eigenvalues of an $n \times n$ GUE, i.e. in the notation of Theorem \ref{lemma:Brownian_Burke}, $\lambda_i = WB_i(1)$. Set $N_{s, t} = \# \{i : s \sqrt{n} \le \lambda_i \le t \sqrt{n}\}$. By Theorem 1.1 in \cite{gotze2005rate},
		$$
		\sup_{t \in [-2, 2], n \in \N} \left|\E N_{t, \infty} - n \int_{t}^2 \frac{1}{2\pi} \sqrt{4 - u^2} du \right| \le c,
		$$
		for an absolute constant $c$. Now, since $(\lambda_1, \dots, \lambda_n)$ is a determinantal point process with a Hermitian kernel, each of the random variables $N_{t, s}$ is a sum of independent Bernoulli random variables. In particular, we have the estimate $\operatorname{Var}(N_{t, s}) \le \E N_{t, s}$, and we may also apply Bernstein's inequality to get tail bounds.
		
		Now, we can choose $C, d > 0$ so that
		$$
		d z \le n \int_{2 - z^{2/3} n^{-2/3}}^2 \frac{1}{2\pi} \sqrt{4 - u^2} du \le C z
		$$
		for all $n \in \N, z \in [0, n]$. Then for small enough $\epsilon_0 > 0$,
		\begin{align*}
			\P( \frac{\lambda_k}{\sqrt{n}} > 2 - \epsilon_0 k^{2/3} n^{-2/3}) &= \P(N_{2-\epsilon_0 k^{2/3} n^{-2/3}, \infty} \geq k) \\
			&\le \P(N_{2-\epsilon_0 k^{2/3} n^{-2/3}, \infty} > 2 \E N_{2-\epsilon_0 k^{2/3} n^{-2/3}, \infty}) \\
			&\le 2\exp(-c' \E N_{2-\epsilon_0 k^{2/3} n^{-2/3}, \infty}) \\
			&\le 2\exp(-c'k), 
		\end{align*}
		where we have used Bernstein's inequality, and the constants may change line by line. The expectation upper bound follows by combining this tail bound with the upper bound on $\lambda_1$ in Lemma \ref{lemma:ledoux_rider_GUE_small}. Similarly, for large enough $\alpha_0 > 0$, for all $z \in [\alpha_0 k, n]$, we have
		\begin{align*}
			\P(\frac{\lambda_k}{\sqrt{n}} < 2 - z^{2/3} n^{-2/3})
			&= \P(N_{2-z^{2/3}n^{-2/3},\infty} \le k-1)
			\le 2\exp(-c'z)
		\end{align*}
		and the lower bound on the expectation follows by combining this with the fact that $\lambda_k \ge \lambda_n \stackrel{d}{=} - \lambda_1$, and another application of the upper tail in Lemma \ref{lemma:ledoux_rider_GUE_small}.
	\end{proof}
	
	For the modulus of continuity, we use the following theorem from \cite{wu2024applications}.
	
	\begin{thm}[one case of Theorem 1.1(ii), \cite{wu2024applications}]
		\label{T:mod-cont}
		Let $W = WB$ be an $n$-level Dyson's Brownian motion with $\beta = 2$, as in Theorem \ref{lemma:Brownian_Burke}, and let $\hat W_i = W_i - \E W_i$. There exist universal constants $C_1, C_2 > 0$ such that for all $0\leq a < b$, $1 \le j \le n$, and $K \ge 0$,
		$$
		\P\left(\sup_{t, s \in [a, b] , \, t\neq s} \frac{|\hat W_j(t) - \hat W_{j}(s)|}{\sqrt{|t-s| \log (\frac{2 (b-a)}{|t - s|})}}  \ge K \right) \le C_1 \exp(-C_2 K^2).
		$$
	\end{thm}
	
	Observe that in Theorem \ref{T:mod-cont}, we may replace $W$ with $A^n$ or $\tilde A^n$. Indeed, $A^n$ can be constructed from $W$ by:
	\begin{itemize}[nosep]
		\item First translating and subtracting an affine function. This does not affect the theorem, which already removes the mean,
		\item Then applying Brownian scaling, under which the theorem is invariant.  
	\end{itemize} 
	\begin{proof}[Proof of Lemma \ref{lemma:k_tightness}]
		Without loss of generality, we will consider compact sets $K$ of the form $K= [0, a]^k_\le \times [0, a]^k_\le$. By Theorem \ref{T:shadow-sheet}(i), $\inf_K \cS_n$ is tight. Therefore for $I = (I_1, \dots, I_k)$, if we let
		$$
		X_n(I, \bx, \by) = A^n[(0, I) \to (\mathbf{y}, 1)] + \tilde A^n[(0, I) \to (\mathbf{x}, 1)] + \sum_{j \in I} \tilde A^n_j(0),
		$$
		it is enough to show that for any $b \in \R$: 
		\begin{equation}
			\label{E:minfity}
			\lim_{m \to \infty} \limsup_{n \to \infty} \P\left(\sup_{(\bx, \by) \in K} \max_{I : I_k \ge m} X_n(I, \bx, \by) \ge   b\right) = 0. 
		\end{equation}
		Here, to simplify matters note we may replace $\tilde A^n(x)$ with $\underline{A}(x) :=A^n(-x)$ without changing the law by Lemma \ref{L:An-limit}. Next, 
		$$
		X_n(I, \bx, \by) \le \sum_{j=1}^k X_n(I_j, x_j, y_j) \le \sum_{j=1}^{k-1} \cS_n(x_j, y_j) + X_n(I_k, x_k, y_k),
		$$
		and so 
		$$
		\sup_{(\bx, \by) \in K} \max_{I : I_k \ge m} X_n(I, \bx, \by) \le (k-1)\sup_{x, y \in [0, a]} \cS_n(x, y) +  \max_{m' \ge m, x, y \in [0, a]} X_n(m', x, y).
		$$
		By Theorem \ref{T:shadow-sheet}(i), the first term on the right-hand side above is tight, and so it suffices to focus on the second term. Equivalently, it is enough to prove \eqref{E:minfity} with $k = 1$. For this, writing $A^n = \hat A^n + \E A^n$, we have the bound
		$$
		X_n(m', x, y) \le A^n_{m'}(0) + \hat A^n[(0, m') \to (y, 1)] + \underline{\hat A}^n[(0, m') \to (x, 1)] + \E A^n[(0, m') \to (y, 1)] + \E \underline{A}^n[(0, m') \to (x, 1)].
		$$
		Now, by applying Theorem \ref{T:mod-cont} to all lines $\hat A^n_{m'}$ for $1 \le m' \le n$ with $K = \log (m + m')$, we have that with probability $1 - C_1 m \exp(-C_2 \log^2 m)$, for all $m \le m' \le n$ and all $0 \le y \le a$:
		\begin{equation*}
			\hat A^n[(0, m') \to (y, 1)] \le \sup_{0 = t_{m'} \le t_{m' - 1} \le \cdots \le t_1 \le t_0 = y} \log (m + m') \sum_{i=1}^{m'} \sqrt{|t_i - t_{i-1}| \log (\frac{2 a}{|t_i - t_{i-1}|})} \le 4 \log(m')^2 \sqrt{m'a}.
		\end{equation*}
		An identical bound holds for $\underline{\hat A}^n[(0, m') \to (y, 1)]$. Combining this with Lemma \ref{L:GUE-tail}, we have that there exist absolute constants $C_1, C_2, C_3, \epsilon_0 > 0$ such that with probability at least $1 - C_1 m \exp(-C_2 \log^2 m)$, for all large enough $n$ and all $m' \ge m$, we have 
		\begin{equation}
			\label{E:supxy}
			\sup_{0 \le x, y \le a} X_n(m', x, y) \le C_3 - \epsilon_0 (m')^{2/3}.  
		\end{equation}
		Therefore to complete the proof, we just need to bound the deterministic term
		$$
		\E A^n[(0, m') \to (y, 1)] + \E \underline{A}^n[(0, m') \to (x, 1)].
		$$
		By Brownian scaling $\E WB^n_i(t) = c_{n, i} \sqrt{t}$, where $c_{n, 1} \ge \cdots \ge c_{n, n}$ are the expected values of the eigenvalues in an $n \times n$ GUE. Using this and translating to expectations for $A^n$, we have:
		\begin{align*}
			&\E A^n[(0, m') \to (y, 1)] + \E \underline{A}^n[(0, m') \to (x, 1)] \\
			= &n^{-1/3}[c_{n, 1}(\sqrt{n + 2 n^{2/3}y} - \sqrt{n}) - 2 y n^{2/3} +  c_{n, m'}(\sqrt{n- 2 n^{2/3}x}- \sqrt{n}) + 2 x n^{2/3}]. 
		\end{align*}
		Now, $\sqrt{n + 2 y n^{2/3}} - \sqrt{n} \le n^{1/6} y$ for all $y$. Using this bound, the above display is bounded above by
		$$
		|c_{n, 1} - 2\sqrt{n}| n^{-1/6} y + 2 |2\sqrt{n} - c_{n, m'}| n^{-1/6} x.
		$$ 
		The lemma then follows from \eqref{E:supxy} and the mean estimate $c_{n, m'} = 2\sqrt{n} - O(n^{-1/6} (m')^{2/3})$ from Lemma \ref{L:GUE-tail}, where the constant in the big-$O$ term is uniform in $n, m'$.
	\end{proof}

	\section{Applications to the directed landscape}
	\label{S:applications}
	
	In this section, we prove our two applications of the DtM isometry in the directed landscape. We start by characterizing the maximizing index $k$ in \eqref{E:DtM_AirySheet_simple}.
	
	\subsection{Geometric Characterization of the Turnaround Index} \label{SS:turnaround-geometry}
	
	In the single-path case of Theorem \ref{T:shadow-sheet}, the index set $I$ in Theorem \ref{T:shadow-sheet-simple} will be a singleton $\{ k\}$ for some $k \in \N$. We refer to this index $k$ as the \textbf{turnaround index}. The turnaround index can be understood via coalescence in the directed landscape, which we now introduce formally.
	
	The directed landscape $\cL$ is a random continuous function from the parameter space 
	$$
	\R^4_\uparrow = \{u = (p; q) = (x, s; y, t) \in \R^4 : s < t\}
	$$
	to $\R$. 
	As with LPP, the value $\cL(p; q) = \cL(x, s; y, t)$ is best thought of as a distance between two points $p$ and $q$, and similar to LPP, we only assign distances between points in the correct order (in this case, requiring ordering on the time coordinates). Again similarly to LPP, the directed landscape satisfies a reverse triangle inequality,
	\begin{equation}
		\label{E:triangle}
		\cL(p; r) \ge \cL(p; q) + \cL(q; r) \qquad \mbox{ for all }(p; r), (p; q), (q; r) \in \mathbb R^4_\uparrow,
	\end{equation}
	which allows us to define path length by subdivision. For a function $\pi: [s, t] \to \R$, henceforth a \textbf{path}, define its length
	\begin{align*}
		|| \pi ||_{\Lag} &= \inf_{s = t_0 < t_1 < t_2 < \ldots < t_n = t}\sum_{i = 1}^{n} \Lag(\pi(t_{i - 1}) , t_{i - 1} ; \pi(t_i), t_i). 
	\end{align*}
	Note that almost surely any deterministic path satisfies $\|\pi \|_{\Lag} = - \infty$, and therefore paths of finite length depend on the random geometry $\Lag$. For $x, y \in \R$ and $s < t$, we say a path $\pi: [s, t] \to \R$ with endpoints $\pi(s) = x$, $\pi(t) = y$ is a \textbf{geodesic} if $\|\pi\|_\cL = \cL(x, s; y, t)$. Equivalently, for any other path $\tilde{\pi} : [s, t] \to \R$ with the same endpoints,
	\begin{align*}
		\|\pi \|_{\Lag} \geq  \|\tilde{\pi} \|_{\Lag}.
	\end{align*}
	Next, for $\mathbf{x}, \mathbf{y} \in \R^k_\le$, we say a $k$-tuple of paths $\mathbf{\pi} = (\pi_1, \ldots, \pi_k)$ is a \textbf{(disjoint) $k$-tuple} from $(\mathbf{x}, s)$ to $(\mathbf{y}, t)$ if $\pi_i(s) = x_i$, $\pi_i(t) = y_i$ and   $\pi_{i}(u) < \pi_{i + 1}(u)$ for all $u \in (s, t)$. We define the extended landscape value
	\begin{equation}
		\label{E:Lxy-11}
		\cL(\bx, s; \by, t) := \sup_{\pi} \sum_{i=1}^k \| \pi_i\|_{\Lag},
	\end{equation}
	where the supremum is over all disjoint $k$-tuples $\pi$ from $(\mathbf{x}, s)$ to $(\mathbf{y}, t)$. We say that $\pi$ is a \textbf{(disjoint) optimizer} if it attains this supremum. The structure of disjoint optimizers in the directed landscape was investigated thoroughly in \cite{dauvergne2021disjoint}. The following theorem summarizes the results from that paper that we will need.
	\begin{thm}
		\label{T:dauvergne2021disjoint}
		\begin{enumerate}
			\item Almost surely, for all $k \in \N, s < t \in \R$ and $\bx, \by \in \R^k_\le$, there exists an optimizer in $\cL$ from $(\bx, s)$ to $(\by, t)$. Moreover, among the set of optimizers from $(\bx, s)$ to $(\by, t)$, there is a rightmost one.
			\item For fixed $k \in \N, s < t \in \R$ and $\bx, \by \in \R^k_\le$, almost surely there exists a unique disjoint optimizer  from $(\bx, s)$ to $(\by, t)$.
			\item The process $\cS(\bx, \by) :=\cL(\bx, 0; \by, 1)$ is an extended Airy sheet, in the sense that it is a limit-in-law of the extended Brownian LPP sheet $(\bx, \by)  \mapsto \cS_n(\bx, \by)$.
		\end{enumerate}
	\end{thm}
	Now, by part $3$ of the above theorem and Theorem \ref{T:shadow-sheet}, we can couple $\cL$ to a parabolic Airy line ensemble $\cA$ so that Theorem \ref{T:shadow-sheet}(iii) holds with $\cS(\bx, \by) :=\cL(\bx, 0; \by, 1)$. In the remainder of this section, we work with this coupling. 
	
	To talk about LPP across the Airy line ensemble, it will be useful to introduce \textbf{switchback paths}. For $-x \le 0 \le y$, a function $\pi:[-x, y] \to \N$ is a \textbf{switchback path} if $\pi$ is nondecreasing and left-continuous on $[-x, 0]$ and nonincreasing and right-continuous on $[0, y]$. For $u\in[0,x]$, set $\tilde\pi(u)=\pi(-u)$. We define its length 
	$$
	\|\pi\|_\cA = \|\tilde \pi|_{[0, x]}\|_{\tilde \cA} +  \cA_{\pi(0)}(0) + \|\pi|_{[0, y]}\|_\cA. 
	$$
	For $\bx, \by \in [0, \infty)^k_\le$, a disjoint $k$-tuple $\pi = (\pi_1, \dots, \pi_k)$ of switchback paths from $(-\bx, 1)$ to $(\by, 1)$ is such that each $\pi_i$ is a switchback path on $[-x_i, y_i]$, and $\pi_i < \pi_{i+1}$ for all $i$. With this definition, we can reformulate \eqref{E:DtM_AirySheet} as saying that
	\begin{equation}
		\label{E:Sxy11}
		\cS(\bx, \by) = \sup_\pi \sum_{i=1}^k \|\pi_i\|_\cA,
	\end{equation}
	where the supremum is over all disjoint $k$-tuples of switchback paths from $-\bx$ to $\by$. Since the maximum in \eqref{E:DtM_AirySheet} is attained, this supremum is a maximum, and we call a maximizer an Airy optimizer (or geodesic when $k = 1$).

	\begin{lemma} \label{lemma:coalescence_equivalence}
		In the above coupling, the following holds almost surely for all $k, \ell \in \N$, $\mathbf{x}, \mathbf{y} \in [0, \infty)^{k}_{\leq}$, and $\mathbf{x}', \mathbf{y}'  \in [0, \infty)^{\ell}_{\leq}$ with $x_k \le x_1'$ and $y_k \le y_1'$.
		
		There exist optimizers $\pi, \pi'$ in $\cL$ from $(\mathbf{x}, 0) \to (\mathbf{y}, 1)$ and $(\mathbf{x}', 0) \to (\mathbf{y}', 1)$ such that $(\pi, \pi')$ is a disjoint $(k + \ell)$-tuple if and only if there are Airy optimizers $\tau, \tau'$ from $-\mathbf{x}$ to $\mathbf{y}$ and $-\mathbf{x}'$ to $\mathbf{y}'$ such that $(\tau, \tau')$ is a disjoint $(k + \ell)$-tuple.
	\end{lemma}
	
	\begin{proof}
		The above immediately follows after noticing that both events are equal to the event where
		\begin{align*}
			\mathcal{S}(\mathbf{x}, \mathbf{y}) + \mathcal{S}(\mathbf{x}', \mathbf{y}') &= \mathcal{S}( \{\mathbf{x}  , \mathbf{x}' \} , \{ \mathbf{y} , \mathbf{y}' \} ).
		\end{align*}
		Note that this uses that the suprema in \eqref{E:Lxy-11} and \eqref{E:Sxy11} are always attained. 
	\end{proof}

	\begin{defn} \label{defn:turnaround_index}
		For $x, y > 0$, define the \textbf{turnaround index} $k = k(x, y) \in \N$ to be the largest positive integer so that 
		\begin{align}
			\label{E:Sxy}
			\mathcal{S}(x, y) &= \tilde{\mathcal{A}}[(0, k) \to (x, 1)] + \mathcal{A} [(0, k) \to (y, 1)] + \mathcal{A}_k(0).
		\end{align}
	\end{defn}
	Note that it is not difficult to check (i.e. by Brownian Gibbs resampling) that for fixed $x, y$, almost surely there is a unique $k$ such that \eqref{E:Sxy} holds. See \cite[Lemma B.1]{Ham19d} and \cite[Lemmas 2.15--2.16]{dauvergne2021disjoint} for similar proofs in the setting of Brownian LPP. We need to define \eqref{E:Sxy} using the maximal $k$ to take care of exceptional $x, y$ with multiple maximizers.

	The following theorem shows that the turnaround may be characterized by a coalescence problem in $\Lag$. 
	
	\begin{thm} \label{T:k_geom}
		For $\ell \in \N$, let $\tau^\ell = (\tau^\ell_1, \dots, \tau^\ell_\ell)$ be the (almost surely unique) optimizer in $\cL$ from $(0^{\ell}, 0)$ to $(0^{\ell}, 1)$ where $0^{\ell} = (0, 0, \ldots, 0)$ is an $\ell$-tuple (and for $\ell = 0$, set $\tau^0 \defeq \emptyset$). In addition, for $x, y > 0$, let $\pi_x^y$ be the rightmost geodesic from $(x,0)$ to $(y , 1)$. Then almost surely for any $x, y > 0$,
		\begin{align*}
			k(x, y) = \min \{ \ell \in \N \, | \, \tau_\ell^\ell(r) = \pi_x^y(r) \text{ for some } r \in [0, 1]\}. 
		\end{align*}
		
	\end{thm}
	
	\begin{remark}
		Although we only state and prove the result for the Airy sheet, a version of this characterization of the turnaround index holds for any last-passage percolation model.
	\end{remark}
	
	\begin{proof}
		By Theorem \ref{T:dauvergne2021disjoint}.2, we may restrict to the almost sure event where for each $\ell \in \N$, there exists a unique disjoint optimizer from $(0^{\ell}, 0)$ to $(0^{\ell}, 1)$. Since $\pi_x^y$ is rightmost, we have that $\tau_\ell^\ell(r) < \pi_x^y(r)$ for all $r \in [0, 1]$ if and only if
		\begin{align}
			\label{E:l-split}
			\mathcal{S} (0^{\ell}, 0^{\ell}) + \mathcal{S} (x, y) &= \mathcal{S} \left( \{ 0^{\ell} , x\} , \{ 0^{\ell} , y\} \right).
		\end{align}
		Using that \eqref{E:Sxy11} is always attained, this is equivalent to the event in $\cA$ where there is a switchback geodesic $\pi'$ from $-x$ to $y$ which is disjoint from an optimizer from $0^\ell$ to $0^\ell$. Since the Airy lines are almost surely strictly ordered at all points, there is almost surely a unique Airy optimizer $\sigma$ from $0^\ell$ to $0^\ell$ given by $\sigma_i(0) = i$ for all $i$. Therefore the above display is equivalent to the event where there exists a switchback geodesic $\pi'$ from $-x$ to $y$ with $\pi'(0) \ge \ell + 1$. The maximal value of $\pi'(0)$ is precisely the turnaround index $k(x, y)$, and so \eqref{E:l-split} is equivalent to the event where $\ell < k(x, y)$. This yields the result.

	\end{proof}
	
	\begin{remark}
		Using this geometric definition of $k(x, y)$ and the fact that geodesic melon widths tend to $\infty$, one can prove Theorem \ref{T:shadow-sheet-simple}  using just the coalescing properties of the landscape. We sketch out this argument in Appendix \ref{sec:shadow_sheet_geomproof}. 
	\end{remark}

	\subsection{The Width of a Geodesic Melon in the Directed Landscape}
	
	Using our characterization of the turnaround index, we are able to obtain explicit estimates on the width of geodesic melons, proving Theorem \ref{T:melon_width}. As a technical input, we first bound the following last passage problem in the Airy line ensemble.

	\begin{lemma} \label{lemma:AiryLPP_bound} 
		Let $\{ \hat{\mathcal{A}}_{n} \}_{n \in \N}$ be the stationary Airy line ensemble given by $\hat{\cA}_n(x) = \mathcal{A}_n(x) + x^2$. There exist constants $c, d > 0$ so that for all $x > 0$, $k \in \N$ and $s \geq 0$
		\begin{align*}
			\prob\left(  \hmA [(0,k) \to (x, 1)] \geq 2 \sqrt{2 kx} + s \sqrt{x} k^{-1/6} \right)  & \leq c \exp \left( - d s^{3/2} \right)
		\end{align*}		
	\end{lemma}
	
	The proof of Lemma \ref{lemma:AiryLPP_bound} is a short application of the powerful results from \cite{wu2024applications}.
	
	\begin{lemma}[Special case of Proposition 2.9, \cite{wu2024applications}]
		\label{lemma:wu_convex}
		Fix $R = \II{k, \ell} \times[a, b] \subset \Z \times \R$, and let $(B_k, \ldots, B_{\ell})$ be independent Brownian motions on $[a, b]$ of diffusion coefficient $2$, all starting at $0$. Let $\cA^n$ be the approximation of the parabolic Airy line ensemble introduced in Lemma \ref{L:An-limit}. Let $G: C \left( R , \, \R \right) \to \R$ be a functional satisfying the following:
		\begin{itemize}
			\item $G$ only depends convexly on finitely many times. That is, there exists a finite set $\Pi \subset [a,b]$, and a convex function $g: \R^{\II{k, \ell} \times \Pi} \to \R$ so that $G= g \circ \operatorname{Res}_{\II{k, \ell} \times \Pi}$. Here $\operatorname{Res}_{\II{k, \ell} \times \Pi}$ is the restriction map given by $z \mapsto \{ z_{j}(t)\}_{(j, t) \in \II{k, \ell} \times \Pi}$. 
			\item $G$ is invariant under adding constants to individual lines:
			$G((f_i+c_i)_{i=k}^{\ell})=G(f)$ for every
			$(c_k,\ldots,c_\ell)\in\R^{\ell-k+1}$.
		\end{itemize}
		Then
		\begin{align*}
			\E [G(\cA^n - \E \cA^n)] &\leq \E [G (B)].
		\end{align*}
	\end{lemma}
	
	We would like to apply Lemma \ref{lemma:wu_convex} to the Airy line ensemble $\cA$, rather than $\cA^n$, by using that $\cA^n - \E \cA^n \implies \cA - \E \cA$. Here the convergence follows from Lemma \ref{L:An-limit} (the convergence of the expectations follows by uniform
	integrability from the GUE eigenvalue-counting estimates
	in \cite[Lemmas 2.2--2.3]{gustavsson2005gaussian}
	and the tail bounds in \cite{ledoux2010small}). All our functions $G$ are non-negative, and so the conclusion of the lemma then holds in the $n \to \infty$ limit by Fatou's lemma.
	
	\begin{proof}[Proof of Lemma \ref{lemma:AiryLPP_bound}]
		Let $\pi: (0, k) \to (x, 1)$ be any upright path. Notice for any such path, the following functional is convex and depends on finitely many times:
		\begin{align*}
			G_{\pi}(f) &= f(\pi) \defeq \sum_{i = 1}^{k} f_i(t_{i - 1}) - f_i(t_{i}),
		\end{align*}
		where in the above, $0 = t_{k} \leq \cdots \leq t_0 = x$ are the jump times of $\pi$. Each $G_\pi$ is invariant under adding constants to individual
		lines, since every path weight is a sum of increments on each
		line. 
		Since the supremum of a family of convex functions is again a convex function, for any finite set $\Pi \subset [0, x]$, the function
		\begin{align*}
			G_{\Pi}(f) &= \sup \{ f(\pi) \, | \, \textrm{ jump times of } \pi  \subset \Pi \}
		\end{align*}
		satisfies Lemma \ref{lemma:wu_convex}.
		Furthermore, since the exponential of a convex function is also convex for any $\lambda \geq 0$, the function $\exp \left( \lambda G_{\Pi}(f) \right)$ also satisfies this lemma. Since $\exp( \lambda G_{\Pi})$ is non-negative, we can apply Lemma \ref{lemma:wu_convex} to the Airy line ensemble with this function. 
		
		Letting $\Pi_n \defeq 2^{-n} \Z \cap [0, x] \cup \{ x\}$, we have
		\begin{align*}
			\E \left[ \exp \left( \lambda G_{\Pi_n}(\cA - \E \cA) \right)  \right] & \leq \E \left[ \exp \left( \lambda G_{\Pi_n}(B) \right)  \right].
		\end{align*}
		As the sequence $G_{\Pi_n}(f)$ is increasing in $n$, by the monotone convergence theorem we have 
		\begin{align*}
			\E \left[ \exp \left( \lambda \cdot (\cA - \E \cA)[(0, k) \to (x, 1)]  \right)  \right] & \leq 	\E \left[ \exp \left( \lambda \cdot B [(0, k) \to (x, 1)]  \right)  \right],
		\end{align*}
		where in the above we have used the fact that for any $f \in C(\II{1, k} \times  [0, x] , \R)$,
		\begin{align*}
			\lim_{n \to \infty} G_{\Pi_n}(f) = f [(0, k) \to (x, 1)]. 
		\end{align*}
		At this point, we can replace $\cA - \E \cA$ by the stationary Airy line ensemble $\hmA$, noting that LPP values are invariant under constant shifts of the lines.
		The proof then follows from Markov's inequality. Define
		\begin{align*}
			Z_{\hmA} &= \frac{\hmA[(0,k ) \to (x, 1)]  - 2 \sqrt{2kx}}{\sqrt{x} k^{-1/6}} , \qquad 
			Z_{B} = \frac{B[(0,k ) \to (x, 1)]  - 2 \sqrt{2kx}}{\sqrt{x} k^{-1/6}}  
		\end{align*}
		
		Then
		\begin{align} \label{E:ldp_estimate}
			\prob\left(  \hmA [(0,k) \to (x, 1)] \geq 2 \sqrt{2kx} + s \sqrt{x} k^{-1/6} \right) & \leq \inf_{\lambda \geq 0} e^{- \lambda s} \E e^{\lambda Z_{\hmA}}  \leq \inf_{\lambda \geq 0} e^{-\lambda s} \E e^{\lambda Z_{B}}. 
		\end{align}		
		By Theorem \ref{lemma:Brownian_Burke}, we have that $B[(0, k) \to (x, 1)]\stackrel{d}{=} W^{(k)}_1(x)$. Hence, we may apply the GUE tail bounds from \cite{ledoux2010small} to bound the Laplace Transform of $Z_B$,
		\begin{align*}
			\E e^{\lambda Z_B} & \leq 1 + \lambda \int_{0}^{\infty} e^{\lambda t } \prob \left( Z_B > t \right) \, dt  \leq 1 + c \lambda \int_{0}^{\infty} e^{\lambda t - d t^{3/2}} \, dt  \leq c e^{d \lambda^{3}},
		\end{align*}
		where in the above we allow $c, d$ to change line by line. Plugging this back into \eqref{E:ldp_estimate} and optimizing over $\lambda$ concludes the proof. 
	\end{proof}
	
	To prove Theorem \ref{T:melon_width}, we also need standard bounds on the Airy point locations and geodesic deviations.
	
	\begin{lemma}[Corollary 4.3, \cite{dauvergne2021bulk}]
		\label{L:bulk-lemma}
		Let $\kappa = (3 \pi/2)^{2/3}$. There exist $c, b >0$ such that for all $i \in \mathbb{N}$ and $m > b\log(i +1)$, we
		have that
		$$
		\mathbb P(|\cA_i(0)+ \kappa i^{2/3}| \ge m i^{-1/3}) \le c e^{-m/5}.
		$$
	\end{lemma}
	
	\begin{lemma}[part of Proposition 12.3, \cite{dauvergne2022directed}]
		\label{L:geo-bound}
		Let $\pi$ be the (almost surely unique) geodesic in $\cL$ from $(0,0)$ to $(0, 1)$. Then there are constants $c, d > 0$ such that for all $m > 0$, 
		$$
		\P(\sup_{t \in [0, 1]} |\pi(t)| > m) \le c e^{-dm^3}.
		$$
	\end{lemma}

	\begin{proof}[Proof of Theorem \ref{T:melon_width}]
		We begin with the proof of the upper bound, which is easier.  Let $\pi_{\ell}=(\pi_{\ell,1},\ldots,\pi_{\ell,\ell}) \defeq \pi_{0^{\ell}}^{0^{\ell}}$ be the disjoint optimizer from $(0^{\ell}, 0)$ to $(0^{\ell}, 1)$. Thus $\pi_{\ell,\ell}$ is its rightmost component. We work on the almost sure event where $\pi_{\ell}$ is unique for each $\ell \in \N$. Fix a constant $C > 2^{-1/2} (3 \pi / 2)^{1/3}$ and let $\tau_{\ell}$ be the rightmost geodesic from $(C \ell^{1/3} , 0)$ to $(C \ell^{1/3}, 1)$. It then follows that
		\begin{align*}
			\prob \left( \max_{t \in [0, 1]} \pi_{\ell,\ell}(t) \leq 2 C \ell^{1/3} \right) & \geq 1  - \prob \left(  \pi_{\ell} \cap \tau_{\ell} \neq \emptyset  \right) - \prob \left( \max_{t \in [0, 1]} \tau_{\ell}(t) \geq 2 C \ell^{1/3} \right) \\
			& \geq 1  - \prob \left( \pi_{\ell} \cap \tau_{\ell} \neq \emptyset \right) - 2e^{-d \ell},
		\end{align*}
		where in the last step we used Lemma \ref{L:geo-bound}. Hence, to conclude the upper bound on the optimizer width, it suffices to bound $\prob (\pi_{\ell} \cap \tau_{\ell} \neq \emptyset)$. By Theorem \ref{T:k_geom},
		\begin{align} \label{E:k_sum}
			\prob \left( \pi_{\ell} \cap \tau_{\ell} \neq \emptyset \right)  &= \prob \left(  k (C \ell^{1/3}  , C \ell^{1/3}) \leq \ell \right).
		\end{align}
		To bound the right-hand side above, we use that by \eqref{E:DtM_AirySheet_simple}, on the event $k(x, y) \le \ell$, we have
		\begin{align*}
			\cS(x, y) &= \max_{k \le \ell} \tilde{\mathcal{A}}[(0,k) \to (x, 1)] + \mathcal{A}_k(0) + \mathcal{A}[(0, k) \to (y, 1)] \\
			&\le \max_{k \le \ell} \mathcal{A}_1(-x) + \mathcal  A_1(y) - \mathcal{A}_k(0) \\
			&= \mathcal{A}_1(-x) + \mathcal  A_1(y) - \mathcal{A}_\ell(0).
		\end{align*}
		Here in the inequality, we have used that for any environment $f$ with $f_1 > f_2 > \cdots > f_k$, $f[(0, k) \to (x, 1)] \le f_1(x) - f_k(0)$.
		Now, using that $\cS(C \ell^{1/3}, C \ell^{1/3}) \stackrel{d}{=} \cA_1(0)$, and appealing to the tail bounds in Lemma \ref{L:bulk-lemma}, and standard Tracy-Widom tail bounds and the stationarity of $\cA(x) + x^2$, we have
		\begin{align*}
			\prob \left(\cS(C \ell^{1/3}, C \ell^{1/3}) \le \mathcal{A}_1(-C \ell^{1/3}) + \mathcal  A_1(C \ell^{1/3}) - \mathcal{A}_\ell(0)  \right) &\leq ce^{-d\ell}.
		\end{align*}

		Now we proceed with the lower bound. Fix $\epsilon > 0$ and let $\sigma_{\ell}$ be the rightmost geodesic from $( 2\epsilon \ell^{1/3}    , 0 )$ to $(2\epsilon \ell^{1/3}   , 1)$. Then 
		\begin{align*}
			\prob \left(  \max_{t \in [0, 1]}  \pi_{\ell,\ell}(t) \geq \epsilon  \ell^{1/3}   \right) & \geq 1 - \prob \left(  \pi_{\ell} \cap \sigma_{\ell}    = \emptyset \right) - \prob \left(  \min_{t \in [0, 1]} \sigma_{\ell}(t)   \leq \epsilon \ell^{1/3}   \right) \\
			&\geq 1 - \prob \left(  \pi_{\ell} \cap \sigma_{\ell}   = \emptyset \right) - 2 e^{-d \ell}
		\end{align*}
		where again in the last step we used Lemma \ref{L:geo-bound}. By Theorem \ref{T:k_geom},
		\begin{align*}
			\left\{ \pi_{\ell} \cap \sigma_{\ell}   = \emptyset  \right\}  
			&= \{ k(2 \epsilon \ell^{1/3}   , 2 \epsilon \ell^{1/3}   ) > \ell \}.
		\end{align*}
		Now, 
		\begin{align*}
			\P(  k(2 \epsilon \ell^{1/3}&   , 2 \epsilon \ell^{1/3}) > \ell) \\
			&\le \P(\cS(2 \epsilon \ell^{1/3}   , 2 \epsilon \ell^{1/3}) < - \ell^{2/3}) + \sum_{j = \ell + 1}^\infty \P(\cS(2 \epsilon \ell^{1/3}   , 2 \epsilon \ell^{1/3}) > - \ell^{2/3}, k(2 \epsilon \ell^{1/3}   , 2 \epsilon \ell^{1/3}) = j) \\      
			&\le c e^{-d\ell^2} + \sum_{j = \ell + 1}^\infty \P(\mathcal{A}_j(0) + \mathcal{A}[(0, j) \to (2 \epsilon \ell^{1/3}  , 1) ]  + \tilde{\mathcal{A}}[(0, j) \to (2 \epsilon \ell^{1/3}   , 1) ]    \geq - \ell^{2/3}), 
		\end{align*}
		where the inequality uses \eqref{E:DtM_AirySheet_simple} and Tracy-Widom lower tail bound. Since $\mathcal{A}[(0, j) \to (2 \epsilon \ell^{1/3}  , 1) ]  \stackrel{d}{=} \tilde{\mathcal{A}}[(0, j) \to (2 \epsilon \ell^{1/3}   , 1) ]$, each of the probabilities in the above sum is bounded above by
		$$
		\P(\cA_j(0) > -2j^{2/3})
		+2\P(\mathcal{A}[(0,j)\to(2\epsilon\ell^{1/3},1)]
		>j^{2/3}/2)
		\le c\exp(-dj).
		$$
		where the bound uses Lemma \ref{lemma:AiryLPP_bound} and Lemma \ref{L:bulk-lemma}. Summing the above over $j \ge \ell +1$ yields the result.
	\end{proof}

	\subsection{Proof of Theorem \ref{thm:Hammond_conj}} \label{sec:Hammond_proof}
	
	In this section we prove Theorem \ref{thm:Hammond_conj}. The proof combines the DtM representation for the Airy sheet with Brownian Gibbs resampling for the Airy line ensemble. Before proceeding with the proof, we recall the Brownian Gibbs property.
	
	\begin{thm}[\cite{CH}]
		\label{T:Brownian-Gibbs}
		Let $\cA$ be the parabolic Airy line ensemble. For an index $k \in \N$ and an interval $[a, b]$, let $\mathcal F_{k, a, b}$ be the $\sigma$-algebra generated by $(\cA_i(x) : (i, x) \notin \II{1, k} \times [a, b])$. Then conditional on  $\mathcal F_{k, a, b}$, the law of $(\cA_1, \dots, \cA_k):[a, b] \to \R$ is equal to that of $k$ independent Brownian bridges $B_1, \dots, B_k:[a, b] \to \R$ with diffusivity $2$ and endpoints $B_i(a) = \cA_i(a), B_i(b) = \cA_i(b)$, conditioned on the non-intersection event
		$$
		B_1(t) > \cdots > B_k(t) > \cA_{k+1}(t), \quad t \in [a, b].
		$$
	\end{thm}
	
	Now, recall from the introduction that we let $M(\epsilon)$ denote the maximal number of disjoint geodesics in the directed landscape which start and end in the interval $[0, \epsilon]$ and go from time $0$ to $1$. By Lemma \ref{lemma:coalescence_equivalence}, the event $\{M(\epsilon) \ge k\}$ is equivalent to the event where there exist $\bx, \by \in [0, \epsilon]^k_<$ and Airy geodesics $\pi_i$ from $-x_i$ to $y_i$ such that $(\pi_1, \dots, \pi_k)$ is a disjoint $k$-tuple. In particular, letting $\epsilon_k = \frac{\epsilon}{k}(1, \ldots, k )$, we have that 
	$$
	\P(M(\epsilon) \ge k) \ge \P\left( \text{There exist disjoint Airy geodesics from $-\epsilon_{k, i}$ to $\epsilon_{k, i}$ for $i = 1, \dots, k$} \right).
	$$
	Call the event on the right-hand side above $E^{\epsilon, k}$.
	To prove the lower bound on $ \P(M(\epsilon) \ge k)$ in Theorem \ref{thm:Hammond_conj}, we will show that the right-hand side above is at least $c \epsilon^{(k^2 - 1)/2}$.
	
	\begin{figure}[htbp]
		\centering
		\includegraphics[width=0.8\textwidth]{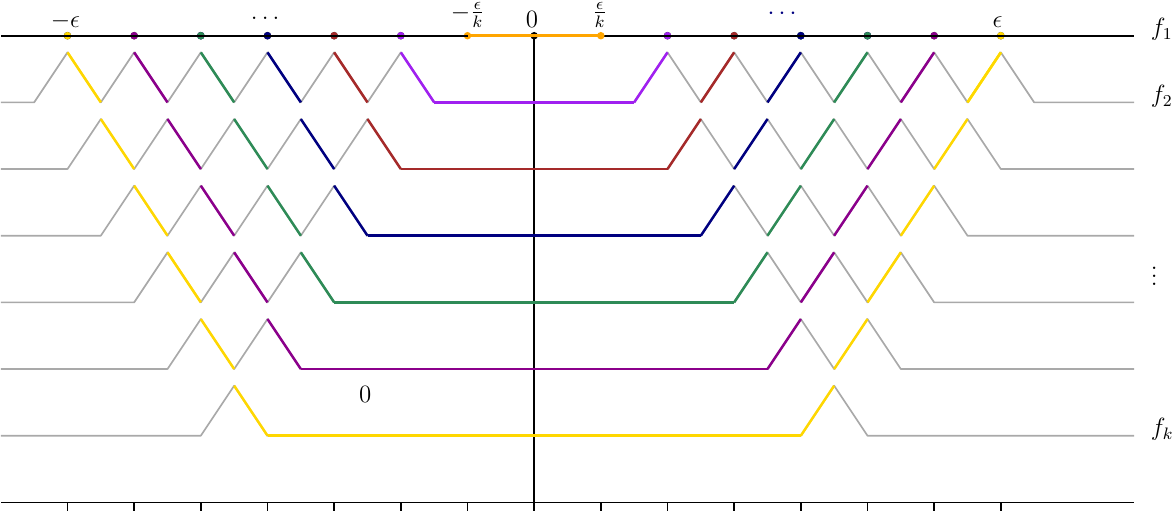}
		\caption{The functions and Airy optimizers given in the proof of Theorem \ref{thm:Hammond_conj}}
		\label{fig:nonint_paths_example}
	\end{figure}
	To that end, we define functions $f_1,\ldots,f_k:\R\to\R$ as in Figure \ref{fig:nonint_paths_example}. First, define the saw-tooth function
	$$
	g(x)=\min_{z\in 2\Z}|x-z|.
	$$
	Set $h_1\equiv 0$ and, for $i=2,\ldots,k$, define
	$$
	h_i(x)
	=
	1-i+g(|x|+i)\mathbf{1}\{|x|\in[i,2k+2-i]\}.
	$$
	Finally, setting $s=\epsilon/(2k)$, define
	$$
	f_i(x)=s^{1/2}h_i(x/s),
	\qquad i=1,\ldots,k.
	$$
	
	Now let $\hat{\mathcal{A}}_i(x) = \mathcal{A}_i(x) - \mathcal{A}_1(x)$, and for $a, b\in \R$ consider the events
	\begin{align*}
		T^a &= \left\{ \mathcal{A}_1(0) \geq a  \right\} ,\\
		B^{\epsilon}_k &= \left\{  \max_{i = 2, \ldots , k} \max_{x \in [-\epsilon, \epsilon]}  | \hat{\mathcal{A}}_i(x) -  f_i(x)| \leq \alpha \epsilon^{1/2} \right\}
		\cap   \left\{  \sup_{x, y \in [-\epsilon, \epsilon]}  |\mathcal{A}_{1}(y) - \mathcal{A}_1(x)| \leq \alpha \epsilon^{1/2} \right\} , \\
		G_k^b &= \left\{ \sup_{j > k} \sup_{x, y \in [0, 1]} \mathcal{A}_j(0) + \mathcal{A}[(0, j) \to (y, k+1)] + \tilde{\mathcal{A}}[(0, j) \to (x, k + 1)]  \leq b \right\}  , \\
		A^{a, b}_{\epsilon, k} &=   T^a \cap B^{\epsilon}_k   \cap G_k^b. 
	\end{align*}
	
	\begin{lemma}
		\label{L:ab1}
		Suppose $a - b > 1$. There exists $\alpha_0, \epsilon_0 > 0$ depending only on $k$ such that for all $0 <\alpha < \alpha_0$ and $0 < \epsilon < \epsilon_0$ we have
		\begin{align*}
			A^{a, b}_{\epsilon, k} \subset E^{\epsilon, k}. 
		\end{align*} 
	\end{lemma}
	
	\begin{proof}
		Throughout the proof we assume $\epsilon, \alpha$ are sufficiently small. First, on the event $A^{a, b}_{\epsilon, k}$, the turnaround index for any Airy geodesic starting at $-x \in [-\epsilon, 0]$ and ending at $y \in [0, \epsilon]$ is at most $k$. Indeed, for this we just need to check that for any switchback path $\pi$ with $\pi(0) = j \ge k + 1$, that 
		$$
		\|\pi\|_\cA < a - 1/2 < \cA_1(-x) + \cA_1(y) - \cA_1(0),
		$$
		since the right-hand side above is the length of the switchback path that stays on the top line. Here the second inequality uses the regularity of $\cA_1$ in the event $B^\epsilon_k$ and that $\cA_1(0) \ge a$.
		By splitting up $\pi$ into pieces from $(-x, 1)$ to line $k$, from line $k + 1$ down to line $j$ and back, and again from line $k$ to $(y, 1)$, and using the definition of $G^b_k$, we can write
		\begin{align*}
			\|\pi\|_\cA \le b + \sup_{\substack{w,x,y,z\in[0,\epsilon]\\w\leq y,\ z\leq x}} \mathcal{A}[(w, k) \to (y, 1)] + \tilde{\mathcal{A}}[(z, k) \to (x, 1)].  
		\end{align*}
		The event $B^\epsilon_k$ guarantees that the second term above is at most $1/2$ for small $\epsilon$. Since $a - b > 1$, the result follows.
		
		From here it is easy to see that for small enough $\alpha$, any Airy geodesics from $-\epsilon_{k, i}$ to $\epsilon_{k, i}$ are disjoint for $i = 1, \dots, k$. By the above, these geodesics only use the top $k$ lines, and so this is a deterministic statement about functions satisfying the bounds in $B^\epsilon_k$. Next, it is easy to check (i.e. by examining Figure \ref{fig:nonint_paths_example}) that this is true when $\alpha = 0$, when the Airy paths equal the functions $f_i$ and geodesics are uniquely obtained and disjoint. By a continuity argument, this also holds for all small enough $\alpha$. Finally, this problem is scale-invariant in $\epsilon$, and so this holds for all small enough $\alpha$, uniformly over $\epsilon > 0$ as desired.
	\end{proof}
	
	The theorem now follows from the following bound.
	
	\begin{lemma} \label{lemma:prob_lwr_bnd}
		We can find $a > b + 1$ and $\alpha > 0$ such that for all $\epsilon > 0$ small enough, 
		\begin{align*}
			\prob (A^{a, b}_{\epsilon , k}) \geq C \epsilon^{(k^2 - 1)/2},
		\end{align*}
		where $C > 0$ is a constant depending on $a, b, \alpha$, and $k$.
	\end{lemma}
	
	The main contribution to the rarity of this event is due to enforcing the gaps between the first $k$ Airy lines to lie within a window of size $\mathcal{O} (\epsilon^{1/2})$. This should scale according to the Karlin-McGregor formula for Dyson Brownian Motion which, to leading order, is given by
	\begin{align*}
		\prod_{i < j} (\epsilon^{1/2})^{2}  \times \textrm{Vol} \sim \epsilon^{\frac{k (k - 1)}{2}} \times \epsilon^{\frac{k - 1}{2}},
	\end{align*}
	The first term above follows from a lower bound on the probability density and the volume term corresponds to the Lebesgue measure of the window.

	\begin{proof}[Proof of Lemma \ref{lemma:prob_lwr_bnd}]

		Fix $k \in \N$ and let $\mathcal{F}_{\textrm{ext}} = \sigma \left\{   \mathcal{A}_i(x) : \, (i, x) \not\in \{ 1, \ldots, k\} \times (-1, 1) \right\}$. For $\sigma >1$ and  $\delta > 0$, let $F_k^{\sigma, \delta}$ be the intersection of the following four $\mathcal{F}_{\textrm{ext}}$-measurable events:
		\begin{align*}
			&\max_{j = 1, \ldots, k, t = \pm 1}  |\mathcal{A}_j(t) - \E \mathcal{A}_{j}(1)|   \leq  \sigma , \quad \max_{t \in [- 1, 1]}  |\mathcal{A}_{k+1}(t) - \E \mathcal{A}_{k+1}(0)|   \leq  \sigma, \\
			&\min_{j = 1, \ldots, k, t = \pm 1} \mathcal{A}_j(t) - \cA_{j+1}(t) \ge 2 \delta, \quad
			\max_{t \in [1 - \sigma^{-1},1], \iota = \pm 1}|\mathcal{A}_{k+1}(\iota t) - \mathcal{A}_{k+1}(\iota)|   \leq  \delta.
		\end{align*}
		It suffices to show for some choice of $a > b + 1$ and $\sigma, \delta > 0$ that
		\begin{align*}
			\prob (A^{a, b}_{\epsilon, k} \cap F_k^{\sigma, \delta}) & \geq C \epsilon^{\frac{k^2 - 1}{2}}. 
		\end{align*}
		First, observe that by continuity and strict ordering of $\cA$, $\P(F_k^{\sigma, \delta}) > 0$ for all small enough $\delta > 0$ and all large enough $\sigma > 0$ given $\delta$. For the remainder of the proof, we will fix $\sigma, \delta > 0$ so that this is the case. Next, observe that $\P(G^b_k) \to 1$ as $b \to \infty$.
		For this, observe that $G^b_k$ is $\mathcal{F}_{\ext}$-measurable and contains the event where
		\begin{align*}
			\left\{ \sup_{x, y\in [0, 1]}  \mathcal{S}(x, y) \leq b  \right\} ,
		\end{align*}
		since on the above event,
		\begin{align*}
			&\sup_{j \geq k + 1} \max_{x, y \in [0, 1]} \mathcal{A}_j(0) + \mathcal{A}[(0, j) \to (y, k+1)] + \tilde{\mathcal{A}}[(0, j) \to (x, k+1)]   \\
			&\leq \sup_{j \in \N} \max_{x, y \in [0, 1]} \mathcal{A}_j(0) + \mathcal{A}[(0, j) \to (y, 1)] + \tilde{\mathcal{A}}[(0, j) \to (x, 1)]   \\
			&=  \max_{x, y \in [0, 1]} \mathcal{S}(x, y) \leq b.
		\end{align*}
		Hence, we may pick $b$ large enough so that $\P(G^b_k \cap F_k^{\sigma, \delta}) > 0$. For this choice of $b, \sigma, \delta$, to complete the proof we must show that
		\begin{align*}
			\prob (A^{a, b}_{\epsilon, k} \mid F_k^{\sigma, \delta} \cap G^b_k) = \prob (T^a \cap B^\epsilon_k \mid F_k^{\sigma, \delta} \cap G^b_k)  \geq C \epsilon^{\frac{k^2 - 1}{2}}. 
		\end{align*}
		Now fix some $0 < \eta < \alpha/10$ and define the event
		\begin{align*}
			C_{k,a}^{\epsilon} &= \left\{  \max_{i = 2, \ldots , k} \max_{t = \pm \epsilon}  |\hat{\mathcal{A}}_i(t) - f_i(t)|  \leq \eta \epsilon^{1/2} \right\} \cap \left\{  \mathcal{A}_1(- \epsilon) \in [a+1, a + 2] \right\} \cap \left\{  |\mathcal{A}_1(\epsilon) - \mathcal{A}_1(-\epsilon) |  \leq \eta \sqrt{\epsilon}\right\},
		\end{align*}
		and let $\mathcal P$ denote the conditional law of $\cA$ given $\mathcal F_\ext$. Then
		\begin{align} \label{E:conditioning_step}
			\E \mathcal{P} (B^{\epsilon}_k \cap T^a) \mathbbm{1}_{F_k^{\sigma, \delta} \cap G^b_k} & \geq \E \mathcal{P}(C_{k,a}^{\epsilon}) \mathcal{P}(B^{\epsilon}_k   \cap T^a \, | \, C_{k,a}^{\epsilon})\mathbbm{1}_{F_k^{\sigma, \delta} \cap G^b_k}.
		\end{align}
		By Brownian scaling and the Brownian Gibbs property, for large enough $a$ there exists a constant $c > 0$ depending on $k, \alpha, \sigma, \delta$ but not on $\epsilon$ such that
		\begin{align*}
			\mathcal{P} (B^{\epsilon}_k  \cap T^a \, | \, C_{k,a}^{\epsilon}) \mathbbm{1}_{F_k^{\sigma, \delta} \cap G^b_k} & \geq c \mathbbm{1}_{F_k^{\sigma, \delta} \cap G^b_k}.
		\end{align*}
		Indeed, given the event $C_{k,a}^{\epsilon}$ and $a$ large enough, on the interval $[-\epsilon, \epsilon]$, the event $B^\epsilon_k$ automatically guarantees that $\cA_{k}(x) > \cA_{k+1}(x)$ for $x \in [-\epsilon, \epsilon]$, so the probability of this event is bounded below by the case of no lower boundary. It then follows by Brownian scaling that this probability is bounded below by a strictly positive constant. 
		
		It remains to bound the probability of the event $C_{k,a}^{\epsilon}$ on the event $F_k^{\sigma, \delta}$. Let $\mathbb{P}_{\mathrm{NI}}^{\mathbf{x}, \mathbf{y}}$ be the law of $k$ non-intersecting Brownian Motions $B_1, \ldots, B_k$ on the interval $[-1, 1]$ starting at $\mathbf{x}$ and ending at $\mathbf{y}$. Letting $x = (\mathcal{A}_1(-1), \ldots, \mathcal{A}_k(-1))$ and $y = (\mathcal{A}_1(1), \ldots, \mathcal{A}_k(1))$, and $H = \{ B_k(s) \geq \mathcal{A}_{k + 1}(s) \textrm{ for all } s \in [-1, 1]\}$, we find that 
		\begin{align*}
			\E \mathcal{P} (C_{k,a}^{\epsilon}) \mathbbm{1}_{F_k^{\sigma, \delta}} &\geq \E \mathbb{P}_{\mathrm{NI}}^{\mathbf{x}, \mathbf{y}} \left( C_{k,a}^{\epsilon} \cap H  \right)  \mathbbm{1}_{F_k^{\sigma, \delta}} 
			= \E \mathbb{P}_{\mathrm{NI}}^{\mathbf{x}, \mathbf{y}} \left( C_{k,a}^{\epsilon} \right) \mathbb{P}_{\mathrm{NI}}^{\mathbf{x},\mathbf{y}} \left( H \, |\,  C_{k,a}^{\epsilon} \right)  \mathbbm{1}_{F_k^{\sigma, \delta}} . 
		\end{align*}
		Now, on $F_k^{\sigma, \delta}$, there exists an $\epsilon$-independent constant $d$ so that 
		\begin{align*}
			\mathbb{P}_{\mathrm{NI}}^{\mathbf{x},\mathbf{y}} \left( H \, |\,  C_{k,a}^{\epsilon} \right)  \mathbbm{1}_{F_k^{\sigma, \delta}} & \geq d.
		\end{align*}
		
		We may bound the probability of $C_{k,a}^{\epsilon}$ using the Karlin-McGregor formula and the Markov Property for Dyson Brownian Motion. For $k$ tuples $\mathbf{u}, \mathbf{v} \in \R^{k}_{>}$ and $t > 0$, define $p_t(\mathbf{u}, \mathbf{v})$ as
		\begin{align*}
			p_t(\mathbf{u}, \mathbf{v}) &= \det(p_t (u_i, v_j))_{i, j = 1}^k
		\end{align*}
		where $p_t(u_i, v_j) = \frac{1}{\sqrt{4 \pi t}} e^{-(u_i - v_j)^2/(4t)}$. Define the sets $U, V \subset \R^k$ as
		\begin{align*}
			U &= \left\{\mathbf u\in\R^k:
			u_1\in[a+1,a+2],\
			|u_i-u_1-f_i(\epsilon)|
			\leq \frac{\eta}{4}\sqrt{\epsilon}
			\quad\forall i=2,\ldots,k\right\},\\
			V &= \left\{\mathbf v\in\R^k:
			|v_i|\leq \frac{\eta}{4}\sqrt{\epsilon}
			\quad\forall i=1,\ldots,k\right\}.
		\end{align*}
		Then on the event $F_k^{\sigma, \delta}$, for $\mathbf{x} = (\mathcal{A}_1(-1), \ldots, \mathcal{A}_k(-1))$ and $\mathbf{y} = (\mathcal{A}_1(1), \ldots, \mathcal{A}_k(1))$ we have
		\begin{align*}
			\prob^{\mathbf{x}, \mathbf{y}}_{NI} (C_{k,a}^{\epsilon}) \mathbbm{1}_{F_k^{\sigma, \delta}} & \geq \left( p_2(\mathbf{x}, \mathbf{y}) \right)^{-1} \int_{\mathbf{u} \in U} p_{1 - \epsilon} (\mathbf{x}, \mathbf{u}) \int_{\mathbf{v} \in V} p_{2 \epsilon} (\mathbf{u}, \mathbf{u} + \mathbf{v}) p_{1 - \epsilon} (\mathbf{u} + \mathbf{v}  , \mathbf{y}) \, d \mathbf{v} \, d \mathbf{u}. 
		\end{align*}
		On the event $F_k^{\sigma, \delta}$, it is clear that $(p_2(\mathbf{x}, \mathbf{y}))^{-1}$ is non-vanishing as $\epsilon \to 0$ for all $\mathbf{u} \in U$, $\mathbf{v} \in V$. Furthermore, by Brownian scaling, it is also clear that for all $\mathbf{u} \in U$ that the second integral term $\int_{v \in V} p_{2 \epsilon} (\mathbf{u} , \mathbf{u} + \mathbf{v}) \, d \mathbf{v}$ is also non-vanishing as $\epsilon \to 0$. Next, the following function extends continuously to the boundary of the chamber and is uniformly bounded below by a positive constant on any bounded set,
		\begin{align*}
			q_t(\mathbf{w}, \mathbf{z}) &= \frac{p_t(\mathbf{w}, \mathbf{z})}{\Delta(\mathbf{w})  \Delta(\mathbf{z} )},
		\end{align*}
		where $\Delta(\mathbf{w}) = \prod_{i < j} (w_i - w_j)$ is the Vandermonde product. Hence, there exists a constant $\tilde{c} > 0$ independent of $\epsilon$ so that for all $\epsilon $ small enough and all $\mathbf{u} \in U$, $\mathbf{v} \in V$
		\begin{align*}
			p_{1 - \epsilon} (\mathbf{x}, \mathbf{u}) p_{1 - \epsilon} (\mathbf{u} + \mathbf{v}, \mathbf{y}) & \geq \tilde{c} \Delta (\mathbf{u}) \Delta(\mathbf{u} + \mathbf{v}) 
			\geq c' \epsilon^{\frac{k^{2} - k}{4}} \cdot \epsilon^{\frac{k^2 - k}{4}}= c'\epsilon^{\frac{k^2 - k}{2}},
		\end{align*}
		Hence, 
		\begin{align*}
		 \prob^{\mathbf{x} , \mathbf{y}}_{NI} (C^{\epsilon}_{k,a}) \mathbbm{1}_{F^{\sigma, \delta}_k} \      \geq  c' \epsilon^{\frac{k^2 - k}{2}} \textrm{Vol} (U) \mathbbm{1}_{F_k^{\sigma, \delta}} \geq c \epsilon^{\frac{k^2 - k}{2}} \epsilon^{\frac{k - 1}{2}} \mathbbm{1}_{F_k^{\sigma, \delta}} = c \epsilon^{\frac{k^2 - 1}{2}} \mathbbm{1}_{F_k^{\sigma, \delta}} .
		\end{align*}
		Furthermore, on the event $F_k^{\sigma, \delta}$, there exists a constant $d > 0$ so that
		\begin{align*}
			\mathcal{P}(C^{\epsilon}_{k, a}) = \frac{\prob^{\mathbf{x} , \mathbf{y}}_{NI} (H \cap C^{\epsilon}_{k,a})}{\prob^{\mathbf{x} , \mathbf{y}}_{NI} (H)} \geq \prob^{\mathbf{x} , \mathbf{y}}_{NI} (C^{\epsilon}_{k,a}) \prob^{\mathbf{x} , \mathbf{y}}_{NI} (H \, | \, C^{\epsilon}_{k, a}) \geq d \cdot \prob^{\mathbf{x}, \mathbf{y}}_{NI} (C^{\epsilon}_{k, a}). 
		\end{align*}
		Taking an expectation $\E$, we therefore find that there is a constant $c > 0$ so that
		\begin{align*}
			\E \mathcal{P}  \left(T^a \cap B^{\epsilon}_k \right) \mathbbm{1}_{F_k^{\sigma, \delta} \cap G^b_k}  \geq c  \epsilon^{\frac{k^2 - 1}{2}} \prob \left(  F_k^{\sigma, \delta} \cap G^b_k\right),
		\end{align*}
		as desired.
	\end{proof}

	\appendix

	\section{A geometric approach to Theorem \ref{T:shadow-sheet-simple}} \label{sec:shadow_sheet_geomproof}
	
	The classification of the turnaround index in Theorem \ref{T:k_geom} indicates that the DtM Isometry for the Airy sheet is closely connected with the coalescing nature of the directed landscape. To further clarify this relationship, we prove the following theorem, which gives an alternative perspective on the DtM isometry in Theorem \ref{T:shadow-sheet-simple} and whose proof is based purely on path-switching arguments.
	
	\begin{thm}
		\label{T:alternate-approach}
		Let $\cS$ be the extended Airy sheet.
		For all $x, y \ge 0$ and $k \in \N$
		$$
		\cS(x, y) \ge \cS(\{0^{k-1}, x \}, 0^k) + \cS(0^k, \{0^{k-1}, y \}) -\cS (0^k, 0^k) - \cS(0^{k-1}, 0^{k-1}).
		$$
		Moreover, if there exists some $k \in \N$ such that the optimizer from $(0^k, 0)$ to $(0^k, 1)$ intersects the rightmost geodesic from $(x, 0)$ to $(y, 1)$, then the above inequality is an equality for the smallest such $k$.
	\end{thm}
	
	To see how Theorem \ref{T:alternate-approach} is connected to Theorem \ref{T:shadow-sheet-simple}, observe first that the `if' condition in the `Moreover' statement follows from the divergence 
	\begin{equation}
		\label{E:diverge}
		\lim_{\ell \to \infty} \max_{t \in [0, 1]} \tau^\ell_\ell(t) = \infty,
	\end{equation}
	where $\tau^\ell$ is the unique optimizer in $\cL$ from $(0^\ell, 0)$ to $(0^\ell, 1)$. Equation \eqref{E:diverge} can be proven using soft arguments, as in \cite{basu_interlacing}, though we do not do so here. Alternately, a line ensemble proof is possible along the lines of Theorem \ref{T:melon_width}. Given \eqref{E:diverge}, Theorem \ref{T:alternate-approach} implies that 
	$$
	\cS(x, y) = \max_{k \in \N} \cS(\{0^{k-1}, x \}, 0^k) + \cS(0^k, \{0^{k-1}, y \}) -\cS (0^k, 0^k) - \cS(0^{k-1}, 0^{k-1}).
	$$
	The DtM isometry follows from this statement, together with the following lemma.

	\begin{lemma}
		\label{L:A0ky1}
		Under the coupling of the Airy Sheet and Airy line ensemble in Theorem \ref{T:shadow-sheet}, for $k \in \N$, $x, y \in \R_{+}$
		\begin{align}
			\A[(0, k) \to (y, 1)] &= \Sheet(0^k ,  \{ 0^{k - 1} , y\}) - \Sheet(0^k , 0^k) \label{E:AS_coupling1} \\
			\tilde{\cA}[(0, k) \to (x, 1)] &=  \Sheet(\{ 0^{k - 1}, x \}, 0^k) - \Sheet(0^k , 0^k)  \label{E:AS_coupling2}
		\end{align}
		where $\tilde{\mathcal{A}}_i(x) = \mathcal{A}_i(-x)$. 
	\end{lemma}
	
	\begin{proof}
		We only prove the first identity, as the second follows identically. Observe that if we set $\bx = 0^k$ in \eqref{E:shadow-sheet-pre} then $\tilde A^n[(0, I) \to (0^k, 1)] + \sum_{j \in I} \tilde A^n_j(0) \le \sum_{j=1}^k \tilde A^n_j(0)$ with equality if and only if $I = (1, \dots, k)$. Therefore
		$$
		\cS_n(0^k, \by) - \cS_n(0^k, 0^k) = A^n[(0, (1, \dots, k)) \to (\mathbf{y}, 1)].
		$$
		Setting $\by = (0^{k-1}, y)$, the right-hand side above equals $A^n[(0, k) \to (y, 1)]$, and \eqref{E:AS_coupling1} follows by taking $n \to \infty$.
	\end{proof}
	
	We expect that Lemma \ref{L:A0ky1} also has a purely geometric proof, but have not pursued this here. For example, in the $k = 2$ case, \eqref{E:AS_coupling1} reduces to the statement that
	$$
	\cS(0, 0) + \cS(0, y) - \cS(0^2, \{0, y\}) = \min_{x \in [0, y]} 2\cS(0,x) - \cS(0^2, x^2).
	$$
	Both sides above represent disjointness costs: the left-hand side is the penalty for forcing two paths from $(0,0)$ to $(0,1)$ and $(y, 1)$ to be disjoint, and the right-hand side is the minimal penalty over $x \in  [0, y]$ for forcing two paths from $(0,0)$ to $(x, 1)$ to be disjoint. Note that local minima of the gap process $2\cS(0,x) - \cS(0^2, x^2)$ are in one-to-one correspondence with non-uniqueness points for geodesics from $(0,0)$ to $\R \times \{1\}$ \cite[Theorem 1.1]{dauvergne2026gap}, which already gives some geometric interpretation to the right-hand side above.
	
	\begin{proof}[Proof of Theorem \ref{T:alternate-approach}]
		By continuity it is enough to prove the result for $x, y \in \Q_+$. We furthermore work in the almost sure event where for every $x, y \in \Q_{+}$, there exists a unique directed geodesic between $(x, 0)$ and $(y, 1)$, and for each $\ell \in \N$, there is a unique disjoint optimizer from $(0^{\ell}, 0) \to (0^{\ell}, 1)$ where $0^{\ell} = (0, 0 , \ldots, 0)$ is an $\ell$-tuple.
		
		We will first use a path-crossing argument to show that for all $k \in \N$,
		\begin{align}
			\Sheet(\{ 0^{k - 1}, x\} , 0^k) + \Sheet(0^k , \{ 0^{k - 1}, y\}) \leq \Sheet(0^{k - 1}, 0^{k - 1}) + \Sheet(0^k , 0^k) + \Sheet(x, y) \label{E:airysheet_ineq1}
		\end{align}
		Suppose $\pi \in Q[\{ 0^{k - 1}, x\} \to 0^k]$ and $\pi' \in Q[0^k \to \{ 0^{k - 1}, y\}]$ are rightmost optimizers. For every $t \in [0, 1]$, let $\sigma_{2k}(t) \ge \cdots \ge \sigma_1(t)$ be the order statistics of the set $\{\pi_i(t), \pi_i'(t) : i = 1, \dots, k\}$. We leave it as an exercise to check that
		$$
		\sum_{i=1}^{2k} \|\sigma_i\|_\cL = \sum_{i=1}^{k} \|\pi_i\|_\cL + \sum_{i=1}^{k} \|\pi_i'\|_\cL = \Sheet(\{ 0^{k - 1}, x\} , 0^k) + \Sheet(0^k , \{ 0^{k - 1}, y\}). 
		$$

		Now define $\tau' = (\sigma_1, \sigma_3, \dots, \sigma_{2k-1})$, let $\tau = (\sigma_2, \sigma_4, \dots, \sigma_{2k-2})$, and set $\rho = \sigma_{2k}$. Noting that $\rho$ is a path from $(x, 0)$ to $(y, 1)$, to complete the proof of \eqref{E:airysheet_ineq1}, we just need to show that $\tau, \tau'$ are disjoint $(k-1)$- and $k$-tuples with all start and end points at $0$. This is immediate from the disjointness of the paths $\pi, \pi'$, which implies that for any indices $i, i + 1, i + 2$ with $i \in \II{1, 2k-2}$ and any $t \in (0, 1)$, one of the inequalities
		$$
		\sigma_i(t) \le \sigma_{i+1}(t) \le \sigma_{i+2}(t)
		$$
		must be strict.

        \begin{figure}[htbp]
		\centering
		\includegraphics[width=0.8\textwidth]{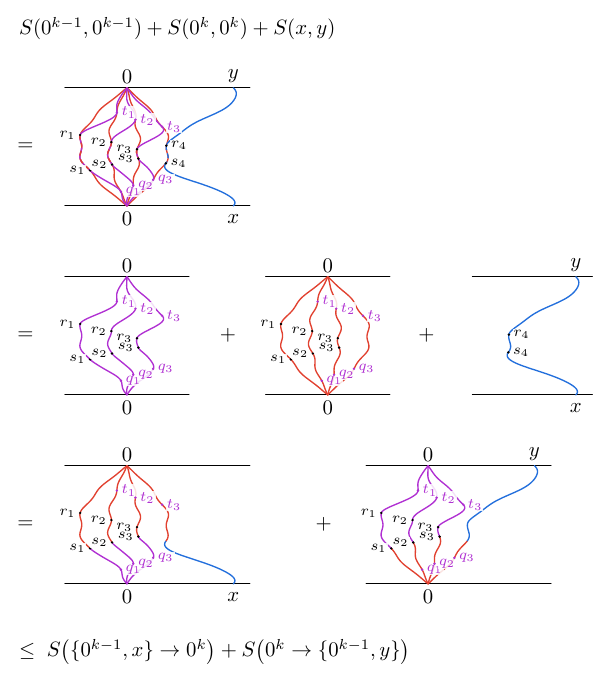} 
		\caption{A sketch of the path crossing argument used in the 'Moreover' statement.} 
		\label{fig:Geometric_Melon_pathswitch}
	\end{figure} 
		
		We move on to the `Moreover' statement. Let $k$ be the smallest integer such that the optimizer $\tau'$ from $(0^k,0)$ to $(0^k,1)$ intersects the rightmost geodesic $\rho$ from $(x, 0)$ to $(y, 1)$.  First, in the case when $k = 1$, the desired equality is simply
		$$
		\cS(x, y) + \cS(0, 0) = \cS(x, 0) + \cS(0, y).
		$$
		To prove this, letting $t$ be any time when $\tau'(t) = \rho(t)$, the concatenations $\tau|_{[0, t]} \oplus \rho|_{[t, 1]}$ and $\rho|_{[0, t]} \oplus \tau|_{[t, 1]}$ give paths from $(0,0)$ to $(y, 1)$ and from $(x, 0)$ to $(0, 1)$, respectively. The inequality LHS$\le$RHS above follows. 
		
		The case when $k \ge 2$ requires more delicacy.
		Let $\tau$ be the optimizer from $(0^{k-1},0)$ to $(0^{k-1},1)$. By our choice of $k$, $\tau$ and $\rho$ are disjoint, whereas $\tau'$ and $\rho$ intersect.
		
		To show that \eqref{E:airysheet_ineq1} is an equality, we will construct disjoint $k$-tuples $\pi$ from $(\{0^{k-1}, x\}, 0)$ to $(0^k, 1)$ and $\pi'$ from $(0^k, 0)$ to $(\{0^{k-1}, y\}, 1)$ such that for all $t \in [0, 1]$, the $2k$-tuples
		\begin{equation}
			\label{E:pipipi}
			(\pi_i(t), \pi_i'(t) : i = 1, \dots, k), \qquad (\tau_1(t), \dots \tau_{k-1}(t), \tau_1'(t), \dots, \tau_k'(t), \rho(t))
		\end{equation}
		are permutations of each other.

		First, note that $\tau$ and $\tau'$ are interlaced, so that  $\tau_{i}' \leq \tau_{i} \leq \tau_{i+1}'$ for $i = 1, \ldots, k - 1$. This can be proven with path-crossing arguments, or alternately follows from the same fact in the prelimit, see \cite[Lemma 2.3]{dauvergne2021disjoint} or \cite[Proposition 2.5]{basu_interlacing}.
		Hence, $\tau_{k-1}'$ is disjoint from $\rho$. Next, we claim that the set of $t \in [0, 1]$ where $\tau_k'(t) = \rho(t)$ is an interval. Indeed, let 
		$$
		s_k = \inf \{t \in [0, 1] : \tau_k'(t) = \rho(t)\}, \qquad r_k = \sup \{t \in [0, 1] : \tau_k'(t) = \rho(t)\}. 
		$$
		The path $\rho|_{[s_k, r_k]}$ is a geodesic from $(\rho(s_k), s_k)$ to $(\rho(r_k), r_k)$ which is disjoint from $\tau_{k-1}'$. Therefore $\tau'_k|_{[s_k, r_k]} = \rho|_{[s_k, r_k]}$, since otherwise we could swap this segment out for $\rho|_{[s_k, r_k]}$ and create a new disjoint $k$-tuple from $(0^k, 0)$ to $(0^k, 1)$ of equal or greater weight, contradicting the optimality/uniqueness of $\tau'$.
		
		Now, let $t_k$ be the first time $\tau_{k - 1}$ rejoins $\tau_k'$ after $r_k$, and let $q_k$ be the last time $\tau_{k-1}$ and $\tau_k'$ meet before time $s_k$. That is:
		\begin{align*}
			t_k &= \inf \{ t \in [r_k, 1] : \tau_{k-1}(t) = \tau_{k}'(t)\}, \\
			q_k &= \sup \{ q \in [0, s_k] : \tau_{k-1}(q) = \tau_k'(q) \}.
		\end{align*}
		Note that $q_k < s_k < r_k < t_k$. Now, we claim that $\tau_{k-1}$ and $\tau_{k-1}'$ must meet in the interval $[q_k, t_k]$. Indeed, if not then we could switch the two path segments $\tau_k'|_{[q_k, t_k]}$ and $\tau_{k-1}|_{[q_k, t_k]}$ in either $\tau$ or $\tau'$ and contradict either uniqueness or optimality of one of these $k$-tuples.
		Define
		\begin{align*}
			s_{k-1} &= \inf \{ s \in [q_k, t_k] : \tau_{k-1}'(s) = \tau_{k-1}(s)\}.\\
			r_{k-1} &= \sup \{ r \in [q_k, t_k] : \tau_{k-1}'(r) = \tau_{k-1}(r)\}.		
		\end{align*}
		Similarly to before, on the interval $[s_{k - 1}, r_{k-1}]$, one must necessarily have that $\tau_{k-1}' = \tau_{k - 1}$, since otherwise we could switch the path segments $\tau_{k-1}|_{[s_{k - 1}, r_{k-1}]}$ and $\tau_{k-1}'|_{[s_{k - 1}, r_{k-1}]}$ in either $\tau$ or $\tau'$.

		We may now iterate this process, defining for $i = k - 1, k - 2, \ldots , 2$:
		\begin{align*}
			t_i &= \inf \{ t \in [r_i, 1] : \tau_{i-1}(t) = \tau_i'(t)\} \\
			q_i &= \sup \{ q \in [0, s_i] : \tau_{i-1}(q) = \tau_i'(q) \} \\
			r_{i-1} &= \sup \{ r \in [q_i, t_i] : \tau_{i-1}'(r) = \tau_{i-1}(r)\} \\
			s_{i-1} &= \inf \{ s \in [q_i, t_i] : \tau_{i-1}'(s) = \tau_{i-1}(s)\}
		\end{align*}
		We now define $\pi \in Q[\{ 0^{k - 1}, x\} \to 0^k]$ and $\pi' \in Q[0^{k} \to \{ 0^{k - 1}, y\}]$ as follows. Define for $i = 1, \ldots, k$
		\begin{align*}
			\pi_{i}(t) &= 
			\begin{cases}
				\tau_i'(t) & t \geq s_i \\ 
				\rho(t) & t \leq s_i \textrm{ and } i = k \\
				\tau_{i}(t) & t \leq s_i  \textrm{ and } i \neq k
			\end{cases} \\
			\pi_i'(t) &= \begin{cases}
				\tau_i'(t) & t \leq r_i \\
				\rho(t) & t \geq r_i \textrm{ and } i = k \\
				\tau_{i}(t) & t \geq r_i \textrm{ and } i \neq k
			\end{cases} 
		\end{align*}
		Since $\tau_i = \tau_i'$ on the intervals $[s_i, r_i]$ for $i < k$, and $\tau_k' = \rho$ on $[s_k, r_k]$, the permutation property \eqref{E:pipipi} holds. It is also straightforward to check that the start and endpoints of $\pi, \pi'$ are as desired. It remains to show disjointness of the paths in $\pi$ and $\pi'$. We show that the paths in $\pi'$ are disjoint, as the proof for $\pi$ is symmetric.
		
		We first prove that $\pi_k'$ is disjoint from $\pi_{k - 1}'$. Since both $\tau_{k-1}$ and $\tau_{k-1}'$ are disjoint from $\rho$, the only possible way for $\pi_k'$ and $\pi_{k- 1}'$ to coalesce is if $\pi_{k - 1}'$ leaves $\tau_{k-1}'$ and coalesces into $\pi_{k}'$ before $\pi_{k}'$ switches to $\rho$. In other words, there would exist a point $t \in (r_{k-1}, s_k)$ with $\tau_{k - 1}(t) = \tau_k'(t)$. However, the construction of $t_k, q_k$ implies that $\tau_{k-1}$ and $\tau_k'$ are disjoint on $(q_k, t_k)$. Since $r_{k-1} \ge q_k$, no such $t$ can exist.

		Now we show that $\pi_i'$ and $\pi_{i-1}'$ are disjoint for $2 \le i \le k-1$. Again, $\tau_{i-1}'(r) < \tau_i(r)$ for all $r \in (0, 1)$ by interlacing. Therefore $\pi_i'$ and $\pi_{i-1}'$ can only coalesce if $\pi_{i-1}'$ leaves $\tau_{i-1}'$ and coalesces with $\pi_{i}'$ before $\pi_i'$ joins $\tau_i$. In other words, there would exist a point $t \in (r_{i-1}, s_i)$ with $\tau_{i - 1}(t) = \tau_i'(t)$. Again, the construction of $t_i, q_i$ implies that $\tau_{i-1}$ and $\tau_i'$ are disjoint on $(q_i, t_i)$. Since $r_{i-1} \ge q_i$, no such $t$ can exist. This completes the proof.
	\end{proof}
	
	\printbibliography
	
\end{document}